\documentclass[9pt]{amsart}
\usepackage{amsmath, amssymb}
\usepackage{amsfonts,amscd}
\usepackage[a4paper, margin=1.25in]{geometry}
\usepackage[parfill]{parskip}
\usepackage{xcolor}
\usepackage{comment}
\usepackage{mathtools}
\usepackage{commath}

\allowdisplaybreaks

\newtheorem{theorem}{Theorem}

\newtheorem{prop}[theorem]{Proposition}
\renewenvironment{proof}{\par\noindent{\bf Proof.}}{$\square$\par\bigskip}
\newtheorem{lemma}[theorem]{Lemma}
\newtheorem{remark}[theorem]{Remark}
\newtheorem{cor}[theorem]{Corollary}

\newtheorem{question}[theorem]{Question}
\newtheorem{notation}[theorem]{Notation}

\newtheorem{thm}{Theorem}

\newtheorem*{unmthm}{Theorem}

\def\Z{\mathbb Z}
\def\N{\mathbb N}
\def\Q{\mathbb Q}
\def\R{\mathbb R}
\def\C{\mathbb C}

\def\H{\mathbb H}

\def\Sym{\operatorname{Sym}}

\def\O{\operatorname{O}}
\def\o{\operatorname{o}}
\def\log{\operatorname{log}}

\def\thtp{\operatorname{\theta_{\textit{f}}\,(\textit{p})}}

\def\thtq{\operatorname{\theta_{\textit{f}}\,(\textit{q})}}

\def\Lf{\operatorname{\mathcal L_{\textit{f}}}}

\allowdisplaybreaks

\begin{document}

\title{Higher moments of the pair correlation function for Sato-Tate sequences}
\author[Jewel Mahajan]{Jewel Mahajan}
\date{\today}
\address{Jewel Mahajan, IISER Pune, Dr Homi Bhabha Road, Pashan, Pune - 411008, Maharashtra, India}
\email{jewel.mahajan@students.iiserpune.ac.in}
\author[Kaneenika Sinha]{Kaneenika Sinha}
\address{Kaneenika Sinha, IISER Pune, Dr Homi Bhabha Road, Pashan, Pune - 411008, Maharashtra, India}
\email{kaneenika@iiserpune.ac.in}

\begin{abstract} In \cite{BS}, Balasubramanyam and the second named author derived the first moment of the pair correlation function for Hecke angles lying in small subintervals of $[0,1]$ upon averaging over large families of Hecke newforms of weight $k$ with respect to $\Gamma_0(N)$.  The goal of this article is to study higher moments of this pair correlation function.  For an integer $r \geq 2$, we present bounds for its $r$-th power moments.  We apply these bounds to record lower order error terms in the computation of the second and third moments.  As a result, one can obtain the convergence of the second and third moments of this pair correlation function for suitably small intervals, and under appropriate growth conditions for the size of the  families of Hecke newforms.  
\end{abstract}
\bigskip

\maketitle

\section{Introduction}\label{introduction}
Let $k$ and $N$ be positive integers with $k$ even.  Let  $S(N,k)$ denote the space of modular cusp forms of weight $k$ with respect to $\Gamma_0(N).$  For $n \geq 1,$ let $T_n$ denote the $n$-th Hecke operator acting on $S(N,k).$  We denote the set of Hecke newforms  in $S(N,k)$  by $\mathcal F_{N,k}$.    Any $f(z) \in \mathcal F_{N,k}$ has a Fourier expansion
$$f(z) = \sum_{n=1}^{\infty} {n^{\frac{k-1}{2}}}a_f(n) q^n, \qquad q = e^{2\pi i z},$$
where 
$a_f(1) = 1$ and
$$\frac{T_n(f(z))}{n^{\frac{k-1}{2}}} = a_f(n) f(z),\,n \geq 1.$$
We denote $s(N,k) := |\mathcal F_{N,k}|$ and note that $s(N,k)$ is the dimension of the subspace of primitive cusp forms in $S(N,k)$.

Let $p$ be a prime number with $(p,N) = 1.$  By a theorem of Deligne, the eigenvalues $a_f(p)$ lie in the interval $[-2,2].$  We denote $a_f(p) = 2\cos \pi \thtp,\,$ with $ \thtp \in [0,1]$.

The Sato-Tate conjecture, now a theorem \cite{BGHT}, is the assertion that if $f$ is a non-CM  newform in $\mathcal F_{N,k},$ then the Sato-Tate sequence 
\begin{equation}\label{ST-sequence}
\{\thtp : \,p \text{ prime},\,(p,N) = 1\} \subseteq [0,1]
\end{equation}
is equidistributed in the interval $[0,1]$ with respect to the measure $\mu(t) dt$, where $\mu(t) = 2\sin^2(\pi t)$.
That is, for an interval $I = [a,b]$ such that $0\leq a < b \leq 1$,
\begin{equation}\label{ST-asymp}
\lim_{x \to \infty}\frac{1}{\pi_N(x)}\#\{p \leq x :\, (p,N) = 1,\,\thtp \in [a,b]\} = \int_I \mu(t) dt,
\end{equation}
where $\pi_N(x)$ denotes the number of primes $p \leq x$ such that $(p,N)=1$.  Henceforth, for an interval $I$, $\mu(I)$ denotes the measure $\int_I \mu(t) dt$. 

We ``straighten out" the Sato-Tate sequence into a uniformly distributed sequence by defining
$$H(\thtp) := \int_0^{ \thtp} \mu(t) dt.$$
As an immediate consequence of \eqref{ST-asymp}, we see that the sequence $\{H(\thtp):\,(p,N) = 1\}$ is uniformly distributed in the interval $[0,1]$.  

A study of the moments of the pair correlation function for the sequence $\{H(\thtp):\,(p,N) = 1\}$ as one varies $f$ over appropriate families $\mathcal F_{N,k}$ was initiated in \cite{BS}.  This study is primarily motivated by a question of Katz and Sarnak that compares the spacings between straightened Hecke angles to spacings between points arising from independent and uniformly distributed random variables in the unit interval.  One way to address these questions is via the pair correlation function, which looks at the spacings between unordered elements of a uniformly distributed sequence.  In this context, the question of Katz and Sarnak  can be stated as follows:
\begin{question}[Katz, Sarnak \cite{KS}]

For any $s>0$, the pair correlation function of the sequence $\{H(\thtp) : \,p \text{ prime},\,(p,N) = 1\}$ is defined as:
$$R(x,s)(f) := \frac{1}{\pi_N(x)}\#\left\{(p,q):\,p \neq q \leq x,\,\begin{array}{c} (p,N) = (q,N) = 1,\\
H(\thtp) - H(\thtq) \in \left[\frac{-s}{\pi_N(x)},\frac{s}{\pi_N(x)}\right] + \Z \end{array}\right\}.$$
For any $s >0$, does the limit $\lim_{x \to \infty}  R(x,s)(f)$ exist and is it equal to $2s$?

If the answer is yes, we say that the sequence $\{H(\thtp)\}$ has Poissonnian pair correlation.
\end{question}
A variation of the question above was addressed in \cite{BS} by restricting $\thtp$ to short intervals $I$, such that $|I| \to 0$ as $x \to \infty$.  
\begin{question}\label{question-shrinking}
Let $0 < \psi < 1$ and $I_{\delta}$ denote intervals of the form 
$$\left[\psi - \delta,\psi + \delta \right],\, \delta = \delta(x) \to 0 \text{ as }x \to \infty.$$
Suppose 
\begin{equation}\label{ST-asymp-shrinking}
 \#\left\{p \leq x:\,(p,N) =1,\,\thtp \in I_{\delta} \right\} \sim \pi_N(x) \mu(I_{\delta})\text{ as }x \to \infty.
 \end{equation}
We define
\begin{equation}\label{R-delta}
\widetilde{R}_{\delta}(x,s)(f) := \frac{1}{\pi_N(x) \mu(I_{\delta})}\#\left\{(p,q):\,p \neq q \leq x,\,  \begin{array}{c} (p,N) = (q,N) = 1,\,\thtp, \thtq) \in I_{\delta},\\H(\thtp) - H(\thtq) \in \left[\frac{-s}{\pi_N(x)},\frac{s}{\pi_N(x)}\right] \end{array}\right\}.
\end{equation}
Does the limit $\lim_{x \to \infty}  \widetilde{R}_{\delta}(x,s)(f)$ exist and is it equal to $2s$?
\end{question}
To answer the above question meaningfully, we need conditions on $\delta(x)$ for which \eqref{ST-asymp-shrinking} holds.  The existence and distribution of Hecke angles in shrinking intervals $I$ with $|I| \to 0$ as $x \to \infty$ is inextricably linked to effective error terms in the Sato-Tate equidistribution theorem (we explain this in detail in Section \ref{sec:small-scales}).  These error terms have been addressed in \cite{Murty}, {\cite{RT}, \cite{Thorner} and \cite{HIJS}.  In this context, an unconditional theorem of Thorner leads to the following result:
\begin{unmthm}[Thorner, \cite{Thorner}] 
Let $F(x)$ be a monotonically increasing function with $\lim_{x \to \infty} F(x) = \infty$.  Then, for any interval $ I \subset [0,1]$ of length $$\mu(I) \geq \frac{\log(kN \log x) F(x)}{\sqrt{\log x}},$$
we have,
$$\lim_{x \to \infty}\frac{1}{\pi_N(x)}\#\{p \leq x :\, (p,N) = 1,\,\thtp \in I\} = \mu(I).$$

In particular, if $\delta(x) \to 0$ is chosen such that  
$$\mu(I_{\delta}) \geq \frac{\log(kN \log x) F(x)}{\sqrt{\log x}},$$
then
$$\#\{p \leq x :\, (p,N) = 1,\,\thtp \in I_{\delta}\} \sim \pi_N(x) \mu(I_{\delta}) \text{ as }x \to \infty.$$
\end{unmthm}
One simplifies Question \ref{question-shrinking} as follows:  for $0 < \psi < 1$, henceforth, we denote $A := 2\sin^2\pi\psi$.  Let us consider intervals  
$$\mathcal I_L (\psi):= \left[\psi - \frac{1}{AL},\psi + \frac{1}{AL}\right],$$
such that $L = L(x) \to \infty$ as $x \to \infty$, and \eqref{ST-asymp-shrinking} holds for $\delta = 1/AL$.  Then, as $x \to \infty$,
\begin{equation}\label{ST-asymp-local}
\Lf\,(\psi) := \#\left\{p \leq x :\, (p,N) = 1,\,\thtp \in \mathcal I_L(\psi)\right\} \sim \pi_N(x)\mu_{\frac{1}{AL}}.
\end{equation} 
The advantage of localizing our intervals around $\psi$ is that the Sato-Tate density $2\sin^2\pi t \sim A$ is essentially constant in short intervals and the straightening of the Hecke angles is more or less equivalent to rescaling them.  That is,
\begin{equation}\label{ST-asymp-local-1}
 \Lf \,(\psi)\sim \pi_N(x) \int_{\psi - \frac{1}{AL}}^{\psi +  \frac{1}{AL}} 2\sin^2 \pi t \,dt \sim A \frac{2}{AL}\pi_N(x) = \frac{2\pi_N(x)}{L}.
 \end{equation}
 If $\thtp,\,\thtp \in \mathcal I_L(\psi)$, then
\begin{equation}\label{ST-asymp-local-2}
H(\thtp) - H(\thtq) = \int_{\thtq}^{\thtp} 2 \sin^2 \pi t\, dt  \sim A(\thtp - \thtq)\text{ as }x \to \infty.
\end{equation}
Thus,
$$\widetilde{R}_{\frac{1}{AL}}(x,s)(f) = \frac{1}{\Lf\,(\psi)}\#\left\{(p,q):\,p \neq q \leq x,\,  \begin{array}{c} (p,N) = (q,N) = 1,\,\thtp), \thtq \in \mathcal I_L(\psi),\\H(\thtp) - H(\thtq) \in \left[\frac{-s}{\pi_N(x)},\frac{s}{\pi_N(x)}\right]  \end{array}\right\}$$
$$ \sim \frac{1}{\Lf\,(\psi)}\#\left\{(p,q):\,p \neq q \leq x,\,  \begin{array}{c} (p,N) = (q,N) = 1,\,\thtp,\,\thtq \in \mathcal I_L(\psi),\\ \thtp - \thtq \in \widetilde{I}_x  \end{array}\right\},$$
where
$$ \widetilde{I}_x =  \left[\frac{-s}{A\pi_N(x)},\frac{s}{A\pi_N(x)}\right].$$
The pair correlation function of a sequence is obtained by evaluating some exponential sums related to the sequence.  In the case of Hecke angles, we have to remove the imaginary parts of these sums in order to apply existing techniques. Therefore, we modify the above question and consider the families 
\begin{equation*}
\mathcal A_{f,x} := \{\pm \thtp \text{ mod 1}:\, p \leq x,\,(p,N) = 1\}.
\end{equation*}
As explained in Section \ref{section-pair-cor}, the pair correlation function of the families $\mathcal A_{f,x}$ turns out to be asymptotic to 
$${R}_{1/L}(x,s)(f) :=
\frac{L }{8 \pi_N(x)} \sum_{(p,q) \atop {p,q \leq x \atop{(p,N) = (q,N) = 1 \atop {p \neq q}}}}\sum_{n \in \Z} \chi_{\left[-\frac{1}{A},\frac{1}{A}\right]}(L (\pm \thtp - \psi + n)) 
\sum_{n \in \Z}\chi_{\left[-\frac{1}{A},\frac{1}{A}\right]}(L (\pm \thtq - \psi + n))$$
$$\times \sum_{n \in \Z} \chi_{\left[-\frac{s}{2A},\frac{s}{2A}\right]}(\pi_N(x) (\pm \thtp \pm \thtq + n)).$$
While the function ${R}_{1/L}(x,s)(f)$ is difficult to study (we explain this in Sections \ref{sec:small-scales} and \ref{remarks-pcf}), one way to address its convergence can be through the method of moments.  That is, one may study the moments
$$\frac{1}{|\mathcal F_{N,k}|} \sum_{f \in \mathcal F_{N,k}} \left(R_{1/L}(x,s)(f)\right)^r,\,r \in \N.$$
The perspective of averaging quantities related to $f$ over all Hecke newforms (or eigenforms) goes back to the work of \cite{Sarnak}, \cite{CDF} and \cite{Serre}.  In order to approach difficult arithmetic questions pertaining to a Hecke newform $f$ (such as the distribution or spacing properties of Hecke angles $\thtp$), one can ask what happens to those questions ``on average" over families of eigenforms.  Summing over all Hecke newforms (or eigenforms as the case may be) allows us to bring in techniques such as the Eichler-Selberg trace formula for the trace of Hecke operators acting on subspaces of cusp forms of weight $k$ with respect to $\Gamma_0(N)$.  For example, Conrey, Duke and Farmer \cite{CDF} used the trace formula to prove that the Sato-Tate conjecture holds on average over large families.  That is, if $k > e^x$, they showed that
$$\lim_{x \to \infty} \frac{1}{|\mathcal F_{1,k}|}\sum_{f \in \mathcal F_{1,k}} \left( \frac{1}{\pi(x)}\#\{p \leq x :\,\thtp \in [a,b]\}\right) = \int_I 2\sin^2\pi t \,dt,$$
In \cite{Nagoshi}, it is shown that the above asymptotic holds when $\frac{\log k}{\log x} \to \infty \text{ as }x \to \infty.$

In \cite{BS}, this approach of averaging is adopted in the investigation of the pair correlation function for the Hecke angles.  Since we also let the levels $N$ vary, the growth conditions take into account the contribution coming from them.  Moreover, it is feasible to consider a smooth variant of $R_{1/L}(x,s)(f)$.  This leads to the following theorem:
\begin{thm}[Balasubramanyam, Sinha, \cite{BS}]\label{pc-error}
Let us consider families $\mathcal F_{N,k}$ with levels $N = N(x)$ and even weights $k = k(x).$   Let  $g,\,\rho$ be real valued, even functions $\in C^{\infty}(\R)$ in the Schwartz class with Fourier transforms supported in the interval $[-1,1]$.
Let $0 < \psi < 1,\,\psi \neq 1/2$.  Define $A := 2\sin^2\pi\psi$.  Define
\begin{equation}\label{rho}
\rho_L(\theta) := \sum_{n \in \Z} \rho (L (\theta + n)) \text{ for }L = L(x) \geq 1,
\end{equation}
\begin{equation}\label{G}
G_x(\theta) := \sum_{n \in \Z} g\left(\pi_N(x)(\theta + n)\right),
\end{equation}
and the smoothened pair correlation function, 
$$R_2(g,\rho)(f) := \frac{L}{8 \pi_N(x)}\sum_{p,q \leq x \atop{(p,N) = (q,N) = 1 \atop {p \neq q}}}\rho_L( \pm \thtp - \psi)\rho_L(\pm \thtq - \psi) G_x( \pm \thtp \pm \thtq).$$
[Note: A detailed discussion of the above definitions is presented in Section \ref{section-pair-cor}.]
\begin{itemize}
\item [{\bf(a)}] We have
\begin{equation}\label{all-prop}
\begin{split}
&\frac{1}{|\mathcal F_{N,k}|} \sum_{f \in \mathcal F_{N,k}}  R_2(g,\rho)(f) = \frac{T(g,\rho)}{4 L}\\
&+ \O\left(\frac{1}{L}\right) + \O\left(\frac{L(\log \log x)^2}{\pi_N(x)}\right) + \O\left(\frac{x^{\pi_N(x) c} 8^{\nu(N)}}{kN}\right),
\end{split}
\end{equation}
where 
\begin{equation}\label{Tg-sum}
T(g,\rho) = \sum_{l \geq 1}(U(l) - U(l-1))^2 \widehat{g}\left(\frac{l}{\pi_N(x)}\right),
\end{equation}
with
$$U(l) = \widehat{\rho}\left(\frac{l}{L}\right) (2\cos 2\pi l\psi) - \widehat{\rho}\left(\frac{l+1}{L}\right) (2\cos 2\pi (l+1)\psi).$$
In the error term, $c$ refers to an absolute positive constant.
\item [{\bf(b)}] If we choose $L$ such that $$ L = \o\left(\frac{\pi_N(x)}{(\log \log x)^2}\right),$$
and consider families $\mathcal F_{N,k}$ with levels $N = N(x)$ and even weights $k = k(x)$ such that $$ \frac{\log \left(k{N}/8^{\nu(N)}\right)}{x} \to \infty \text{ as }x \to \infty,$$ 
then,
$$\frac{1}{|\mathcal F_{N,k}|} \sum_{f \in \mathcal F_{N,k}}  R_2(g,\rho)(f) \sim \frac{T(g,\rho)}{4 L} \text{ as } x\to \infty.$$
Furthermore,
$$\frac{T(g,\rho)}{4 L} \sim  A^2\widehat{g}(0){\rho \ast \rho}(0) \text{ as } x\to \infty.$$
\end{itemize}
\end{thm}
We make a few remarks here.
\begin{enumerate}
\item The above theorem tells us that the first moment of the pair correlation function
$R_2(g,\rho)(f)$ upon averaging over all newforms $f \in \mathcal F_{N,k}$ is asymptotic to $A^2\widehat{g}(0){\rho \ast \rho}(0)$.  However, we require the size of the families $\mathcal F_{N,k}$ to grow rapidly for this asymptotic to hold.  This limitation comes from the estimation of a term in the Eichler-Selberg trace formula.  The elliptic term in the trace formula leads to estimates of the form
$$ \O\left(\frac{x^{D\pi_N(x)}8^{\nu(N)}} {k{N}}\right)$$
for a positive constant $D$ in the pair correlation sum.  The use of alternative trace formulas such as the Petersson trace formula leads to lower values of $D$ than those obtained by the Eichler-Selberg trace formula, but causes the same problem if we want the above error term to go to zero.  

\item  In \cite{BS}, the first moment of $R_2(g,\rho)(f)$ upon averaging over $f \in \mathcal F_{N,k}$ was obtained for positive integers $N$ and $k$ such that $N$ is prime and $k$ is even.  The techniques can be readily generalized to all levels $N$.  Accordingly, a modified version of the result of \cite{BS} has been stated above. 
\end{enumerate}

The above theorem about the first moment 
$$ \frac{1}{|\mathcal F_{N,k}|}\sum_{f \in \mathcal F_{N,k}}R_2(g,\rho)(f)$$
naturally leads us to questions about higher moments 
$$\frac{1}{|\mathcal F_{N,k}|}\sum_{f \in \mathcal F_{N,k}} \left(R_2(g,\rho)(f)\right)^r,\,r \geq 2.$$

The primary objective of this article is to address this question.  In this direction, we prove the following theorems.

First, we present a bound for the $r$-th moments of $R_2(g,\rho)(f)$. In order to state this bound, we first introduce the following notation.
\begin{notation}\label{notation-r-moment}
Let $r$and $j$ be positive integers such that $r \geq 2$ and $1 \leq j \leq r$.
We define a set of variables indexed as follows:
$$\mathfrak S = \begin{cases}
\{m_1,m_2,\dots,m_j,l_{j+1},l_{j+2},\dots, l_r, m_1',m_2',\dots,m_j',l_{j+1}',\dots,l_r'\}&\text{ if }j<r\\
\{m_1,m_2,\dots,m_r,m_1',m_2',\dots,m_r'\}&\text{ if }j = r.
\end{cases}
$$
Let $t$ be an integer such that $2 \leq t \leq 2r-1$.  We now consider a partition $\mathcal P$ of $\mathfrak S$ into $t$ non-empty subsets $\{\mathfrak S_u,\,1 \leq u \leq t\}$ with the following conditions:
\begin{itemize}
\item $|\mathfrak S_u| \leq r$ for each $1 \leq u \leq t$.
\item Denote $i_u = |\mathfrak S_u|$.  For each $u$, $\mathfrak S_u$ has $d_u$ elements among
$$\{m_1,m_2,\dots ,m_j,m'_1,m'_2,\dots,m'_j\}$$ and $i_u - d_u$ elements among 
$$\{l_{j+1},\dots,l_r,l'_{j+1},\dots,l'_r\}.$$
If $j = r$, then $i_u = d_u$.
\item For each $1 \leq i \leq j$, $m_i$ and $m_i'$ cannot both belong to $\mathfrak S_u$.  Thus, $0 \leq d_u \leq j$ for each $u$.
\item Similarly, for each $j+1 \leq i \leq r$, $l_i$ and $l_i'$ cannot both belong to $\mathfrak S_u$.
\end{itemize}
Let $\mathcal P(r,t,j)$ denote the set of all partitions of $\mathfrak S$ into $t$ subsets satisfying the above conditions.  For $1 \leq i \leq r,\,0 \leq d \leq j$ and for a partition $\mathcal P \in \mathcal P(r,t,j)$,
define
$$x_{i,d} := \#\{1 \leq u \leq t:\,i_u = i,\,d_u = d\}.$$
\end{notation}
\begin{thm}\label{higher-moments-bounds}
Let us consider families $\mathcal{F}_{N,k}$ with levels $N=N(x)$ and even weights $k=k(x)$. Let $g$, $\rho$ be real valued, even functions $\in C^{\infty}(\R)$ in the Schwartz class with Fourier transforms supported in $[-1,1]$ and $L=L(x)\to \infty$ as $ x \to \infty$.  Let $0 < \psi < 1$, $\psi \neq 1/2$, and let $A = 2\sin^2\pi\psi$. 
 Let $\rho_L,\,G_x,\,R_2(g,\rho)(f)$ and $T(g,\rho)$ be as defined in Theorem \ref{pc-error}, and let $L(x) < \frac{\pi_N(x)}{(\log\log x)^2}.$
For any integer $r \geq 3$, we have
\begin{equation*}
\begin{split}
&\frac{1}{|\mathcal F_{N,k}|} \sum_{f \in \mathcal F_{N,k}} \left(R_2 (g,\rho) (f)\right)^r -  \left(\frac{T(g,\rho)}{4 L}\right)^r \\
& \ll_r \sum_{t=2}^{2r} \frac{L^{3r - 2t}}{\pi_N(x)^{2r - t}}  +\frac{L^{1/2}\log \log x}{\pi_N(x)^{1/2}}\\
&\sum_{1 \leq j \leq r} \sum_{t = 2}^{2r-1} \sum_{\mathcal P \in \mathcal P(r,t,j)} \frac{L^{P(\mathcal P)}}{\pi_N(x)^{Q(\mathcal P)}}\left(\begin{cases} (\log \log x)^{t} &\text{ if }t \leq j\\
1 &\text{ if } t > j\\
\end{cases}\right),\\
& + \frac{x^{E(r)\pi_N(x)}8^{\nu(N)}}{kN},
\end{split}
\end{equation*}
where
 $$P(\mathcal P)=\begin{cases}
3r-2j-t-\sum_{i \geq 2} x_{i,0}+\sum_{i \geq 1} x_{i,i} &\text{ if }t \leq j\\
3r - j - 2t-\sum_{i \geq 2} x_{i,0}+\sum_{i \geq 1} x_{i,i} &\text{ if }t > j,\\
\end{cases}$$
$$Q(\mathcal P)=\begin{cases}
3r - j - t-\sum_{i \geq 1} x_{i,0} +\sum_{i \geq 2} x_{i,i} &\text{ if }t \leq j\\
2r - 2j - \sum_{i \geq 1} x_{i,0} +\sum_{i \geq 2} x_{i,i} &\text{ if }t > j,\\
\end{cases}
$$
and $E(r)$ is a quantity that depends only on $r$.
\end{thm}
The above theorem gives us a bound that addresses all the higher-power moments of $R_2(g,\rho)(f)$.  For the $r$-th moment
$$\frac{1}{|\mathcal F_{N,k}|} \sum_{f \in \mathcal F_{N,k}} \left(R_2 (g,\rho) (f)\right)^r$$
 to converge to $$\left(\frac{T(g,\rho)}{4 L}\right)^r,$$
 we need to make an appropriate choice of $L$ such that the ``error terms"
 \begin{equation*}
\begin{split}
& \sum_{t=2}^{2r} \frac{L^{3r - 2t}}{\pi_N(x)^{2r - t}} + \sum_{1 \leq j \leq r} \sum_{t = 2}^{2r-1} \sum_{\mathcal P \in \mathcal P(r,t,j)} \frac{L^{P(\mathcal P)}}{\pi_N(x)^{Q(\mathcal P)}}\left(\begin{cases} (\log \log x)^{t} &\text{ if }t \leq j\\
1 &\text{ if } t > j\\
\end{cases}\right)\\
& + \frac{x^{E(r)\pi_N(x)}8^{\nu(N)}}{kN},
\end{split}
\end{equation*}
converge to 0.  
\begin{itemize}
\item The error term
 $$\frac{x^{E(r)\pi_N(x)}8^{\nu(N)}}{kN} \to 0$$
 if we consider families $\mathcal F_{N,k}$ with levels $N = N(x)$ and even weights $k = k(x)$ such that $$ \frac{\log \left(k{N}/8^{\nu(N)}\right)}{x} \to \infty \text{ as }x \to \infty.$$
 
\item A choice of $L$ such that
 $$L(x) = \o\left(\pi_N(x)^{2/3}\right)$$
 will ensure that
 $$\sum_{t=2}^{2r} \frac{L^{3r - 2t}}{\pi_N(x)^{2r- t}} \to 0\text{ as }x \to \infty.$$
 
\item The obstruction to obtaining the convergence of the $r$-th power moments of $R_2(g,\rho)(f)$ for any $r \geq 2$ comes from the sum
$$\sum_{1 \leq j \leq r} \sum_{t = 2}^{2r-1} \sum_{\mathcal P \in \mathcal P(r,t,j)} \frac{L^{P(\mathcal P)}}{\pi_N(x)^{Q(\mathcal P)}}\left(\begin{cases} (\log \log x)^{t} &\text{ if }t \leq j\\
1 &\text{ if } t > j\\
\end{cases}\right).$$
Individual components of this sum depend on multiple parameters such as the choice of $t$, choice of $j$,  and the values $x_{i,d}$ coming from each partition in $\mathcal P(r,t,j)$.  As $r$ increases, the size of $\mathcal P(r,t,j)$ increases quite rapidly, and the parameters $x_{i,d}$ in each partition behave differently.   As a result, in the above sum, we are currently unable to determine a uniform choice for $L(x)$ which will ensure that each of its components would go to 0.  Theorem \ref{higher-moments-bounds}, therefore, is a first step towards the study of higher-power moments of $R_2(g,\rho)(f)$. 

\item For specific values of $r$, we can list out the set of partitions $\mathcal P(r,t,j)$, apply Theorem \ref{higher-moments-bounds} and make a choice of $L$ that lets each error term in the theorem go to 0.  For these values of $r$, we are able to obtain the convergence of the moments $\left(R_2 (g,\rho) (f)\right)^r$ for clearly specified choices for $L(x)$, and under appropriate growth conditions for $|\mathcal F_{N,k}|$. 
\end{itemize}
In the context of the second moment, by applying Theorem \ref{higher-moments-bounds}, we have the following theorem.
 \begin{thm}\label{variance_main_theorem_squarefree}
Let us consider families $\mathcal{F}_{N,k}$ with levels $N=N(x)$ and even weights $k=k(x)$. Let $g$, $\rho$ be real valued, even functions $\in C^{\infty}(\R)$ in the Schwartz class with Fourier transforms supported in $[-1,1]$ and $L=L(x)\to \infty$ as $ x \to \infty.$  Let $0 < \psi < 1$, $\psi \neq 1/2$, and let $A = 2\sin^2\pi\psi$.  
\begin{itemize}
\item [{\bf(a)}]
 With $\rho_L,\,G_x,\,R_2(g,\rho)(f)$ and $T(g,\rho)$ as above, we have
\begin{equation}
\begin{split}
&\frac{1}{|\mathcal F_{N,k}|} \sum_{f \in \mathcal F_{N,k}} \left(R_2 (g,\rho) (f) - \frac{T(g,\rho)}{4 L}\right)^2\\
&\ll \frac{1}{L} + \frac{L^2 (\log \log x)^2}{\pi_N(x)^2} + \frac{L(\log \log x)^2}{\pi_N(x)} +   \frac{L^{1/2}\log \log x}{\pi_N(x)^{1/2}}\\
& + \frac{x^{C\pi_N(x)}8^{\nu(N)}}{k{N}},  
\end{split}
\end{equation}
where $C>0$ is an absolute constant. 
\item [{\bf(b)}] In particular, if we choose $L(x) \to \infty$ such that $L(x)=\o \left (  \frac{ \pi_N(x)}{(\log \log x)^2} \right)$, and consider families $\mathcal F_{N,k}$ with levels $N = N(x)$ and even weights $k = k(x)$ such that $$ \frac{\log \left(k{N}/8^{\nu(N)}\right)}{x} \to \infty \text{ as }x \to \infty,$$
then
$$\displaystyle\lim_{x \to \infty}\frac{1}{|\mathcal F_{N,k}|} \sum_{f \in \mathcal F_{N,k}} \left(R_2 (g,\rho) (f)  \right)^2 =  (A^2 \widehat{g}(0)\rho \ast \rho (0))^2 $$
and
$$\displaystyle\lim_{x \to \infty}\frac{1}{|\mathcal F_{N,k}|} \sum_{f \in \mathcal F_{N,k}} \left(R_2 (g,\rho) (f) -  A^2 \widehat{g}(0)\rho \ast \rho (0) \right)^2 =0. $$
\end{itemize}
\end{thm}
The above theorem tells us that $\mathbb E[(R_2(g,\rho)(f))^2] \sim \mathbb E[(R_2(g,\rho)(f))]^2$ for very rapidly growing families $\mathcal F_{N,k}$.  In turn, these are asymptotic to what one would expect from a Poissonnian model.  This indicates an affirmative answer to Question  \ref{question-shrinking} for a {\bf random} Hecke newform in $S(N,k)$ with appropriate parameters as specified in Theorem \ref{variance_main_theorem_squarefree}(c). 

We also record the following theorem for the asymptotic moments of $R_2(g,\rho)(f)^r$ for $r = 3$.
\begin{thm}\label{r=3 and more}
Let us consider families $\mathcal{F}_{N,k}$ with levels $N=N(x)$ and even weights $k=k(x)$. Let $g$, $\rho$ be real valued, even functions $\in C^{\infty}(\R)$ in the Schwartz class with Fourier transforms supported in $[-1,1]$ and $L=L(x)\to \infty$ as $ x \to \infty.$  Let $0 < \psi < 1$, $\psi \neq 1/2$, and let $A = 2\sin^2\pi\psi$.  
\begin{itemize}
\item [{\bf(a)}]
 With $\rho_L,\,G_x,\,R_2(g,\rho)(f)$ and $T(g,\rho)$ as above, we have
\begin{equation}
\begin{split}
&\frac{1}{|\mathcal F_{N,k}|} \sum_{f \in \mathcal F_{N,k}} R_2 (g,\rho) (f)^3 - \left(\frac{T(g,\rho)}{4 L}\right)^3\\
&\ll \sum_{t=2}^{6} \frac{L^{9 - 2t}}{\pi_N(x)^{6- t}} + \frac{L^3(\log \log x)^2}{\pi_N(x)} + \frac{L^3(\log \log x)^3}{\pi_N(x)^2} \\
&+  \frac{L^2}{\pi_N(x)^3} +   \frac{L^4}{\pi_N(x)^4} +  \frac{L^4}{\pi_N(x)^6} + \frac{L^{1/2}\log \log x}{\pi_N(x)^{1/2}}\\
& + \frac{x^{E(3)\pi_N(x)}8^{\nu(N)}}{kN},\\
\end{split}
\end{equation}
where $E(3)>0$ is an absolute constant. 
\item [{\bf(b)}] In particular, if we choose $L(x) \to \infty$ such that $L(x)=\o \left (  \frac{ \pi_N(x)}{(\log \log x)^2} \right)$, and consider families $\mathcal F_{N,k}$ with levels $N = N(x)$ and even weights $k = k(x)$ such that $$ \frac{\log \left(k{N}/8^{\nu(N)}\right)}{x} \to \infty \text{ as }x \to \infty,$$
then
$$\displaystyle\lim_{x \to \infty}\frac{1}{|\mathcal F_{N,k}|} \sum_{f \in \mathcal F_{N,k}} \left(R_2 (g,\rho) (f)  \right)^3 =  (A^2 \widehat{g}(0)\rho \ast \rho (0))^3. $$
\end{itemize}

\end{thm}

\subsection*{Organization of the article}
This article is organized as follows:

In Section \ref{section-pair-cor}, we explain how the pair correlation function for our families of Hecke angles is set up, and how it can be reduced to appropriate cosine sums.  We also describe the smooth analogue of the pair correlation function in terms of smooth test functions with compactly supported Fourier transforms.  

Section \ref{sec:small-scales} addresses the fundamental question of equidistribution of Hecke angles in shrinking intervals around a point.  We explain how the discrepancy estimates in the Sato-Tate equidistribution theorem help to determine the rate at which one can allow the intervals to shrink.  We also describe how averaging over the Hecke newforms provides us greater flexibility to address this issue.  We recall estimates from the Eichler-Selberg trace formula, which form a key technique in evaluating moments of the pair correlation function.  In Section \ref{sec:dim-bound}, we discuss formulas and bounds for $|\mathcal F_{N,k}|$.  This discussion will enable us to determine, in later sections, the sizes of families $\mathcal F_{N,k}$ over which the second moment, and some other higher moments of the pair correlation function can be shown to converge.  In Section \ref{remarks-pcf}, we discuss the difficulties encountered in addressing the convergence of the pair correlation function with the help of existing discrepancy estimates, and why it becomes necessary to consider moments.  We also explain the new insights obtained in Theorems \ref{higher-moments-bounds} and \ref{variance_main_theorem_squarefree}. 

In Section \ref{preliminary-estimates}, we review properties and estimates for Hecke eigenvalues which will be required in the study of the higher moments.  

In Section \ref{rM}, we address the higher moments of $R_2(g,\rho)(f)$ and prove Theorem \ref{higher-moments-bounds}.  This section forms the bulk of the article.  The evaluation of the higher moments is done by breaking down powers of $R_2(g,\rho)(f)^r$ into three components.  In Sections \ref{rM-1}, \ref{rM-2} and \ref{rM-3}, we evaluate each of these components respectively.  In Section \ref{Thm2-proof}, we combine the work done in Sections \ref{rM-1} - \ref{rM-3} to obtain a complete proof of Theorem \ref{higher-moments-bounds}.

Section \ref{2-3M} is devoted to proving Theorems \ref{variance_main_theorem_squarefree} and \ref{r=3 and more}.  Both these theorems are applications of Theorem \ref{higher-moments-bounds}.  In particular,  asymptotics of the second moment of the pair correlation function as well as the variance are obtained in Section \ref{2M-V}.  The asymptotic behaviour of the third moment of the pair correlation function is recorded in Section \ref{3M}.  

In Section \ref{conclusion}, we make concluding remarks and leave the reader with two questions motivated by the theorems of this article.

\subsection*{Acknowledgements}
Some of the results of this article are contained in the first named author's doctoral thesis.  We are grateful to Zeev Rudnick and M. Ram Murty for thoughtful inputs and correspondence.  We thank Sneha Chaubey and Neha Prabhu for helpful discussions.

The first named author is supported by a doctoral fellowship from the National Board for Higher Mathematics and the second named author was supported by a Science and Engineering Research Board grant MTR/2019/001108 for part of the duration in which this research was carried out. 

\section{The pair correlation function for Hecke angles}\label{section-pair-cor}
   
Recall,
\begin{equation}\label{modified-families}
\mathcal A_{f,x} := \{\pm \thtp \text{ mod 1}:\, p \leq x,\,(p,N) = 1\} \text{ as }x \to \infty.
\end{equation} 

Note that the family $\mathcal I_L \cap \mathcal A_{f,x}$ has cardinality  
$$ \Lf\,(\psi) + \Lf\,(1 - \psi) \sim \frac{4\pi_N(x)}{L} \text{ as }x \to \infty.$$
Therefore, the mean spacing of the family $\mathcal A_{f,x}$ of Hecke angles in the intervals $\mathcal I_L$ is 
\begin{equation}\label{asymp-mean-spacing}
\frac{|\mathcal I_L|}{\# (\mathcal I_L \cap \mathcal A_{f,x})} \sim
\frac{1}{2 A \pi_N(x)}  \text{ as }x \to \infty.
\end{equation}
We have the following lemma:
\begin{lemma}
Let $0 < \psi < 1$ and $\psi \neq 1/2$.  Then,
\begin{equation*}
\begin{split}
&\#\left\{(i,j):\,i \neq j,\,x_i,x_j \in (\mathcal I_L \cap \mathcal A_{f,x}),\,x_i - x_j \in \left[-\frac{s}{2A\pi_N(x)},\frac{s}{2A\pi_N(x)}\right]\right\}\\
&=\frac{1}{2}\sum_{(p,q) \atop {p,q \leq x \atop{(p,N) = (q,N) = 1 \atop {p \neq q}}}}\Big(\sum_{n \in \Z} \chi_{\left[-\frac{1}{A},\frac{1}{A}\right]}(L (\pm \thtp - \psi + n)) 
\sum_{n \in \Z}\chi_{\left[-\frac{1}{A},\frac{1}{A}\right]}(L (\pm \thtq - \psi + n)) \\
& \quad\quad\quad\quad  \times \sum_{n \in \Z} \chi_{\left[-\frac{s}{2A},\frac{s}{2A}\right]}(\pi_N(x) (\pm \thtp \pm \thtq + n))\Big).
\end{split}
\end{equation*}
\end{lemma}
\begin{proof}
Let
$$I_x := \left[-\frac{s}{2A\pi_N(x)},\frac{s}{2A\pi_N(x)}\right].$$
By a careful counting, we obtain
\begin{equation*}
\begin{split}
&\#\left\{(i,j):\,i \neq j,\,x_i,x_j \in (\mathcal I_L \cap \mathcal A_{f,x}),\,x_i - x_j \in I_x\right\}\\
&= \sum_{(p,q) \atop {p,q \leq x \atop{(p,N) = (q,N) = 1 \atop {p \neq q}}}} (\chi_{\mathcal I_L}(\thtp) + \chi_{\mathcal I_L}(1 - \thtp))(\chi_{\mathcal I_L}(\thtq) + \chi_{\mathcal I_L}(1 - \thtq))\\
& \times 2 B(\chi,\thtp,\thtq,s),
\end{split}
\end{equation*}
where
$$B(\chi,\thtp,\thtq,s) = \chi_{I_x}(\thtp - \thtq - 1) + \chi_{I_x}(\thtp - \thtq )  + \chi_{I_x}(\thtp - \thtq + 1) $$
$$+ \chi_{I_x}(\thtp + \thtq ) + \chi_{I_x}(\thtp + \thtq -1 )  + \chi_{I_x}(\thtp + \thtq -2).$$
The lemma follows immediately.
\end{proof}
Thus, the pair correlation function for the families $\mathcal A_{f,x} \cap \mathcal I_L$ is \begin{equation}\label{pc-counting-local}
\begin{split}
& \frac{1 }{|\mathcal I_L \cap \mathcal A_{f,x}|} \#\left\{(i,j):\,i \neq j,\,x_i,x_j \in (\mathcal I_L \cap \mathcal A_{f,x}),\,x_i - x_j \in \left[-\frac{s}{2A\pi_N(x)},\frac{s}{2A\pi_N(x)}\right]\right\}\\
& \sim \frac{L }{8 \pi_N(x)} \sum_{(p,q) \atop {p,q \leq x \atop{(p,N) = (q,N) = 1 \atop {p \neq q}}}}\sum_{n \in \Z} \chi_{\left[-\frac{1}{A},\frac{1}{A}\right]}(L (\pm \thtp - \psi + n)) 
\sum_{n \in \Z}\chi_{\left[-\frac{1}{A},\frac{1}{A}\right]}(L (\pm \thtq - \psi + n)) \\
& \quad\quad\quad\quad  \times \sum_{n \in \Z} \chi_{\left[-\frac{s}{2A},\frac{s}{2A}\right]}(\pi_N(x) (\pm \thtp \pm \thtq + n)).
\end{split}
\end{equation}
We consider a smooth analogue of the right hand side. 

Let $\rho$ be an even test function in the Schwartz class such that the Fourier transform $\widehat{\rho}$ of $\rho$ is smooth and compactly supported, and normalized so that
$$\sup\{|t|:\,\widehat{\rho}(t) \neq 0\} = 1.$$
We define
$$\rho_L(\theta) := \sum_{n \in \Z} \rho (L (\theta + n)).$$
$\rho_L(\theta)$ is a 1-periodic function, localized to a scale of $1/L$, and can be thought of as a smooth version of a function that counts points $\theta$ such that $|\theta| < 1/L$.  It has the Fourier expansion
\begin{equation}\label{rho-FE}
\begin{split}
&\rho_L(\theta) = \sum_{|l| \leq L} \widehat{\rho_L}(l) e(l\theta)
= \widehat{\rho_L}(0) + \sum_{1 \leq l \leq L} \widehat{\rho_L}(l) 2\cos (2\pi l \theta),
\end{split}
\end{equation}
where $\widehat{\rho_L}(l) = \frac{1}{L} \widehat{\rho}\left(\frac{l}{L}\right)$.

Similarly, let $g$ be an even test function satisfying the same properties as $\rho$, that is, an even test function such that the Fourier transform $\widehat{g}$ of $g$ is smooth and compactly supported, and normalized so that
$$\sup\{|t|:\,\widehat{g}(t) \neq 0\} = 1.$$
We define 
\begin{equation}
\begin{split}
&G_x(\theta) := \sum_{n \in \Z} g\left(\pi_N(x)(\theta + n)\right)\\
&= \sum_{|n| \leq \pi_N(x)} \widehat{G_{x}}(n) e(n\theta)\\
&= \widehat{G_{x}}(0) + \sum_{1 \leq n \leq \pi_N(x)} \widehat{G_{x}}(n) 2\cos (2\pi n \theta),
\end{split}
\end{equation}
where $\widehat{G_{x}}(n) := \frac{1}{\pi_N(x)}\widehat{g}\left(\frac{n}{\pi_N(x)}\right).$
Similar to the case of $\rho_L$, the function $G_x(\theta)$ is a 1-periodic function, localized to a scale of $1/\pi_N(x)$, and therefore, effectively counts points $\theta$ such that $|\theta| < 1/\pi_N(x)$.

The smooth analogue of the right hand side of \eqref{pc-counting-local} is defined as \begin{equation}
\label{pc-smooth-simplified}
R_2(g,\rho)(f) := \frac{L}{8\pi_N(x)} \sum_{(p,q) \atop {p \neq q \leq x \atop{(p,N) = (q,N) = 1}}}\rho_L(\pm \thtp - \psi)\rho_L(\pm \thtq - \psi) G_x(\pm \thtp \pm \thtq).
\end{equation}

We recall the following classical result, which gives a recursive relation between $a_{f}(p^l),$ $l \geq 0.$
For an integer $l \geq 0,$ 
\begin{equation}\label{Ramanujan-Hecke}
2 \cos 2 \pi l \thtp = 
\begin{cases}
2, &\text{ if }l = 0,\\
a_f(p^{2l}) - a_f(p^{2l-2}), &\text{ if }l \geq 1.
\end{cases}
\end{equation}
Denote
$$U(l) = \widehat{\rho}\left(\frac{l}{L}\right) (2\cos 2\pi l\psi) - \widehat{\rho}\left(\frac{l+1}{L}\right) (2\cos 2\pi (l+1)\psi),\,0 \leq l \leq L$$
and
$$G(n) = \widehat{g}\left(\frac{n}{\pi_N(x)}\right),\,0 \leq n \leq \pi_N(x).$$
We have,
\begin{equation}\label{rho-FE-f-0}
\begin{split}
&\rho_L(\pm\thtp - \psi) = \rho_L(\thtp - \psi) + \rho_L(-\thtp - \psi)\\
&=  \sum_{|l| \leq L} \widehat{\rho_L}(l) \left\{e(l (\thtp - \psi)) + e(l (-\thtp - \psi))\right\}\\
&=  \sum_{|l| \leq L} \widehat{\rho_L}(l) e(-l\psi) 2\cos (2\pi l \thtp)\\
&= \left(2\widehat{\rho_L}(0) +  \sum_{1 \leq l \leq L} \widehat{\rho_L}(l) (2\cos 2\pi l\psi)2\cos (2\pi l \thtp )\right).\\
&= \frac{1}{L}\sum_{0 \leq l \leq L} U(l) a_f(p^{2l}).
\end{split}
\end{equation}
Similarly,
\begin{equation*}
\begin{split}
&G_{x}\left(\pm \thtp \pm \thtq\right)\\
&=\frac{1}{\pi_N(x)}\sum_{|n| \leq \pi_N(x)} \widehat{g}\left(\frac{n}{\pi_N(x)}\right) e\left(\pm n\thtp \pm n\thtq\right))\\
&= \frac{1}{\pi_N(x)} \left(4 G(0) + \sum_{n \geq 1} 2G(n) (2 \cos 2\pi n \thtp) (2 \cos 2\pi n \thtq)\right),
\end{split}
\end{equation*}
Using the above Fourier expansions, 
\begin{equation}\label{pc-simplified-a-little}
\begin{split}
R_2(g,\rho)(f) &= \frac{L}{8\pi_N(x)}\frac{1}{L^2 \pi_N(x)}\sum_{(p,q) \atop {p \neq q \leq x \atop{(p,N) = (q,N) = 1}}} \left[ \sum_{l \geq 0} U(l) a_f(p^{2l}) \right] \left[ \sum_{l' \geq 0} U(l') a_f(q^{2l'}) \right] \\
&\quad \quad \quad \quad \quad \left[4 G(0) + \sum_{n \geq 1} 2G(n)(a_f(p^{2n}) - a_f(p^{2n-2}))(a_f(q^{2n}) - a_f(q^{2n-2}))\right]  \\
\end{split}
\end{equation}

We now denote
$$T_1(p) :=  \sum_{l \geq 0} U(l) a_f(p^{2l}),$$
$$T_2(q) :=  \sum_{l' \geq 0} U(l') a_f(q^{2l'})$$
and
$$T_3(p,q) :=  \sum_{n \geq 0} G(n)A(p,q,n),$$
where
$$A(p,q,n) = \begin{cases}
4 &\text{ if }n = 0\\
2(a_f(p^{2n}) - a_f(p^{2n-2}))(a_f(q^{2n}) - a_f(q^{2n-2})) &\text{ if }n \geq 1.
\end{cases}
$$
Thus, we get
\begin{equation}\label{pc-simplified}
R_2(g,\rho)(f) = \frac{1}{8 \pi_N(x)^2L}\sum_{(p,q) \atop {p \neq q \leq x \atop{(p,N) = (q,N) = 1}}} T_1(p)T_2(q)T_3(p,q).
\end{equation}
Since $\widehat{\rho}$ and $\widehat{g}$ are continuous and compactly supported, we have the bounds $|U(l_i)|,\,|U(k_i)|,\,|G(n_i)| \ll 1$, which will be used in the calculations below.

\subsection{Equidistribution properties of Hecke angles in small scales}\label{sec:small-scales}
In this section, we explain the connection between the error terms in the Sato-Tate distribution theorem and the distribution of the families $\mathcal A_{f,x}$ (defined in \eqref{modified-families}) in shrinking intervals $\mathcal I_L$ where $L=L(x) \to \infty $ as $x \to \infty.$  As we will see below, this provides an insight into the difficulties in obtaining the pair correlation function for a deterministic $f \in \mathcal F_{N,k}$, and why it helps to consider a {\it random} $f \in \mathcal F_{N,k}$ instead.

The question is, what growth conditions on $L = L(x)$ are sufficient to ensure that 
$$\lim_{x \to \infty} \frac{1}{|\mathcal A_{f,x}|} \sum_{\theta \in \mathcal A_{f,x}} \rho_L(\theta - \psi) = \int_0^1 \rho_L(t - \psi)\mu(t) dt?$$
Define
$$N_{\rho,L,f}(x) : = \sum_{\theta \in \mathcal A_{f,x}} \rho_L(\theta - \psi).$$

By \eqref{rho-FE-f-0},
\begin{equation}\label{rho-FE-f}
\begin{split}
\frac{N_{\rho,L,f}(x)}{|\mathcal A_{f,x}|}
 =&\frac{1}{2\pi_N(x)} \left(2\widehat{\rho_L}(0) \pi_N(x) + \sum_{1 \leq l \leq L} \widehat{\rho_L}(l) (2\cos 2\pi l\psi)\sum_{p \leq x \atop {(p,N) = 1}} (a_f(p^{2l}) - a_f(p^{2l-2}))\right)\\
= &\frac{1}{2\pi_N(x)}\sum_{0 \leq l \leq L} \frac{U(l)}{L} \sum_{p \leq x \atop {(p,N) = 1}} a_f(p^{2l}).
\end{split}
\end{equation}
It is easy to see that
$$\frac{U(0)}{2L} = \int_0^1 \rho_L(t - \psi) \mu(t) dt.$$

The following proposition is a consequence of Thorner's discrepancy estimates.
\begin{prop}\label{Thorner-IL}
Let $N \geq 1$ and $k \geq 2$ be integers with $k$ even.  Let $f \in \mathcal F_{N,k}$ be a non-CM newform.  Let $\rho$ be as defined in \eqref{rho}.  If $0 < \psi < 1$ and $L$ is chosen such that $$ L \leq \frac{c_{11} \sqrt{\log x}}{2 \sqrt{\log (kN \log x)}}$$
for a suitably small constant $c_{11}$, we have
\begin{equation}\label{asymp-L}
\frac{N_{\rho,L,f}(x)}{2\pi_N(x)} \sim \int_0^1 \rho_L(t - \psi) \mu(t) dt \text{ as }x \to \infty.
\end{equation}
\end{prop}

\begin{proof}
The key ingredient in the proof is the following estimate which follows from \cite[Proposition 2.1]{Thorner}.  Let $f \in \mathcal F_{N,k}$ be a non-CM newform.  Then there exist constants $c_{9}$ (suitably large) and $c_{10}$ and $c_{11}$ (suitably small) such that if $$ 2l \leq c_{11} \sqrt{\log x} / \sqrt{ \log (kN\log x)},$$ then  
\begin{equation}\label{Thorner-1}
 \sum_{\substack{p \leq x \\ (p,N) = 1}}a_f\left(p^{2l}\right) \ll l^{2}\pi_N(x) \left( x^{- \frac{1}{2 c_{9}l}} +  e^{- \frac{c_{10} \log x}{4l^{2} \log (2kNl)}} + e^{- \frac{c_{10} \sqrt{\log x}}{\sqrt{2l}}} \right ).
 \end{equation}
Note that $\widehat{\rho}$ is a compactly supported, continuous function and therefore, absolutely bounded.  Thus, 
$$\frac{U(l)}{L}  = \frac{1}{L}\left[\widehat{\rho}\left(\frac{l}{L}\right) \cos 2\pi l\psi - \widehat{\rho}\left(\frac{l+1}{L}\right) \cos 2\pi (l+1)\psi\right] \ll \frac{1}{L}.$$  
Choosing $$L \leq \frac{c_{11} \sqrt{\log x}}{2 \sqrt{\log (kN \log x)}},$$
we have
\begin{equation*}
\begin{split}
& \frac{N_{\rho,L,f}(x)}{2\pi_N(x)} -  \frac{U(0)}{2L}\\
&= \frac{1}{2\pi_N(x)}\sum_{1 \leq l \leq L} \frac{U(l)}{L} \sum_{p \leq x \atop {(p,N) = 1}} a_f(p^{2l})\\
& \ll \frac{1}{L\pi_N(x)} \sum_{1 \leq l \leq L} l^{2} \left( x^{- \frac{1}{2 c_{9}l}} +  e^{- \frac{c_{10} \log x}{4l^{2} \log (2kNl)}} + e^{- \frac{c_{10} \sqrt{\log x}}{\sqrt{2l}}} \right )\\
& \ll \frac{L^2}{\pi_N(x)} \left( x^{- \frac{1}{2 c_{9}L}} +  e^{- \frac{c_{10} \log x}{4L^{2} \log (2kNL)}} + e^{- \frac{c_{10} \sqrt{\log x}}{\sqrt{2L}}} \right ).\\
\end{split}
\end{equation*}
The above term $\to 0 \text{ as }x \to \infty$, 
if $ L \ll  \sqrt{\log x} / \sqrt{ \log (kN\log x)}.$  This proves the proposition.
\end{proof}

The limitation of the above proposition is in the range of $L$ for which it holds.  Is it possible to obtain the asymptotic \eqref{asymp-L} for a larger range of $L$, for example, $L(x) \ll x^{\alpha}$ for some $\alpha > 0?$  It turns out that this can be done if one assumes strong analytic hypotheses on symmetric power $L$-functions corresponding to a non-CM newform $f$ with squarefree level $N$.  In this respect, using conditional discrepancy estimates of Rouse and Thorner \cite{RT}, we have the following proposition:

\begin{prop}\label{Thorner-Rouse-IL}
Let $N \geq 1$ and $k \geq 2$ be integers with $N$ squarefree and $k$ even.  Let $f \in \mathcal F_{N,k}$ be a non-CM newform such that for each $l \geq 0$, the following hypotheses hold:
\begin{enumerate}
\item The symmetric power $L$-function $L(s, \Sym^lf)$ is the $L$-function of a cuspidal automorphic representation on $GL_{l+1}(\mathbb A_{\Q})$.
\item The Generalized Riemann hypothesis holds for $L(s, \Sym^lf)$.
\end{enumerate}
Let $\rho$ be as defined in \eqref{rho}.  If $0 < \psi < 1$ and $L$ is chosen such that $$ L = \o\left(\frac{x^{1/2c}}{(\log x)^{2/c}}\right)$$
for a constant $c > 1$, we have
$$ \frac{N_{\rho,L,f}(x)}{2\pi_N(x)} \sim \int_0^1 \rho_L(t - \psi) \mu(t) dt \text{ as }x \to \infty.$$
\end{prop}
\begin{proof}
By \cite[Proposition 3.3]{RT}, if $l \geq 1$ and $x \geq 5 \times 10^5$,
\begin{equation}\label{Rouse-Thorner-1}
 \sum_{\substack{p \leq x \\ (p,N) = 1}} a_f\left(p^{2l}\right) \ll (l \log l) \sqrt{x}\log x \log (N(k-1)).
\end{equation}
As in the proof of Proposition \ref{Thorner-IL}, we have, for $x  \geq 5 \times 10^5$,
\begin{equation*}
\begin{split}
& \frac{N_{\rho,L,f}(x)}{2\pi_N(x)} -  \frac{U(0)}{2L} = \frac{1}{2\pi_N(x)}\sum_{1 \leq l \leq L} \frac{U(l)}{L} \sum_{p \leq x \atop {(p,N) = 1}} a_f(p^{2l})\\
& \ll L \log L \frac{\sqrt{x}\log x}{\pi_N(x)} \log (N(k-1))\\
&\ll L \log L\frac{(\log x)^2}{\sqrt{x}} \log (N(k-1)).
\end{split}
\end{equation*}
Let us choose $L(x)$ such that
$$L(x) = \o\left(\frac{x^{1/2c}}{(\log x)^{2/c}}\right) \text{ for }c>1.$$
Then,
$$ L \log L\frac{(\log x)^2}{\sqrt{x}}\log (N(k-1)) \to 0 \text{ as }x \to \infty.$$
\end{proof}

By the results of Newton and Thorne \cite{NT1},\,\cite{NT2}, the hypothesis (1) in the above proposition is now known to be true for all $l \geq 1$.  Therefore, in comparison to Proposition \ref{Thorner-IL}, we have a larger range of $L$ for which \eqref{asymp-L} holds, under GRH.

In what follows, we make some remarks here about $N_{\rho,L,f}(x)$ for a random $f \in \mathcal F_{N,k}$.  We derive the expected value of $N_{\rho,L,f}(x)$ as we average over all $f \in \mathcal F_{N,k}$ (not just non-CM newforms).  As we will see, averaging enables us to obtain the asymptotic in \eqref{asymp-L} over a more flexible range of $L$. 

We first introduce the following notation:  for any $\phi:\, \mathcal F_{N,k} \to \C$, we denote
$$\langle \phi \rangle := \frac{1}{|\mathcal F_{N,k}|}\sum_{f \in \mathcal F_{N,k}} \phi(f).$$
We recall an estimate that follows from the Eichler-Selberg trace formula for the trace of Hecke operators on the space of primitive cusp forms of weight $k$ and level $N$ (see \cite[Section 3]{MS2}).
\begin{prop}\label{trace-estimate}
Let $k$ be a positive even integer and $N$ be a positive integer.  For a positive integer $n>1$ such that $(n,N) = 1,$ we have
\begin{equation}\label{trace}
\sum_{f \in \mathcal F_{N,k}} a_f(n) = 
\begin{cases}
\frac{|\mathcal F_{N,k}|}{n^{1/2}} + \O\left( n\sigma_0(n)4^{\nu(N)}\right),&\text{ if }n\text{ is a square }\\
\O\left( n\sigma_0(n)4^{\nu(N)}\right),&\text{ otherwise}.
\end{cases}
\end{equation}
Here, $\sigma_0(n)$ refers to the number of positive divisors of $n$ and the implied constant in the error term is absolute.
\end{prop}
We deduce the following corollary from \eqref{trace}.
\begin{cor}\label{trace-prime-power}
Let $k$ be a positive even integer and $N$ be a positive integer.  For a prime $p$ such that $(p,N) = 1$ and $l \geq 1$, we have
$$\frac{1}{|\mathcal F_{N,k}|}\sum_{f \in \mathcal F_{N,k}} a_f(p^{2l}) = \frac{1}{p^l} + \O\left(\frac{lp^{2l}4^{\nu(N)}}{|\mathcal F_{N,k}|}\right).$$
\end{cor}


We use the above estimate to prove the following proposition, which tells us that an average version of \eqref{asymp-L} holds over a range of $L$ that grows with the size of the families $\mathcal F_{N,k}$ under consideration. 

\begin{prop}\label{growth-all}
Consider families $\mathcal F_{N,k}$ with levels $N = N(x)$ and even weights $k = k(x)$ such that
$$ \frac{\log \left(|\mathcal F_{N,k}|/4^{\nu(N)}\right)}{\log x} \to \infty \text{ as }x \to \infty.$$
Let $\rho$ be as chosen above.  If $0 < \psi < 1$ and $L$ is chosen such that $$ L = \o\left(\frac{\log \left(|\mathcal F_{N,k}|/4^{\nu(N)}\right)}{\log x}\right),$$ we have
$$ \left \langle \frac{N_{\rho,L,f}(x)}{2\pi_N(x)} \right \rangle  \sim  \int_0^1 \rho_L(t - \psi) \mu(t) dt \text{ as }x \to \infty.$$
\end{prop}

\begin{proof}
Applying \eqref{rho-FE-f}, Corollary \ref{trace-prime-power} and the estimate $U(l) \ll 1$,
\begin{equation*}
\begin{split}
& \left \langle  \frac{N_{\rho,L,f}(x)}{2\pi_N(x)} -  \frac{U(0)}{2L} \right \rangle \\
&=\frac{1}{\pi_N(x)}\sum_{1 \leq l \leq L} \frac{U(l)}{L}  \sum_{p \leq x \atop {(p,N) = 1}} \left \langle a_f(p^{2l}) \right \rangle\\
&= \frac{1}{2\pi_N(x)}\sum_{1 \leq l \leq L} \frac{U(l)}{L}  \sum_{p \leq x \atop {(p,N) = 1}} \left(\frac{1}{p^l} + \O\left(\frac{lp^{2l}4^{\nu(N)}}{|\mathcal F_{N,k}|}\right)\right)\\
&= \frac{1}{2\pi_N(x)L}\sum_{1 \leq l \leq L}  \sum_{p \leq x \atop {(p,N) = 1}} \frac{U(l)}{p^l} + \O\left(\frac{4^{\nu(N)} x^{2L}\pi_N(x)}{|\mathcal F_{N,k}|}\right)\\
&= \O\left(\frac{\log \log x}{\pi_N(x)}\right) + \O\left(\frac{4^{\nu(N)} x^{2L}}{|\mathcal F_{N,k}|}\right)\\
\end{split}
\end{equation*}
$$\text{ If }L = \o\left(\frac{\log \left(|\mathcal F_{N,k}|/4^{\nu(N)}\right)}{\log x}\right), \text{ then, }x^{2L} = \o\left(\frac{|\mathcal F_{N,k}|}{4^{\nu(N)}}\right).$$
Thus, we have
$$\left \langle  \frac{N_{\rho,L,f}(x)}{2\pi_N(x)} -  \frac{U(0)}{2L}  \right \rangle \to 0\text{ as }x \to \infty.$$

\end{proof}
\subsection{Remarks on dimension formulas for $S(N,k)$}\label{sec:dim-bound}

Formulas and bounds for $|\mathcal F_{N,k}|$, that is, the dimension of the space of primitive cusp forms in $S(N,k)$ have been well studied \cite{Martin}.   It can be shown (\cite[Theorem 6(c)]{Martin}) that for a positive integer $N$,
\begin{equation}\label{dim}
	\frac{A_0(k-1)}{12}\phi(N) + \O(\sqrt{N}) < |\mathcal F_{N,k}| < \frac{(k-1)}{12}\phi(N) + \O\left(2^{\nu(N)}\right),
	\end{equation}
where 
$$A_0 = \prod_p \left(1 - \frac{1}{p^2 - p}\right) \approx 0.373956.$$
From the above, one gets the following estimate.
\begin{lemma}\label{dim-est-1}
Let $N$ and $k$ be positive integers with $k$ even.  Then, as $N + k\to \infty$,
$$|\mathcal F_{N,k}| \gg \frac{Nk}{2^{\nu(N)}},$$
where $\nu(N)$ denotes the number of distinct prime divisors of $N$.
\end{lemma}
\begin{proof}
Note that $\phi(N)  = N\prod_p \left(1 - \frac{1}{p}\right)\geq \frac{N}{2^{\nu(N)}}$.  By \eqref{dim},
$$|\mathcal F_{N,k}| > \frac{A_0(k-1)}{12}\frac{N}{2^{\nu(N)}} + \O(\sqrt{N}),$$
and therefore,
$$|\mathcal F_{N,k}| \frac{2^{\nu(N)}}{kN} > \frac{A_0(k-1)}{12k} + \O\left(\frac{2^{\nu(N)}}{k\sqrt{N}}\right).$$
But, $\frac{2^{\nu(N)}}{k\sqrt{N}} \to 0$ if $N+k \to \infty$.

This proves the lemma.

\end{proof}

  Therefore,
\begin{equation}\label{dim-est}
\frac{n\sigma_0(n)4^{\nu(N)}}{|\mathcal F_{N,k}|} \ll \frac{n\sigma_0(n)8^{\nu(N)}}{k{N}}.
\end{equation}
By Corollary \ref{trace-prime-power},
$$\frac{1}{|\mathcal F_{N,k}|}\sum_{f \in \mathcal F_{N,k}} a_f(p^{2l}) = \frac{1}{p^l} + \O\left(\frac{lp^{2l}8^{\nu(N)}}{k{N}}\right).$$
We therefore deduce the following corollary from Proposition \ref{growth-all}.
\begin{cor}\label{growth-all-1}
We consider families $\mathcal F_{N,k}$ with levels $N = N(x)$ and even weights $k = k(x)$ such that
$$ \frac{\log \left(k{N}/8^{\nu(N)}\right)}{\log x} \to \infty \text{ as }x \to \infty.$$
Let $\rho$ be as chosen above.  If $0 < \psi < 1$ and $L$ is chosen such that $$ L = \o\left(\frac{\log \left(k{N}/8^{\nu(N)}\right)}{\log x}\right),$$ we have
$$ \left \langle \frac{N_{\rho,L,f}(x)}{2\pi_N(x)} \right \rangle  \sim  \int_0^1 \rho_L(t - \psi) \mu(t) dt \text{ as }x \to \infty.$$
\end{cor}
\begin{remark}
The above estimates tell us that if $N$ varies over prime levels, then the above asymptotic will hold for families $\mathcal F_{N,k}$ such that
$$\frac{\log kN}{\log x} \to \infty\text{ as }x \to \infty.$$
If $N = 1$, then  the above asymptotic will hold for families $\mathcal F_{1,k}$ such that
$$\frac{\log k}{\log x} \to \infty\text{ as }x \to \infty.$$
\end{remark}
\subsection{Remarks on the pair correlation function}\label{remarks-pcf}

In Section \ref{sec:small-scales}, we saw that Thorner's unconditional estimate \eqref{Thorner-1} and the conditional estimate of Rouse and Thorner \eqref{Rouse-Thorner-1} for sums $\sum_{ p  \leq x}a_f(p^{2u})$ play a pivotal role in deriving equidistribution properties of Hecke angles.  By \eqref{pc-simplified-a-little}, these sums also appear in the pair correlation function $R_2(g,\rho)(f)$.  But, we need estimates for these sums for $u$ as large as $\pi_N(x)$, where as \eqref{Thorner-1} holds for
$$ u \ll \sqrt{\log x} / \sqrt{ \log (kN\log x)}.$$
The conditional estimate \eqref{Rouse-Thorner-1} holds for all $u \geq 1$, if $x$ is sufficiently large.  However, when we apply this estimate to \eqref{pc-simplified-a-little}, we get
$$R_2(g,\rho)(f) \ll x(\log x)^2 \log^2(N(k-1)),$$
which is not enough to determine the convergence of $R_2(g,\rho)(f)$ as $x \to \infty$.

On the other hand, one can apply the trace formula estimates in Proposition \ref{trace-estimate} to obtain the limit
$$\lim_{x \to \infty}\frac{1}{|\mathcal F_{N,k}|}\sum_{f \in \mathcal F_{N,k}} R_2(g,\rho)(f) = A^2 \widehat{g}(0)\rho \ast \rho(0),$$
albeit for rapidly increasing families $\mathcal F_{N,k}$, parametrized by
\begin{equation}\label{growing-families}
 \frac{\log \left(k{N}/8^{\nu(N)}\right)}{x} \to \infty.
 \end{equation}
 This was the main observation of Theorem \ref{pc-error} \cite{BS}.
 
 The main theorems of this article, Theorems \ref{higher-moments-bounds} and \ref{variance_main_theorem_squarefree} make the following fundamental observations over and above the results of \cite{BS}.
\begin{enumerate}

\item To simplify the $r$-th power moment $$\frac{1}{|\mathcal F_{N,k}|}\sum_{f \in \mathcal F_{N,k}} (R_2(g,\rho)(f))^r$$ addressed in this article, we require a delicate balancing act between estimates for several sums of type
$$\left \langle \sum_{p_1,p_2,\dots,p_r,q_1,q_2,\dots,q_r \leq x } a_f(p_1^{2u_1})a_f(p_2^{2u_2})\dots a_f(p_r^{2u_r})a_f(q_1^{2v_1})a_f(q_2^{2v_2})\dots a_f(q_r^{2v_r}) \right \rangle$$
and the ranges of $u_i$ and $v_i$ in each of these sums.  We also have to account for repetitions among the primes.  Due to multiplicative relationships between $a_f(p^{2u})$ for various prime powers $\,u \geq 0$, the above sum can be written in terms of appropriate polynomials (this is explained in Notations \ref{notation-t-distinct-primes}(c) and \ref{notation-t-distinct-primes-j}(c)).  Here, we are called upon to record the growth of each coefficient in such a polynomial.  By tracking and combining all the above parameters and a multilayered application of the Eichler-Selberg trace formula, we are able to obtain Theorem \ref{higher-moments-bounds}. 


\item The second observation is that for small values of $r$, we obtain the convergence of the $r$-th moment of $R_2(g,\rho)(f)$ with the same choice of $L$ and the same growth conditions for $\mathcal F_{N,k}$ as those required for the convergence of the first moment of $R_2(g,\rho)(f)$.  This is explained in Section \ref{2M-V}.  However, we have obstructions in asserting such a convergence for larger values of $r$ due to the non-uniform behaviour of the parameters coming from several different set-partitions in Theorem \ref{higher-moments-bounds}.  

\item The Katz-Sarnak conjecture implies that as $N + k \to \infty$,
$$ \lim_{x \to \infty} \frac{1}{|\mathcal F_{N,k}|}\sum_{f \in \mathcal F_{N,k}} \left(R_2(g,\rho)(f)\right)^2 = \left(\lim_{x \to \infty}\frac{1}{|\mathcal F_{N,k}|}\sum_{f \in \mathcal F_{N,k}} (R_2(g,\rho)(f))\right)^2.$$
That is, 
$$\mathbb E[(R_2(g,\rho)(f))^2] \sim \mathbb E[(R_2(g,\rho)(f))]^2\text{ as }x \to \infty.$$
By Theorem \ref{variance_main_theorem_squarefree}, we are able to obtain this asymptotic for $(N(x),k(x))$ such that
$$ \frac{\log \left(kN/8^{\nu(N)}\right)}{x} \to \infty.$$

\end{enumerate}


\section{Properties of Hecke eigenvalues and estimates}\label{preliminary-estimates}
In this section, we collect multiplicative properties of Hecke eigenvalues and estimates for their averages.  These will be used in the proof of Theorems \ref{higher-moments-bounds}, \ref{variance_main_theorem_squarefree} and \ref{r=3 and more}.  We start with the following classical multiplicative relationship between Hecke eigenvalues.
 \begin{lemma}\label{Hecke-multiplicative}
For primes $p_1,\,p_2 $ coprime to the level $N$ and nonnegative integers $i,\,j,$ 
$$ a_f(p_1^i)a_f(p_2^j) = 
\begin{cases}
a_f(p_1^ip_2^j), &\text{ if }p_1 \neq p_2,\\
\sum_{l = 0}^{\min{(i,j)}}a_f(p_1^{i + j - 2l}), &\text{ if }p_1 = p_2.
\end{cases}$$
Moreover, if $p_1 = p_2,$ then
$$
\left(a_f(p_1^{2n_1}) - a_f(p_1^{2n_1-2})\right) \left(a_f(p_2^{2n_2}) - a_f(p_2^{2n_2-2})\right) $$
$$ = \begin{cases}
a_f(p_1^{2n_1+ 2n_2}) - a_f(p_1^{2n_1+ 2n_2-2}) + a_f(p_1^{|2n_1- 2n_2|}) - a_f(p_1^{|2n_1- 2n_2| - 2}), &\text{ if }n_1 \neq n_2,\\
a_f(p_1^{4n_1}) - a_f(p_1^{4n_1-2}) + 2,&\text{ if }n_1 = n_2.
\end{cases}$$
\end{lemma}

\begin{lemma}\label{Hecke-multiplicative-diff}
For a prime $p $ coprime to the level $N$ and integers $l \geq 0$ and $n\geq 1$, 
$$a_f(p^{2l})\left(a_f(p^{2n}) - a_f(p^{2n-2})\right) = a_f(p^{2l+2n}) + \begin{cases} a_f(p^{2l-2n})&\text{ if }l \geq n,\\
-a_f(p^{2n-2l-2})&\text{ if }l < n.\\
\end{cases}
$$
\end{lemma}

The following propositions form the key to a major part of the calculations in the rest of this article.
\begin{prop}\label{trace-estimate-later-calculation}
Let $k$ be a positive even integer and $N$ be a positive integer.  Let $n>1$ be a positive integer such that $n$ is a square.  Then, for any fixed real number $c' >1$,
$$\frac{1}{|\mathcal F_{N,k}|}\sum_{f \in \mathcal  F_{N,k}} a_f(n) = \frac{1}{\sqrt{n}} + \O\left(\frac{n^{c'} 8^{\nu(N)}}{k{N}}\right).$$
Here, the implied constant in the error term is absolute.
\end{prop}
\begin{proof}
By Proposition \ref{trace-estimate},
$$\sum_{f \in \mathcal F_{N,k}} a_f(n) = 
\frac{|\mathcal F_{N,k}|}{n^{1/2}} + \O\left( n\sigma_0(n)4^{\nu(N)}\right).$$
By an estimate from elementary number theory,
$$\sigma_0(n) \ll_{\epsilon} n^{\epsilon} \text{ for any }\epsilon >0.$$
We also have, by Lemma \ref{dim-est-1},
$$\frac{1}{|\mathcal F_{N,k}|} \ll \frac{2^{\nu(N)}}{Nk}.$$
The proposition follows by combining the above estimates.

\end{proof}

\begin{prop}\label{sums-distinct-primes}
Let $k = k(x)$ and $N=N(x)$ be positive integers with $k$ even.  Let 
$$\sum_{(p_1,p_2,\dots ,p_t)}^{(1)}$$
 denote a sum over $t$ distinct primes, where each prime is less than or equal to $x$. 
 
 If $(m_1,m_2,\dots ,m_t) = (0,0,\dots,0)$, then 
 \begin{equation}\label{this-lemma-1}
 \frac{1}{|\mathcal F_{N,k}|}\sum_{(p_1,p_2,\dots ,p_t)}^{(1)} \sum_{f \in \mathcal F_{N,k}} a_f(p_1^{2m_1}p_2^{2m_2}\dots p_t^{2m_t}) = \pi_N(x) (\pi_N(x) - 1) (\pi_N(x) - 2) \dots (\pi_N(x) - (t-1)).
\end{equation}
For a tuple $\underline{m} = (m_1,m_2,m_3,\dots,m_t)$, let 
$$Y(\underline{m}) = \#\{1 \leq i \leq t:\,m_i = 0\}.$$
For an integer $a$ such that $0 \leq a \leq t-1$, let
$$\sum_{(m_1,m_2,m_3,\dots,m_t)}^{(a)}$$
denote a sum over a subset of the set of $t$-tuples 
$$\{\underline{m} = (m_1,m_2,m_3,\dots,m_t) \in \Z^t:\, Y(\underline{m}) = a\text{ and }m_i \leq M_i \text{ for each }1 \leq i \leq t \}.$$
Then, 
\begin{equation}\label{this-lemma-other}
\begin{split}
& \frac{1}{|\mathcal F_{N,k}|}\sum_{(p_1,p_2,\dots ,p_t)}^{(1)}\sum_{(m_1,m_2,m_3,\dots,m_t)}^{(a)}\sum_{f \in \mathcal F_{N,k}} a_f(p_1^{2m_1}p_2^{2m_2}\dots p_t^{2m_t})\\
& = \O\left(\pi_N(x)^{a}(\log \log x)^{t-a} \right)+ \O\left(\frac{\pi_N(x)^t x^{(2M_1 + 2M_2 + \dots + 2M_t)c'}8^{\nu(N)}}{k{N}}\right).
\end{split}
\end{equation}
Here, $c'$ is any real number greater than 1 and the implied constant in the error terms depends only on $t$.

\end{prop}
\begin{proof}
The equality in \eqref{this-lemma-1} is immediate.

Let us consider a tuple $\underline{m} = (m_1,m_2,m_3,\dots,m_t)$ such that $Y(\underline{m}) = a$ for some $0 \leq a \leq t-1$. 

If $a = 0$, that is, if  $m_{1},m_{2},\dots m_{t} \neq 0$, then by Proposition \ref{trace-estimate-later-calculation},
\begin{equation}\label{this-lemma-2}
\begin{split}
& \frac{1}{|\mathcal F_{N,k}|}\sum_{(m_1,m_2,m_3,\dots,m_t)}^{(0)}\sum_{(p_1,p_2,\dots ,p_t)}^{(1)}\sum_{f \in \mathcal F_{N,k}} a_f(p_1^{2m_1}p_2^{2m_2}\dots p_t^{2m_t})\\
&= \sum_{(m_1,m_2,m_3,\dots,m_t)}^{(0)}\sum_{(p_1,p_2,\dots ,p_t)}^{(1)} \left[\frac{1}{p_1^{m_1}p_2^{m_2}\dots p_t^{m_t}} + \O\left(\frac{\left(p_1^{2m_1}p_2^{2m_2}\dots p_t^{2m_t}\right)^{c'}8^{\nu(N)}}{k{N}}\right)\right]\\
& \ll_t (\log\log x)^t + \frac{\pi_N(x)^t x^{(2M_1 + 2M_2 + \dots + 2M_t)c'}8^{\nu(N)}}{k{N}}.
\end{split}
\end{equation}

Now, suppose $a >0$.  We denote $t' = t - a$.    Clearly, $1 \leq t' \leq t$. 

Without loss of generality, let us assume that $m_1 = m_2 = \dots = m_{t'} \neq 0$ and $m_{t'+1},m_{t'+2},\dots m_{t} = 0$.

Applying Proposition \ref{trace-estimate-later-calculation} and equation \eqref{this-lemma-1}, we obtain
\begin{equation*}
    \begin{split}
       & \frac{1}{|\mathcal F_{N,k}|}\sum_{(p_1,p_2,\dots ,p_t)}^{(1)} \sum_{f \in \mathcal F_{N,k}} a_f(p_1^{2m_1}p_2^{2m_2}\dots p_t^{2m_t})\\
    &  =  \frac{1}{|\mathcal F_{N,k}|}\sum_{(p_1,p_2,\dots ,p_t)}^{(1)}\sum_{f \in \mathcal F_{N,k}} a_f(p_1^{2m_1}p_2^{2m_2}\dots p_{t'}^{2m_{t'}})\\
    &  \ll_t \pi_N(x)^a
        \sum_{(p_1,p_2,\dots ,p_{t'})}^{(1)} \left ( 
     \frac{1}{p_1^{m_1}p_2^{m_2}\dots p_{t'}^{m_{t'}} } + O \left(\ \frac{8^{\nu(N)} p_{1}^{2m_{1}c'}p_{2}^{2m_{2}c'}\dots p_{t'}^{2m_{t'}c'} }{k {N}}\right)\right).\\
     \end{split}
     \end{equation*}

     Thus,
     \begin{equation}\label{this-lemma-3}
     \begin{split}
    & \frac{1}{|\mathcal F_{N,k}|}\sum_{(p_1,p_2,\dots ,p_t)}^{(1)}\sum_{(m_1,m_2,m_3,\dots,m_t)}^{(a)}\sum_{f \in \mathcal F_{N,k}} a_f(p_1^{2m_1}p_2^{2m_2}\dots p_t^{2m_t})\\ 
    &  \ll \pi_N(x)^{a}  \left(\sum_{p \leq x} 
     \sum_{m \geq 1}\frac{1}{p^m } \right)^{t'} +\pi_N(x)^{a} \sum_{(p_1,p_2,\dots ,p_{t'})}^{(1)}  \left (\frac{8^{\nu(N)}x^{(2M_1 + 2M_2 + \dots + 2M_{t'})c'}}{k{N}}\right)\\
   &  \ll \pi_N(x)^{a}  \left(\sum_{p \leq x} 
     \frac{1}{p } \right)^{t'} +\pi_N(x)^{a} \sum_{(p_1,p_2,\dots ,p_{t'})}^{(1)} \left (\frac{8^{\nu(N)}{N}x^{(2M_1 + 2M_2 + \dots + 2M_{t'})c'}}{kN}\right)\\
    & \ll \pi_N(x)^{a} \left(\log \log x \right)^{t'} + \pi_N(x)^{a} \sum_{(p_1,p_2,\dots ,p_{t'})}^{(1)}  \left (\frac{8^{\nu(N)} x^{(2M_1 + 2M_2 + \dots + 2M_{t'})c'}}{k{N}}\right)\\
      &  \ll \pi_N(x)^{a}(\log \log x)^{t-a} + \left(\frac{8^{\nu(N)}\pi_N(x)^{a+t '}x^{(2M_1 + 2M_2 + \dots + 2M_{t'})c'}}{k{N}}\right).
\end{split}
\end{equation}
Combining equations \eqref{this-lemma-2} and \eqref{this-lemma-3}, we get \eqref{this-lemma-other}.  This completes the proof of the proposition. 
\end{proof}

\begin{lemma}\label{x-powers}
Let us consider positive integers $k = k(x)$ and $N = N(x)$ such that
$$ \frac{\log \left(k{N}/8^{\nu(N)}\right)}{x} \to \infty \text{ as }x \to \infty.$$
Then, for any absolute constant $C >0$, 
$$x^{C\pi_N(x)} = \o\left(\frac{k{N}}{8^{\nu(N)}}\right) \text{ as }x \to \infty$$
\end{lemma}

\begin{proof}
If
$$ \frac{\log \left(k{N}/8^{\nu(N)}\right)}{x} \to \infty \text{ as }x \to \infty,$$
then for any absolute constant $C >0$,
$$C\pi_N(x)\log x \ll x = \o\left(\log \left(k{N}/8^{\nu(N)}\right)\right)\text{ as }x \to \infty.$$
Thus,
 $$x^{C\pi_N(x)} = \o\left(\frac{k{N}}{8^{\nu(N)}}\right) \text{ as }x \to \infty.$$
\end{proof}

\section{Higher moments of the pair correlation function $R_2(g,\rho)(f)$: proof of Theorem \ref{higher-moments-bounds}}\label{rM}

If Question \ref{question-shrinking} for $\delta = 1/AL$ has an affirmative answer, then for a fixed $f \in \mathcal F_{N,k}$, we have
$$R_2(g,\rho)(f)^r \sim \left(A^2\widehat{g}(0)\rho \ast \rho (0)\right)^r$$
for each $r \geq 1$.  Therefore, it would be natural  to study $R_2(g,\rho)(f)^r$ is for a random $f \in \mathcal F_{N,k}$.  Is
$$\frac{1}{| \mathcal F_{N,k}|} \sum_{f \in \mathcal F_{N,k}} R_2(g,\rho)(f)^r \sim \left(A^2\widehat{g}(0)\rho \ast \rho (0)\right)^r$$
for $r \geq 2$ under some conditions on $L,\,N$ and $k$?

By equation \eqref{pc-simplified}, 
\begin{equation}\label{sum-hm-0}
\begin{split}
& R_2(g,\rho)(f)^r \\
& = \left(\frac{1}{8 \pi_N(x)^2L}\sum_{(p,q) \atop {p \neq q \leq x \atop{(p,N) = (q,N) = 1}}} T_1(p)T_2(q)T_3(p,q)\right)^r\\
&= \frac{1}{(8 \pi_N(x)^2L)^r}\sum_{(p_1,q_1) \atop {p_1 \neq q_1 \leq x \atop{(p_1,N) = (q_1,N) = 1}}}\sum_{(p_2,q_2) \atop {p_2 \neq q_2 \leq x \atop{(p_2,N) = (q_2,N) = 1}}}\dots \sum_{(p_r,q_r) \atop {p_r \neq q_r \leq x \atop{(p_r,N) = (q_r,N) = 1}}} \prod_{i=1}^rT_1(p_i)T_2(q_i)T_3(p_i,q_i)\\
&= \frac{1}{(8 \pi_N(x)^2L)^r}\sum_{(p_1,q_1) \atop {p_1 \neq q_1 \leq x \atop{(p_1,N) = (q_1,N) = 1}}}\sum_{(p_2,q_2) \atop {p_2 \neq q_2 \leq x \atop{(p_2,N) = (q_2,N) = 1}}}\dots \sum_{(p_r,q_r) \atop {p_r \neq q_r \leq x \atop{(p_r,N) = (q_r,N) = 1}}}\\
&\sum_{l_1,l_2,\dots, l_r \atop {0 \leq l_i \leq L}}U(l_1)U(l_2)\dots U(l_r) a_f(p_1^{2l_1})a_f(p_2^{2l_2})\dots a_f(p_r^{2l_r})\\
&\sum_{l_1',l_2',\dots, l_r' \atop {0 \leq l_i' \leq L}}U(l_1')U(l_2')\dots U(l_r') a_f(q_1^{2l_1'})a_f(q_2^{2l_2'})\dots a_f(q_r^{2l_r'})\\
&\sum_{n_1,n_2,\dots, n_r \atop {0 \leq n_i \leq \pi_N(x)}}G(n_1)G(n_2)\dots G(n_r) A(p_1,q_1,n_1)\dots A(p_r,q_r,n_r).
\end{split}
\end{equation}
An evaluation of the above sum requires apportioning it into several parts based on the number of distinct primes among $p_1,p_2,\dots,p_r,q_1,q_2,\dots q_r$.  Each such part would need to be further classified, based on the number of $n_i$'s among the indices which are not zero.  Upon averaging over all $f \in \mathcal F_{N,k}$, the analysis of each component would involve estimates from the Eichler-Selberg trace formula for sums of the form
$$\sum_{f \in \mathcal F_{N,k}} \prod_{i=1}^{m'} a_f(p_i^{u_i}),$$
for suitable nonnegative integers $u_i$ which could be as large as constant multiple of $\pi_N(x)$.  
Taking forward equation \eqref{sum-hm-0}, we have
\begin{equation}\label{sum-hm}
\begin{split}
&\frac{1}{(8 \pi_N(x)^2L)^r}\sum_{(p_1,q_1) \atop {p_1 \neq q_1 \leq x \atop{(p_1,N) = (q_1,N) = 1}}}\sum_{(p_2,q_2) \atop {p_2 \neq q_2 \leq x \atop{(p_2,N) = (q_2,N) = 1}}}\dots \sum_{(p_r,q_r) \atop {p_r \neq q_r \leq x \atop{(p_r,N) = (q_r,N) = 1}}}\\
&\sum_{l_1,l_2,\dots, l_r \atop {0 \leq l_i \leq L}}U(l_1)U(l_2)\dots U(l_r) a_f(p_1^{2l_1})a_f(p_2^{2l_2})\dots a_f(p_r^{2l_r})\\
&\sum_{l_1',l_2',\dots, l_r' \atop {0 \leq l_i' \leq L}}U(l_1')U(l_2')\dots U(l_r') a_f(q_1^{2l_1'})a_f(q_2^{2l_2'})\dots a_f(q_r^{2l_r'})\\
&\sum_{n_1,n_2,\dots, n_r \atop {0 \leq n_i \leq \pi_N(x)}}G(n_1)G(n_2)\dots G(n_r) A(p_1,q_1,n_1)\dots A(p_r,q_r,n_r)\\
&= \frac{1}{(8 \pi_N(x)^2L)^r}\sum_{(p_1,q_1) \atop {p_1 \neq q_1 \leq x \atop{(p_1,N) = (q_1,N) = 1}}}\sum_{(p_2,q_2) \atop {p_2 \neq q_2 \leq x \atop{(p_2,N) = (q_2,N) = 1}}}\dots \sum_{(p_r,q_r) \atop {p_r \neq q_r \leq x \atop{(p_r,N) = (q_r,N) = 1}}}\\
& T_{1,r}(\underline{p,q}) + T_{2,r}(\underline{p,q}) + T_{3,r}(\underline{p,q}),\\
\end{split}
\end{equation}
where
\begin{equation}\label{T1}
\begin{split}
&T_{1,r}(\underline{p,q}) = (4(G(0))^r\sum_{l_1,l_2,\dots, l_r \atop {0 \leq l_i \leq L}}U(l_1)U(l_2)\dots U(l_r) \sum_{l_1',l_2',\dots, l_r' \atop {0 \leq l_i' \leq L}}U(l_1')U(l_2')\dots U(l_r') \\
&a_f(p_1^{2l_1})a_f(p_2^{2l_2})\dots a_f(p_r^{2l_r})a_f(q_1^{2l_1'})a_f(q_2^{2l_2'})\dots a_f(q_r^{2l_r'}),\\
\end{split}
\end{equation}
\begin{equation}\label{T2}
\begin{split}
&T_{2,r}(\underline{p,q}) = \sum_{j=1}^{r-1} (4(G(0))^{r - j}\binom{r}{j} 2^j \sum_{n_1,n_2,\dots, n_j \atop {1 \leq n_i \leq \pi_N(x)}}G(n_1)G(n_2)\dots G(n_j)\\
&\sum_{l_1,l_2,\dots, l_r \atop {0 \leq l_i \leq L}}U(l_1)U(l_2)\dots U(l_r) \sum_{l_1',l_2',\dots, l_r' \atop {0 \leq l_i' \leq L}}U(l_1')U(l_2')\dots U(l_r') \\
&\prod_{i=1}^j I(p_i,q_i,n_i,l_i,l_i')\prod_{i = j+1}^ r a_f(p_i^{2l_i}) a_f(q_i^{2l_i'}),\\
\end{split}
\end{equation}
and
\begin{equation}\label{T3}
\begin{split}
&T_{3,r}(\underline{p,q}) = 2^r  \sum_{n_1,n_2,\dots, n_r \atop {1 \leq n_i \leq \pi_N(x)}}G(n_1)G(n_2)\dots G(n_r) \\
&\sum_{l_1,l_2,\dots, l_r \atop {0 \leq l_i \leq L}}U(l_1)U(l_2)\dots U(l_r) \sum_{l_1',l_2',\dots, l_r' \atop {0 \leq l_i' \leq L}}U(l_1')U(l_2')\dots U(l_r') \prod_{i=1}^r I(p_i,q_i,n_i,l_i,l_i'),
\end{split}
\end{equation}
with
\begin{equation}\label{l-rep*}
\begin{split}
&I(p_i,q_i,n_i,l_i,l_i')\\
& = a_f(p_i^{2l_i})a_f(q_i^{2l_i'})(a_f(p_i^{2n_i}) - a_f(p_i^{2n_i-2}))(a_f(q_i^{2n_i}) - a_f(q_i^{2n_i-2}))\\
\end{split}
\end{equation}
By Lemma \ref{Hecke-multiplicative-diff}, we see that
\begin{equation}\label{l-rep}
\begin{split}
&I(p_i,q_i,n_i,l_i,l_i')\\
& = a_f(p_i^{2l_i})a_f(q_i^{2l_i'})(a_f(p_i^{2n_i}) - a_f(p_i^{2n_i-2}))(a_f(q_i^{2n_i}) - a_f(q_i^{2n_i-2}))\\
&= \begin{cases}
\left(a_f(p_i^{2l_i+2n_i}) + a_f(p_i^{2l_i-2n_i})\right)\left(a_f(q_i^{2l'_i+2n_i}) + a_f(q_i^{2l_i'-2n_i})\right) &\text{ if }l_i,l_i' \geq n_i\\
\left(a_f(p_i^{2l_i+2n_i}) - a_f(p^{2n_i-2l_i-2})\right)\left(a_f(q_i^{2l_i'+2n_i}) + a_f(q_i^{2l_i'-2n_i})\right) &\text{ if }l_i<n_i \leq l_i'\\
\left(a_f(p_i^{2l_i+2n_i}) + a_f(p_i^{2l_i-2n_i})\right)\left(a_f(q_i^{2l_i'+2n_i}) - a_f(q_i^{2n_i-2l_i'-2})\right) &\text{ if }l_i'<n_i \leq l_i\\
\left(a_f(p_i^{2l_i+2n_i}) - a_f(p_i^{2n_i-2l_i-2})\right)\left(a_f(q_i^{2l_i'+2n_i}) - a_f(q_i^{2n_i-2l_i'-2})\right) &\text{ if }l_i,l_i' < n_i.
\end{cases}
\end{split}
\end{equation}

We evaluate $\langle T_{1,r}(\underline{p,q}) \rangle$, $\langle T_{2,r}(\underline{p,q}) \rangle$ and $\langle T_{3,r}(\underline{p,q}) \rangle$ in Sections \ref{rM-1}, \ref{rM-2} and \ref{rM-3} respectively.

\subsection{$\langle T_{1,r}(\underline{p,q}) \rangle$}\label{rM-1}

In this section, we evaluate
$$\left \langle \sum_{(p_1,q_1) \atop {p_1 \neq q_1 \leq x \atop{(p_1,N) = (q_1,N) = 1}}}\sum_{(p_2,q_2) \atop {p_2 \neq q_2 \leq x \atop{(p_2,N) = (q_2,N) = 1}}}\dots \sum_{(p_r,q_r) \atop {p_r \neq q_r \leq x \atop{(p_r,N) = (q_r,N) = 1}}} T_{1,r}(\underline{p,q}) \right \rangle.$$
By Proposition \ref{sums-distinct-primes}, we derive the following propositions.
\begin{prop}\label{2r-distinct-primes}
Let $\sum_{\underline{(p,q)}}^{(2r)}$ denote a sum over all $2r$-tuples of distinct primes $p_1,q_1,p_2,q_2,\dots,p_r,q_r  \leq x$ coprime to $N$.  We consider $L = L(x) \to \infty$ such that 
$$L(x) < \frac{\pi_N(x)}{\log \log x}.$$
Then,
\begin{equation*}
\begin{split}
& \frac{1}{(8 \pi_N(x)^2L)^r} \sum_{\underline{(p,q)}}^{(2r)}\sum_{l_1,l_2,\dots, l_r \atop {0 \leq l_i \leq L}}U(l_1)U(l_2)\dots U(l_r) \sum_{l_1',l_2',\dots, l_r' \atop {0 \leq l_i' \leq L}}U(l_1')U(l_2')\dots U(l_r') \\
&\frac{1}{|\mathcal F_{N,k}|} \sum_{f \in \mathcal F_{N,k}}a_f(p_1^{2l_1})a_f(p_2^{2l_2})\dots a_f(p_r^{2l_r})a_f(q_1^{2l_1'})a_f(q_2^{2l_2'})\dots a_f(q_r^{2l_r'})\\
& \ll \frac{1}{L^r}  + \frac{x^{4rLc'}8^{\nu(N)}}{kN}.
\end{split}
\end{equation*}
\end{prop}
\begin{proof}
In a $2r$-tuple $\underline{\ell} = (l_1,l_2,\dots,l_r,l_1',l_2',\dots,l_r')$, let $Y(\underline{\ell})$ denote the number of components of the tuple which are equal to 0.  Clearly, $0 \leq Y(\underline{\ell}) \leq 2r$.  Also, $|U(l_i)|,\,|U(l_i')| \ll 1.$

By Proposition \ref{sums-distinct-primes}, we have
\begin{equation*}
\begin{split}
& \frac{1}{(8 \pi_N(x)^2L)^r} \sum_{\underline{(p,q)}}^{(2r)}\sum_{l_1,l_2,\dots, l_r \atop {0 \leq l_i \leq L}}U(l_1)U(l_2)\dots U(l_r) \sum_{l_1',l_2',\dots, l_r' \atop {0 \leq l_i' \leq L}}U(l_1')U(l_2')\dots U(l_r') \\
&\frac{1}{|\mathcal F_{N,k}|} \sum_{f \in \mathcal F_{N,k}}a_f(p_1^{2l_1})a_f(p_2^{2l_2})\dots a_f(p_r^{2l_r})a_f(q_1^{2l_1'})a_f(q_2^{2l_2'})\dots a_f(q_r^{2l_r'})\\
& \ll \frac{1}{(\pi_N(x)^2L)^r}\sum_{\underline{(p,q)}}^{(2r)} \sum_{a = 0}^{2r}\sum_{\underline{\ell}\atop {0 \leq l_i,l_i' \leq L \atop {Y(\underline{\ell}) = a}}}|U(l_1)U(l_2)\dots U(l_r)U(l_1')U(l_2')\dots U(l_r')| \\
&\times \left|\frac{1}{|\mathcal F_{N,k}|} \sum_{f \in \mathcal F_{N,k}}a_f(p_1^{2l_1})a_f(p_2^{2l_2})\dots a_f(p_r^{2l_r})a_f(q_1^{2l_1'})a_f(q_2^{2l_2'})\dots a_f(q_r^{2l_r'})\right|\\
&\ll \frac{1}{(\pi_N(x)^2L)^r}\sum_{\underline{(p,q)}}^{(2r)} \sum_{a = 0}^{2r}\sum_{\underline{\ell}\atop {0 \leq l_i,l_i' \leq L \atop {Y(\underline{\ell}) = a}}}\left|\frac{1}{|\mathcal F_{N,k}|} \sum_{f \in \mathcal F_{N,k}}a_f(p_1^{2l_1})a_f(p_2^{2l_2})\dots a_f(p_r^{2l_r})a_f(q_1^{2l_1'})a_f(q_2^{2l_2'})\dots a_f(q_r^{2l_r'})\right|\\
&\ll  \frac{1}{(\pi_N(x)^2L)^r} \sum_{a = 0}^{2r}\pi_N(x)^a L^{2r-a}(\log \log x)^{2r- a} + \frac{x^{4rLc'}8^{\nu(N)}}{k N}\\
&\ll  \frac{1}{(\pi_N(x)^2L)^r} (L \log \log x)^{2r} \sum_{a = 0}^{2r} \left(\frac{\pi_N(x)}{L\log \log x}\right)^a + \frac{x^{4rLc'}8^{\nu(N)}}{k N} \\
&\ll \frac{1}{L^r} + \frac{x^{4rLc'}8^{\nu(N)}}{k N}, \,\,\,\left(\text{ since }\frac{\pi_N(x)}{L\log \log x} > 1\right).
\end{split}
\end{equation*}

\end{proof}

\begin{remark}
Note that the implied constants in each of the above inequalities depend on $r$.  This observation also applies to the various estimates that follow below.  We will use the notations $\ll$, $\ll_r$, $\O(.)$ and $\O_r(.)$ interchangeably in the rest of this article.
\end{remark}

\begin{notation}\label{notation-t-distinct-primes}
We use the following notation to study the terms in $T_{1,r}(\underline{p,q})$ such that the primes in the $2r$-tuples $(p_1,p_2,\dots,p_r,q_1,q_2,\dots, q_r)$ are not all distinct. 

\begin{enumerate}
\item[{\bf (a)}] Suppose that the number of distinct primes in a tuple $(p_1,p_2,\dots,p_r,q_1,q_2,\dots, q_r)$ is $t$.  Then, $2 \leq t \leq 2r - 1$.  Let us denote the distinct primes in the above tuple as $s_1,s_2,\dots, s_t$.  
\item[{\bf (b)}] We write
\begin{equation}\label{t-distinct-primes}
a_f(p_1^{2l_1})a_f(p_2^{2l_2})\dots a_f(p_r^{2l_r})a_f(q_1^{2l_1'})a_f(q_2^{2l_2'})\dots a_f(q_r^{2l_r'}) =  \prod_{u=1}^ t \left(\prod_{i_u \in \mathcal I(s_u)} a_f(s_u^{2i_u})\right),
\end{equation}
where, for each $1 \leq u \leq t$, 
$$ \mathcal I(s_u) \subset \{l_1,l_2,\dots ,l_r,l_1',l_2',\dots,l_r'\}$$
consists of all elements which appear as exponents of $s_u$ in \eqref{t-distinct-primes}.
Note that 
\begin{itemize}
\item the sets $\mathcal I(s_u)$ are mutually disjoint, 
\item $\cup_{u=1}^t \mathcal I(s_u) = \{l_1,l_2,\dots ,l_r,l_1',l_2',\dots,l_r'\},$
\item For each $1 \leq u \leq t$, $|\mathcal I(s_u)| \leq r$ (since for any $i$, $l_i$ and $l_i'$ cannot be elements of the same set $\mathcal I(s_u)$), and 
\item $\sum_{u=1}^t |\mathcal I(s_u)| = 2r.$
\end{itemize}

\item[{\bf (c)}] Let $\mathcal J_{\mathcal I(s_u)}$ denote the set of all integers $b_u$, such that $2b_u$ appears in the power of $s_u$ when
$$\prod_{i_u \in \mathcal I(s_u)} a_f(s_u^{2i_u})$$
is expanded into a sum.  That is,
$$\prod_{i_u \in \mathcal I(s_u)} a_f(s_u^{2i_u}) = \sum_{b_u \in \mathcal J_{\mathcal I(s_u)}}D_{\mathcal I(s_u)}(b_u)a_f(s_u^{2b_u}).$$
\item[{\bf (d)}] For each $u$, denote $v(u)$ to be the sum of all the elements of $\mathcal I(s_u)$.  Thus, $0 \leq b_u \leq v(u)$.
\end{enumerate}
\end{notation}
We need estimates for $D(b_u)$ for each $u$.  These are obtained with the help of the following lemma.

\begin{lemma}\label{D(m)}
Let
$$\mathcal I(s_u) \subset \{l_1,l_2,\dots ,l_r,l_1',l_2',\dots,l_r'\},\,0 \leq l_i,l_i' \leq L.$$ 
 Let $\mathcal J_{\mathcal I(s_u)}$ denote the set of all integers $b_u$, such that $2b_u$ appears in the power of $s_u$ when
$$\prod_{i_u \in \mathcal I(s_u)} a_f(s_u^{2i_u})$$
is expanded into a sum.  That is,
$$\prod_{i_u \in \mathcal I(s_u)} a_f(s_u^{2i_u}) = \sum_{b_u \in \mathcal J_{\mathcal I(s_u)}}D(b_u)a_f(s_u^{2b_u}).$$
For each $b_u \in \mathcal J_{\mathcal I(s_u)}$,
$$D(b_u) = \begin{cases}
\O(1) &\text{ if }|\mathcal I(s_u)| = 1,\\
\O(L^{|\mathcal I(s_u)| - 2})&\text{ if }|\mathcal I(s_u)| \geq 2,\\
\O(L^{|\mathcal I(s_u)| - 3})&\text{ if }| \mathcal I(s_u)| \geq 3,\,b_u = 0.\\
\end{cases}$$
\end{lemma}
\begin{proof}
The result is immediate for $|\mathcal I(s_u)| = 1$.

Let $|\mathcal I(s_u)| = 2$ and let $\mathcal I(s_u) = \{l_{u_1},l_{u_2}\}$.  Then,
$$a_f(s_u^{2l_{u_1}})a_f(s_u^{2l_{u_2}}) = \sum_{j=0}^{\min\{2l_{u_1},2l_{u_1}\}} a_f(s_u^{2l_{u_1} + 2l_{u_2} - 2j}).$$
Thus, in this case, $D(b_u) = 0$ or 1.  We now proceed by induction.  

Assume that for some $k \geq 2$, $D(b_u) = \O(L^{k - 2})$ if $|\mathcal I(s_u)| = k$.

We now consider the expansion
$$a_f(s_u^{2l_{u_{k+1}}}) \left(a_f(s_u^{2l_{u_{1}}})a_f(s_u^{2l_{u_{2}}})\dots a_f(s_u^{2l_{u_{k}}})\right) = a_f(s_u^{2l_{u_{k+1}}})\sum_{t \in \mathcal J_{\{l_{u_{1}},l_{u_{2}},\dots, l_{u_{k}}\}}} D(t) a_f(p^{2t}).$$
By induction hypothesis for the sum
$$\sum_{t \in \mathcal J_{\{l_{u_{1}},l_{u_{2}},\dots, l_{u_{k}}\}}} D(t) a_f(p^{2t}),$$
we get that $D_{\{l_{u_{1}},l_{u_{2}},\dots, l_{u_{k}}\}}(t) = \O(L^{k - 2})$.

From the above equation, we also note that $b_u \in \mathcal J_{\{l_{u_{1}},l_{u_{2}},\dots, l_{u_{k+1}}\}}$ can occur at most once in each of the expansions
$$a_f(s_u^{2l_{u_{k+1}}})a_f(s_u^{2t}),\,t \in \mathcal J_{\{l_{u_{1}},l_{u_{2}},\dots, l_{u_{k}}\}}.$$
Thus, for $b_u \in \mathcal J_{\{l_{u_{1}},l_{u_{2}},\dots, l_{u_{k+1}}\}}$,
$$D_{\{l_{u_{1}},l_{u_{2}},\dots, l_{u_{k}}, l_{u_{k+1}}\}}(b_u) \ll D_{\{l_{u_{1}},l_{u_{2}},\dots, l_{u_{k}}\}}(t)\mathcal J_{\{l_{u_{1}},l_{u_{2}},\dots, l_{u_{k}}\}} \ll L^{k - 2}(\sum_{j=1}^k 2l_{u_j}) \ll_k L^{k-1},$$
since each $l_{u_j} \leq L$.

By induction, we have proved that
$$D_{\mathcal I(s_u)}(b_u) = \O(L^{|\mathcal I(s_u)| - 2}).$$
We now have to prove that $D_{\mathcal I(s_u)}(0) = \O(L^{|\mathcal I(s_u)| - 3})$ if $|\mathcal I(s_u)| \geq 3$.

Note that if $|\mathcal I(s_u)| = 3$ and $\mathcal I(s_u) = \{l_{u_1},l_{u_2},l_{u_3}\}$,  then for a prime $s$, $a_f(s^0)$ occurs in the expansion
$$a_f(s^{2l_{u_1}})a_f(s^{2l_{u_2}})a_f(s^{2l_{u_3}}) = a_f(s^{2l_{u_3}})\left[\sum_{j=0}^{\min\{2l_{u_1},2l_{u_1}\}} a_f(s_u^{2l_{u_1} + 2l_{u_2} - 2j})\right]$$
if and only if $2l_{u_3} = 2l_{u_1} + 2l_{u_2} - 2j$ for some $0 \leq j \leq \min\{2l_{u_1},2l_{u_2}\}$.  Since this can happen for at most one $j$, and the coefficient of $a_f(s^0)$ in $a_f(s^{2l_{u_3}})a_f(s^{2l_{u_3}})$ is 1, we get that $D_{\mathcal I(s_u)}(0) \leq 1$.

Thus, $D_{\mathcal I(s_u)}(0) = \O(L^{|\mathcal I(s_u)| - 3})$ if $|\mathcal I(s_u)| = 3$.  Let us now assume that for some $k \geq 3$, $$D_{\mathcal I(s_u)}(0) = \O(L^{|\mathcal I(s_u)| - 3}) \text{ if }|\mathcal I(s_u)| = k.$$  That is, in the sum 
$$\sum_{t \in \mathcal J_{\{l_{u_{1}},l_{u_{2}},\dots, l_{u_{k}}\}}} D(t) a_f(s^{2t}),$$
$D(0) = \O(L^{k - 3})$.

Now, 
\begin{equation*}
\begin{split}
&a_f(s^{2l_{u_{k+1}}}) \left(a_f(s^{2l_{u_{1}}})a_f(s^{2l_{u_{2}}})\dots a_f(s^{2l_{u_{k}}})\right) \\
&= a_f(s^{2l_{u_{k+1}}}) \left(\sum_{t \in \mathcal J_{\{l_{u_{1}},l_{u_{2}},\dots, l_{u_{k}}\}}} D_{\{l_{u_{1}},l_{u_{2}},\dots, l_{u_{k}}\}}(t) a_f(s^{2t})\right)\\
&= \sum_{t \in \mathcal J_{\{l_{u_{1}},l_{u_{2}},\dots, l_{u_{k}}\}}}  \sum_{i =0 }^{\min \{2t,2l_{u_{k+1}}\}} D_{\{l_{u_{1}},l_{u_{2}},\dots, l_{u_{k}}\}}(t) a_f(s^{2t + 2l_{u_{k+1}} - 2i}).\\
\end{split}
\end{equation*}
Note that $2t + 2l_{u_{k+1}} - 2i = 0$ if and only if $i = 2t = 2l_{u_{k+1}}$.  Therefore, by the previous part of the lemma, the coefficient of $a_f(s^0)$ in the expansion of
$$a_f(s^{2l_{u_{k+1}}}) \left(a_f(s^{2l_{u_{1}}})a_f(s^{2l_{u_{2}}})\dots a_f(s^{2l_{u_{k}}})\right)$$
is
$$ \leq D_{\{l_{u_{1}},l_{u_{2}},\dots, l_{u_{k}}\}}(l_{u_{k+1}}) \ll_k L^{k-2}.$$
By induction, we prove the lemma completely.

\end{proof}
\begin{remark} Henceforth, if the context of the underlying set is clear, then we may simply indicate the coefficient $D_{\mathcal I(s_u)}(b_u)$ as $D(b_u)$.
\end{remark}
Thus,
\begin{equation}\label{expression-T1-term}
\begin{split}
&\frac{1}{(8 \pi_N(x)^2L)^r} (4(G(0))^r\sum_{l_1,l_2,\dots, l_r \atop {0 \leq l_i \leq L}}U(l_1)U(l_2)\dots U(l_r) \sum_{l_1',l_2',\dots, l_r' \atop {0 \leq l_i' \leq L}}U(l_1')U(l_2')\dots U(l_r')\\
&\quad\quad\quad\quad\quad\quad \sum_{(s_1,s_2,\dots ,s_t)} \left(\prod_{i_u \in \mathcal I(s_u)} a_f(s_u^{2i_u})\right)\\
&=\frac{1}{(8 \pi_N(x)^2L)^r} (4(G(0))^r\sum_{l_1,l_2,\dots, l_r \atop {0 \leq l_i \leq L}}U(l_1)U(l_2)\dots U(l_r) \sum_{l_1',l_2',\dots, l_r' \atop {0 \leq l_i' \leq L}}U(l_1')U(l_2')\dots U(l_r')\\
&\quad\quad\quad\quad\quad\quad \sum_{(s_1,s_2,\dots ,s_t)} \prod_{u=1}^ t \left(\sum_{b_u \in \mathcal J_{\mathcal I(s_u)}} D(b_u) a_f(s_u^{2b_u})\right).
\end{split}
\end{equation}
We now prove the following proposition:
\begin{prop}\label{T1rpq}
We consider $L = L(x) \to \infty$ such that 
$$L(x) < \frac{\pi_N(x)}{\log \log x}.$$
\begin{enumerate}
\item[{\bf(a)}] Let us choose $t$ such that $2 \leq t \leq 2r-1$, and let us consider a set partition
\begin{equation}\label{set-partition}
\{l_1,l_2,\dots,l_r,l_1',l_2',\dots,l_r'\} = \cup_{u=1}^t \mathcal I(s_u),
\end{equation}
into disjoint subsets $\mathcal I(s_u)$, such that \begin{itemize}
\item for each $1 \leq u \leq t$, $|\mathcal I(s_u)| \leq r$ (since for any $i$, $l_i$ and $l_i'$ cannot be elements of the same set $\mathcal I(s_u)$), and 
\item $\sum_{u=1}^t |\mathcal I(s_u)| = 2r.$
\end{itemize}

We write
$$\prod_{i_u \in \mathcal I(s_u)} a_f(s_u^{2i_u}) = \sum_{b_u \in \mathcal J_{\mathcal I(s_u)}}D(b_u)a_f(s_u^{2b_u}).$$
Then, 
\begin{equation}\label{expression-T1-more}
\begin{split}
&\frac{1}{(8 \pi_N(x)^2L)^r} (4(G(0))^r\sum_{l_1,l_2,\dots, l_r \atop {0 \leq l_i \leq L}}U(l_1)U(l_2)\dots U(l_r) \sum_{l_1',l_2',\dots, l_r' \atop {0 \leq l_i' \leq L}}U(l_1')U(l_2')\dots U(l_r')\\
&\sum_{(s_1,s_2,\dots ,s_t)} \frac{1}{|\mathcal F_{N,k}|}\sum_{f \in \mathcal F_{N,k}} \prod_{u=1}^t \left(\prod_{i_u \in \mathcal I(s_u)} a_f(s_u^{2i_u})\right)\\
&\ll_r \frac{L^{3r - 2t}}{\pi_N(x)^{2r - t}} + \frac{x^{4Lrc'} 8^{\nu(N)}}{k N} L^{3r}\\
\end{split}
\end{equation}

\item[{\bf(b)}] With $T_{1,r}(\underline{p,q})$ defined as in equation \eqref{T1}, we have
\begin{equation*}\label{T1-together}
\begin{split}
&\frac{1}{(8 \pi_N(x)^2L)^r}\sum_{(p_1,q_1) \atop {p_1 \neq q_1 \leq x \atop{(p_1,N) = (q_1,N) = 1}}}\sum_{(p_2,q_2) \atop {p_2 \neq q_2 \leq x \atop{(p_2,N) = (q_2,N) = 1}}}\dots \sum_{(p_r,q_r) \atop {p_r \neq q_r \leq x \atop{(p_r,N) = (q_r,N) = 1}}} T_{1,r}(\underline{p,q})\\
& \ll_r \sum_{t=2}^{2r}  \frac{L^{3r - 2t}}{\pi_N(x)^{2r - t}} + \frac{x^{4Lrc'} 8^{\nu(N)}}{k N} L^{3r}.
\end{split}
\end{equation*}

\end{enumerate}

\end{prop}
\begin{proof}
\begin{enumerate}
\item[{\bf(a)}] By equation \eqref{expression-T1-term}, we have to evaluate, for each set partition as indicated in \eqref{set-partition}, the sum
\begin{equation}\label{expression-T1-more}
\begin{split}
&\frac{1}{(8 \pi_N(x)^2L)^r} (4(G(0))^r\sum_{l_1,l_2,\dots, l_r \atop {0 \leq l_i \leq L}}U(l_1)U(l_2)\dots U(l_r) \sum_{l_1',l_2',\dots, l_r' \atop {0 \leq l_i' \leq L}}U(l_1')U(l_2')\dots U(l_r')\\
&\sum_{(s_1,s_2,\dots ,s_t)}\sum_{(b_1,b_2,\dots,b_t) \atop {b_u \in \mathcal J_{\mathcal I(s_u)}}}D(b_1)D(b_2)\dots D(b_t) \frac{1}{|\mathcal F_{N,k}|}\sum_{f \in \mathcal F_{N,k}} a_f(s_1^{2b_1})a_f(s_2^{2b_2})\dots a_f(s_t^{2b_t})\\
\end{split}
\end{equation}
By Proposition \ref{trace-estimate-later-calculation},
\begin{equation}\label{T1-further-t}
\begin{split}
&\sum_{l_1,l_2,\dots, l_r \atop {0 \leq l_i \leq L}}U(l_1)U(l_2)\dots U(l_r) \sum_{l_1',l_2',\dots, l_r' \atop {0 \leq l_i' \leq L}}U(l_1')U(l_2')\dots U(l_r')\\
&\sum_{(s_1,s_2,\dots ,s_{t})}\sum_{(b_1,b_2,\dots,b_{t}) \atop {b_u \in \mathcal J_{\mathcal I(s_u)}}}D(b_1)D(b_2)\dots D(b_{t}) \frac{1}{|\mathcal F_{N,k}|}\sum_{f \in \mathcal F_{N,k}} a_f(s_1^{2b_1})a_f(s_2^{2b_2})\dots a_f(s_{t}^{2b_{t}})\\
&= \sum_{l_1,l_2,\dots, l_r \atop {0 \leq l_i \leq L}}U(l_1)U(l_2)\dots U(l_r) \sum_{l_1',l_2',\dots, l_r' \atop {0 \leq l_i' \leq L}}U(l_1')U(l_2')\dots U(l_r')\\
&\sum_{(b_1,b_2,\dots,b_{t}) \atop {b_u \in \mathcal J_{\mathcal I(s_u)}}}D(b_1)D(b_2)\dots D(b_{t}) \\
&\times \sum_{(s_1,s_2,\dots ,s_{t})}\begin{cases} 1 &\text{ if }b_1 = b_2 = \dots = b_{t} = 0\\
\frac{1}{s_1^{b_1}s_2^{b_2}\dots s_{t}^{b_{t}}}+ \O\left(\frac{\left(s_1^{2b_1}s_2^{2b_2}\dots s_{t}^{2b_{t}}\right)^{c'} 8^{\nu(N)}}{k{N}}\right)&\text{ if }b_u \geq 1 \text{ for some }u.\\
\end{cases}
 \end{split}
\end{equation}
We first estimate
\begin{equation}\label{T1-further-t-error*}
\begin{split}
&\sum_{l_1,l_2,\dots, l_r \atop {0 \leq l_i \leq L}}\left|U(l_1)U(l_2)\dots U(l_r)\right| \sum_{l_1',l_2',\dots, l_r' \atop {0 \leq l_i' \leq L}}\left|U(l_1')U(l_2')\dots U(l_r')\right|\\
& \sum_{(b_1,b_2,\dots,b_{t}) \atop {b_u \in \mathcal J_{\mathcal I(s_u)} \atop {b_u \geq 1 \text{ for some }u} }} D(b_1)D(b_2)\dots D(b_{t})\sum_{(s_1,s_2,\dots ,s_{t})}
\frac{\left(s_1^{2b_1}s_2^{2b_2}\dots s_{t}^{2b_{t}}\right)^{c'} 8^{\nu(N)}}{k{N}}\\
\end{split}
\end{equation}

By Lemma \ref{D(m)},
$$D(b_1)D(b_2)\dots D(b_{t}) \ll_r L^{\sum_{u=1}^t |\mathcal I(s_u)|} = L^{2r}.$$ 
We recall that for each $u$, $v(u)$ denotes the sum of the elements of $\mathcal I(s_u)$.  Thus, $b_u \leq v(u)$ for each $u$.  Thus, 
\begin{equation*}
\begin{split}
& \sum_{(b_1,b_2,\dots,b_{t}) \atop {b_u \in \mathcal J_{\mathcal I(s_u)} \atop {b_u \geq 1 \text{ for some }u} }} D(b_1)D(b_2)\dots D(b_{t})\sum_{(s_1,s_2,\dots ,s_{t})}\left(s_1^{2b_1}s_2^{2b_2}\dots s_{t}^{2b_{t}}\right)^{c'}\\
& \ll_r L^{2r} \prod_{u=1}^t \sum_{s_u \leq x}\left(\sum_{b_u =0}^{v(u)} \left(s_u^{c'}\right)^{2b_u}\right)\\
& \leq L^{2r} (\pi_N(x))^t x^{2c'\sum_{u=1}^t v(u)}\\
&\leq L^{2r} (\pi_N(x))^t x^{4c'rL} \quad\quad\quad\quad\text{ since }\sum_u v(u) \leq 2rL.
\end{split}
\end{equation*}
Therefore, we get that the expression in \eqref{T1-further-t-error*} is
\begin{equation}\label{T1-further-t-error}
\begin{split}
&\sum_{l_1,l_2,\dots, l_r \atop {0 \leq l_i \leq L}}\left|U(l_1)U(l_2)\dots U(l_r)\right| \sum_{l_1',l_2',\dots, l_r' \atop {0 \leq l_i' \leq L}}\left|U(l_1')U(l_2')\dots U(l_r')\right|\\
& \sum_{(b_1,b_2,\dots,b_{t}) \atop {b_u \in \mathcal J_{\mathcal I(s_u)} \atop {b_u \geq 1 \text{ for some }u} }} D(b_1)D(b_2)\dots D(b_{t})\sum_{(s_1,s_2,\dots ,s_{t})}
\frac{\left(s_1^{2b_1}s_2^{2b_2}\dots s_{t}^{2b_{t}}\right)^{c'} 8^{\nu(N)}}{k{N}}\\
& \ll_r L^{4r}\pi_N(x)^t x^{4c'rL}\frac{8^{\nu(N)}}{k{N}}.
\end{split}
\end{equation}
We now consider the expression
\begin{equation}\label{T1-further-t-main}
\begin{split}
&\sum_{l_1,l_2,\dots, l_r \atop {0 \leq l_i \leq L}}U(l_1)U(l_2)\dots U(l_r) \sum_{l_1',l_2',\dots, l_r' \atop {0 \leq l_i' \leq L}}U(l_1')U(l_2')\dots U(l_r')\\
&\sum_{(s_1,s_2,\dots ,s_{t})}\sum_{(b_1,b_2,\dots,b_{t}) \atop {b_u \in \mathcal J_{\mathcal I(s_u)}}}\frac{D(b_1)D(b_2)\dots D(b_{t})}{s_1^{b_1}s_2^{b_2}\dots s_{t}^{b_{t}}}.
\end{split}
\end{equation}
Estimating the above expression involves an interplay between
$$(l_1,l_2,\dots, l_r,l_1',l_2',\dots, l_r')$$
and the number of $b_u$'s in a tuple $(b_1,b_2,\dots,b_t)$
which are equal to 0.  
Let 
$$x' := \#\{1 \leq u \leq t:\,|\mathcal I(s_u)| = 1\},$$
and
$$y := \#\{1 \leq u \leq t:\,|\mathcal I(s_u)| = 2\}.$$
Then,
$$x' + 2y +\sum_{u = 1 \atop {|\mathcal I(s_u)| \geq 3}}^t  |\mathcal I(s_u)| = 2r.$$

Also, by Lemma \ref{D(m)}, 
$$D(b_u) \ll_r L^{| \mathcal I(s_u)| - 2} \text{ if }|\mathcal I(s_u)| \geq 3, b_u \neq 0,$$
$$D(0) \ll_r L^{| \mathcal I(s_u)| - 3} \text{ if }|\mathcal I(s_u)| \geq 3,$$
and 
$$D(b_u) = \O(1)\text{ if }|\mathcal I(s_u)| = 1\text{ or }2.$$

We partition the expression in \eqref{T1-further-t-main} into sums based on the number 
$$a(\underline{b}) := \#\{1 \leq u \leq t:\,b_u = 0\}.$$
We denote, for a $t$-tuple $\underline{b}$,
$$a_1(\underline{b}) :=   \#\{1 \leq u \leq t:\,|\mathcal I(s_u)| = 1,\,b_u = 0\},$$
$$a_2(\underline{b}) :=   \#\{1 \leq u \leq t:\,|\mathcal I(s_u)| = 2,\,b_u = 0\},$$
and
$$a_3(\underline{b}) =   \#\{1 \leq u \leq t:\,|\mathcal I(s_u)| \geq 3,\,b_u = 0\}.$$
Note that
$$\prod_{u=1}^t D(b_u) \ll_r \prod_{u=1 \atop {|\mathcal I(s_u)| \geq 3}}^t |D(b_u)| \ll_r \prod_{u=1 \atop {|\mathcal I(s_u)| \geq 3}}^t \begin{cases} L^{| \mathcal I(s_u)| - 3}  &\text{ if }b_u = 0\\
L^{| \mathcal I(s_u)| - 2}  &\text{ if }b_u \neq 0 
\end{cases}$$
We observe that 
$$\prod_{u=1 \atop {|\mathcal I(s_u)| \geq 3 \atop {b_u = 0}}}^t L^{-3} = L^{-3a_3(\underline{b})}.$$
Also,
\begin{equation*}
\sum_{u=1 \atop {|\mathcal I(s_u)| \geq 3 \atop {b_u \neq 0}}}^t 1 =  \#\{1 \leq u \leq t:\,|\mathcal I(s_u)| \geq 3\} - a_3(\underline{b}) = t - x'-y - a_3(\underline{b}).
\end{equation*}
Thus, 
\begin{equation*}
\prod_{u=1 \atop {|\mathcal I(s_u)| \geq 3 \atop {b_u \neq 0}}}^t L^{-2} = L^{-2(t - x'-y - a_3(\underline{b}))}.
\end{equation*}
Therefore,
\begin{equation*}
\begin{split}
&\prod_{u=1 \atop {|\mathcal I(s_u)| \geq 3}}^t \begin{cases} L^{| \mathcal I(s_u)| - 3}  &\text{ if }b_u = 0\\
L^{| \mathcal I(s_u)| - 2}  &\text{ if }b_u \neq 0 
\end{cases}\\
& \ll_r L^{2r - x' - 2y}L^{-3a_3(\underline{b})}L^{-2(t - x' - y - a_3(\underline{b}))}\, (\text{since} \sum_{u = 1 \atop {|\mathcal I(s_u)| \geq 3}}^t |\mathcal I(s_u)| = 2r - x' - 2y)\\
 &= L^{2r - 2t + x' - a_3(\underline{b})},
 \end{split}
 \end{equation*}
Using the above considerations, and assuming that $L < \frac{\pi_N(x)}{\log \log x}$, the expression in \eqref{T1-further-t-main} can be simplified as follows.
\begin{equation}\label{T1-further-t-main-more}
\begin{split}
&\sum_{l_1,l_2,\dots, l_r \atop {0 \leq l_i \leq L}}U(l_1)U(l_2)\dots U(l_r) \sum_{l_1',l_2',\dots, l_r' \atop {0 \leq l_i' \leq L}}U(l_1')U(l_2')\dots U(l_r')\\
&\sum_{(s_1,s_2,\dots ,s_{t})}\sum_{(b_1,b_2,\dots,b_{t}) \atop {b_u \in \mathcal J_{\mathcal I(s_u)}}}\frac{D(b_1)D(b_2)\dots D(b_{t})}{s_1^{b_1}s_2^{b_2}\dots s_{t}^{b_{t}}}\\
&\ll \sum_{l_1,l_2,\dots, l_r \atop {0 \leq l_i \leq L}}\sum_{l_1',l_2',\dots, l_r' \atop {0 \leq l_i' \leq L}}\sum_{(s_1,s_2,\dots ,s_{t})}\sum_{(b_1,b_2,\dots,b_{t}) \atop {b_u \in \mathcal J_{\mathcal I(s_u)}}}\frac{D(b_1)D(b_2)\dots D(b_{t})}{s_1^{b_1}s_2^{b_2}\dots s_{t}^{b_{t}}}\\
&\ll \sum_{l_1,l_2,\dots, l_r \atop {0 \leq l_i \leq L}}\sum_{l_1',l_2',\dots, l_r' \atop {0 \leq l_i' \leq L}}\sum_{(s_1,s_2,\dots ,s_{t})}\sum_{(b_1,b_2,\dots,b_{t}) \atop {b_u \in \mathcal J_{\mathcal I(s_u)}}} \left(\prod_{u=1 \atop {|\mathcal I(s_u)| \geq 3}}^t \begin{cases} L^{| \mathcal I(s_u)| - 3}  &\text{ if }b_u = 0\\
L^{| \mathcal I(s_u)| - 2}  &\text{ if }b_u \neq 0\\
\end{cases}\right) \frac{1}{s_1^{b_1}s_2^{b_2}\dots s_{t}^{b_{t}}}\\
&\ll_r \sum_{l_1,l_2,\dots, l_r \atop {0 \leq l_i \leq L}}\sum_{l_1',l_2',\dots, l_r' \atop {0 \leq l_i' \leq L}}\sum_{(b_1,b_2,\dots,b_{t}) \atop {b_u \in \mathcal J_{\mathcal I(s_u)}}}L^{2r + x' - 2t- a_3(\underline{b})}\pi_N(x)^{a(\underline{b})} (\log \log x)^{t-a(\underline{b})}\\
\end{split}
\end{equation}
Since $0 \leq a(\underline{b}) \leq t$,  we may rearrange the above sum as follows:
\begin{itemize}
\item We choose $a$ such that $0 \leq a \leq t$, and run over all appropriate tuples $\underline{b}$ such that $a(\underline{b}) = a$.
\item  Now, let us denote
$$Z_t :=  \#\{1 \leq u \leq t:\,0 \in \mathcal J_{\mathcal I(s_u)}\},$$
$$a_1:=   \#\{1 \leq u \leq t:\,|\mathcal I(s_u)| = 1,\,0 \in \mathcal J_{\mathcal I(s_u)}\},$$
$$a_2:=   \#\{1 \leq u \leq t:\,|\mathcal I(s_u)| = 2,\,0 \in \mathcal J_{\mathcal I(s_u)}\},$$
and
$$a_3:=   \#\{1 \leq u \leq t:\,|\mathcal I(s_u)| \geq 3,\,0 \in \mathcal J_{\mathcal I(s_u)}\}.$$
We further define
$$E_a(x) := \#\{(l_1,l_2,\dots, l_r,l_1',l_2',\dots, l_r'):\,0 \leq l_i,\,l_i' \leq L,\,Z_t = a\}.$$

If $|\mathcal I(s_u)| = 1$ and $0 \in \mathcal J_{\mathcal I(s_u)}$, then $\mathcal I(s_u) = \{0\}$.  So, $a_1$ among $\{l_1,l_2,\dots,l_r,l_1',l_2',\dots,l_r'\}$ are equal to 0, and can take no other value.  

On the other hand, suppose $|\mathcal I(s_u)| = 2$, and $\mathcal I(s_u) = \{l,l'\}$.  If $0 \in \mathcal J_{\mathcal I(s_u)}$, then $l = l'$.  Thus, $a_2$ values among $\{l_1,l_2,\dots,l_r,l_1',l_2',\dots,l_r'\}$ are also determined.  

More generally, if $|\mathcal I(s_u)| \geq 3$, and $0 \in \mathcal J_{\mathcal I(s_u)}$, then at least one element in the set $\mathcal I(s_u)$ is determined by other elements of this set.  
Thus, for each $a$, we have the bound
$$E_a(x) \ll L^{2r - a_1 - a_2 - a_3} \ll L^{2r - a},$$
since $a_1 + a_2 + a_3 \geq a(\underline{b}) = a$.
\end{itemize}
Thus,
\begin{equation}\label{T1-further-t-main-more*}
\begin{split}
&\sum_{l_1,l_2,\dots, l_r \atop {0 \leq l_i \leq L}}\sum_{l_1',l_2',\dots, l_r' \atop {0 \leq l_i' \leq L}}\sum_{(b_1,b_2,\dots,b_{t}) \atop {b_u \in \mathcal J_{\mathcal I(s_u)}}}L^{2r + x' - 2t- a_3(\underline{b})}\pi_N(x)^{a(\underline{b})} (\log \log x)^{t-a(\underline{b})}\\
&\ll \sum_{a = 0}^t \pi_N(x)^a (\log \log x)^{t-a}L^{2r + x' - 2t}\sum_{(b_1,b_2,\dots,b_{t}) \atop {b_u \in \mathcal J_{\mathcal I(s_u)} \atop {a(\underline{b}) = a}}}L^{-a_3(\underline{b})}\,E_a(x)\\
&\ll  \sum_{a=0}^t \pi_N(x)^a (\log \log x)^{t-a} L^{2r - a}L^{2r + x' - 2t} \sum_{(b_1,b_2,\dots,b_{t}) \atop {b_u \in \mathcal J_{\mathcal I(s_u)} \atop {a(\underline{b}) = a}}}L^{-a_3(\underline{b})} \\
\end{split}
\end{equation}
Again, note that 
$$a_1 + a_2 + a_3(\underline{b}) \geq a_1(\underline{b}) + a_2(\underline{b})+ a_3(\underline{b}) = a(\underline{b}).$$
Thus, if $a(\underline{b}) = a$, then 
$$a_1 + a_2 + a_3(\underline{b}) \geq a.$$
Thus,
\begin{equation}\label{T1-further-t-more-more}
\begin{split}
& \sum_{a=0}^t \pi_N(x)^a (\log \log x)^{t-a} L^{2r - a}L^{2r + x' - 2t} \sum_{(b_1,b_2,\dots,b_{t}) \atop {b_u \in \mathcal J_{\mathcal I(s_u)} \atop {a(\underline{b}) = a}}}L^{-a_3(\underline{b})} \\
&\ll_r  \sum_{a=0}^t \pi_N(x)^a (\log \log x)^{t-a} L^{2r - a}L^{2r + x' - 2t},
\end{split}
\end{equation}
since
$$\sum_{(b_1,b_2,\dots,b_{t}) \atop {b_u \in \mathcal J_{\mathcal I(s_u)}}} L^{-a_3(\underline{b})} \ll_r 1.$$
Finally, combining equations \eqref{T1-further-t-error}, \eqref{T1-further-t-main-more}, \eqref{T1-further-t-main-more*} and \eqref{T1-further-t-more-more}, we get
\begin{equation}\label{T1-further-t-normalized}
\begin{split}
&\frac{1}{(8\pi_N(x)^2L)^r}(4(G(0)))^r \sum_{l_1,l_2,\dots, l_r \atop {0 \leq l_i \leq L}}U(l_1)U(l_2)\dots U(l_r) \sum_{l_1',l_2',\dots, l_r' \atop {0 \leq l_i' \leq L}}U(l_1')U(l_2')\dots U(l_r')\\
&\sum_{(s_1,s_2,\dots ,s_{t})}\sum_{(b_1,b_2,\dots,b_{t}) \atop {b_u \in \mathcal J_{\mathcal I(s_u)}}}D(b_1)D(b_2)\dots D(b_{t}) \frac{1}{|\mathcal F_{N,k}|}\sum_{f \in \mathcal F_{N,k}} a_f(s_1^{2b_1})a_f(s_2^{2b_2})\dots a_f(s_{t}^{2b_{t}})\\
& \ll_r \frac{1}{(8\pi_N(x)^2L)^r}\sum_{a=0}^t \pi_N(x)^a (\log \log x)^{t-a} L^{2r - a}L^{2r + x' - 2t} + \frac{x^{4Lrc'}\pi_N(x)^t 8^{\nu(N)}}{k N} L^{4r}\\
&\ll  \frac{1}{(8\pi_N(x)^2L)^r}L^{4r + x'  - 2t }(\log \log x)^t \frac{\pi_N(x)^t}{(L\log \log x)^t } + \frac{x^{4Lrc'} 8^{\nu(N)}}{k N} L^{3r}\\
&\ll \frac{L^{3r + x'  - 3t}}{\pi_N(x)^{2r - t}} + \frac{x^{4Lrc'} 8^{\nu(N)}}{k N} L^{3r}\\
&\ll \frac{L^{3r - 2t}}{\pi_N(x)^{2r - t}} + \frac{x^{4Lrc'} 8^{\nu(N)}}{k N} L^{3r},\\
\end{split}
\end{equation}
since $x'  \leq t$.
This proves (a) of the proposition.
\item[{\bf(b)}] 
Let $C_{r,t}$ denote the number of set partitions
$$\{l_1,l_2,\dots,l_r,l_1',l_2',\dots,l_r'\} = \cup_{u=1}^t \mathcal I(s_u),$$
into disjoint subsets $\mathcal I(s_u)$, such that $|\mathcal I(s_u)| \leq r$ for each $u$.  Clearly,$C_{r,t} \ll_r 1$ for each $2 \leq t \leq 2r-1$.  

By part (a) of this proposition, the contribution to
$$\left \langle \frac{1}{(8 \pi_N(x)^2L)^r}\sum_{(p_1,q_1) \atop {p_1 \neq q_1 \leq x \atop{(p_1,N) = (q_1,N) = 1}}}\sum_{(p_2,q_2) \atop {p_2 \neq q_2 \leq x \atop{(p_2,N) = (q_2,N) = 1}}}\dots \sum_{(p_r,q_r) \atop {p_r \neq q_r \leq x \atop{(p_r,N) = (q_r,N) = 1}}} T_{1,r}(\underline{p,q}) \right \rangle$$
from the $2r$-tuples $(p_1,p_2,\dots,p_r,q_1,q_2,\dots, q_r)$ in which the primes are not all distinct is $$\ll_r \sum_{t=2}^{2r-1} C_{r,t} \left(\frac{L^{3r - 2t}}{\pi_N(x)^{2r - t}} + \frac{x^{4Lrc'} 8^{\nu(N)}}{k N} L^{3r}\right).$$
Thus, combining Propositions \ref{2r-distinct-primes} and \ref{T1rpq}(a), we have
\begin{equation*}
\begin{split}
&\frac{1}{(8 \pi_N(x)^2L)^r}\sum_{(p_1,q_1) \atop {p_1 \neq q_1 \leq x \atop{(p_1,N) = (q_1,N) = 1}}}\sum_{(p_2,q_2) \atop {p_2 \neq q_2 \leq x \atop{(p_2,N) = (q_2,N) = 1}}}\dots \sum_{(p_r,q_r) \atop {p_r \neq q_r \leq x \atop{(p_r,N) = (q_r,N) = 1}}} T_{1,r}(\underline{p,q})\\
&\ll_r  \frac{1}{L^r}  + \frac{x^{4rLc'}8^{\nu(N)}}{k N}  + \sum_{t=2}^{2r-1} \left(\frac{L^{3r - 2t}}{\pi_N(x)^{2r - t}} + \frac{x^{4Lrc'} 8^{\nu(N)}}{k N} L^{3r}\right) \\
&\ll_r \sum_{t=2}^{2r} \left(\frac{L^{3r - 2t}}{\pi_N(x)^{2r - t}} \right)+ \frac{x^{4Lrc'} 8^{\nu(N)}}{k N} L^{3r} \\
\end{split}
\end{equation*}
This proves (b) of the proposition.
\end{enumerate}
\end{proof}

\subsection{$\langle T_{2,r}(\underline{p,q}) \rangle$}\label{rM-2}
The goal of this section is to evaluate 
$$ \left \langle \frac{1}{(8 \pi_N(x)^2L)^r}\sum_{(p_1,q_1) \atop {p_1 \neq q_1 \leq x \atop{(p_1,N) = (q_1,N) = 1}}}\sum_{(p_2,q_2) \atop {p_2 \neq q_2 \leq x \atop{(p_2,N) = (q_2,N) = 1}}}\dots \sum_{(p_r,q_r) \atop {p_r \neq q_r \leq x \atop{(p_r,N) = (q_r,N) = 1}}}\\
 T_{2,r}(\underline{p,q})\right \rangle,$$
where
$T_{2,r}(\underline{p,q})$ is as defined in \eqref{T2}.
\begin{equation*}
\begin{split}
&T_{2,r}(\underline{p,q}) = \sum_{j=1}^{r-1} (4(G(0))^{r - j}\binom{r}{j} 2^j \sum_{n_1,n_2,\dots, n_j \atop {1 \leq n_i \leq \pi_N(x)}}G(n_1)G(n_2)\dots G(n_j)\\
&\sum_{l_1,l_2,\dots, l_r \atop {0 \leq l_i \leq L}}U(l_1)U(l_2)\dots U(l_r) \sum_{l_1',l_2',\dots, l_r' \atop {0 \leq l_i' \leq L}}U(l_1')U(l_2')\dots U(l_r') \\
&\prod_{i=1}^j I(p_i,q_i,n_i,l_i,l_i')\prod_{i = j+1}^ r a_f(p_i^{2l_i}) a_f(q_i^{2l_i'}),\\
\end{split}
\end{equation*}
where $I(p_i,q_i,n_i,l_i,l_i')$ is as defined in \eqref{l-rep}.

The innermost part of each term in the above sum is of the form
\begin{equation*}
\begin{split}
&\pm a_f(p_1^{2m_1}) a_f(p_2^{2m_2})\dots a_f(p_j^{2m_j}) a_f(q_1^{2m'_1}) a_f(q_2^{2m'_2})\dots a_f(q_j^{2m'_j})\\
&a_f(p_{j+1}^{2l_{j+1}})a_f(p_{j+2}^{2l_{j+2}})\dots a_f(p_{r}^{2l_{r}})a_f(q_{j+1}^{2l'_{j+1}})a_f(q_{j+2}^{2l'_{j+2}})\dots a_f(q_{r}^{2l'_{r}}),
\end{split}
\end{equation*}
where, for $1 \leq i \leq j$,
\begin{equation}\label{m_i}
\begin{split}
&m_i = l_i + n_i \text{ or }l_i - n_i\, (\text{if }l_i \geq n_i) \text{ or }n_i - l_i - 1 \,(\text{if }l_i < n_i), \text{ and }\\
&m'_i = l'_i + n_i \text{ or }l'_i - n_i \,(\text{if }l'_i \geq n_i) \text{ or }n_i - l'_i - 1 \,(\text{if }l'_i < n_i).
\end{split}
\end{equation}
\begin{lemma}\label{2r-distinct-primes-T2}
Let $1 \leq j \leq r-1$.  $\sum_{\underline{(p,q)}}^{(2r)}$ denote a sum over all $2r$-tuples of distinct primes $p_1,q_1,p_2,q_2,\dots,p_r,q_r  \leq x$ coprime to $N$.  We choose $L = L(x) \to \infty$ such that $L(x) < \frac{\pi_N(x)}{(\log \log x)^2}$.  Then,
\begin{equation*}
\begin{split}
& \frac{1}{(8 \pi_N(x)^2L)^r} \sum_{\underline{(p,q)}}^{(2r)} \sum_{n_1,n_2,\dots, n_j \atop {1 \leq n_i \leq \pi_N(x)}}G(n_1)G(n_2)\dots G(n_j)\\
&\sum_{l_1,l_2,\dots, l_r \atop {0 \leq l_i \leq L}}U(l_1)U(l_2)\dots U(l_r) \sum_{l_1',l_2',\dots, l_r' \atop {0 \leq l_i' \leq L}}U(l_1')U(l_2')\dots U(l_r') \\
&\frac{1}{|\mathcal F_{N,k}|} \sum_{f \in \mathcal F_{N,k}}\prod_{i=1}^j I(p_i,q_i,n_i,l_i,l_i')\prod_{i = j+1}^ r a_f(p_i^{2l_i}) a_f(q_i^{2l_i'})\\
& \ll_r  \frac{L^j}{L^r} + \frac{x^{\delta(r) \pi_N(x)}8^{\nu(N)}}{k  N}.
\end{split}
\end{equation*}
Here $\delta(r)$ is a positive real number that depends only on $r$.
\end{lemma}

\begin{proof}
The above sum can be broken down into $R$ sums of the form
\begin{equation}\label{Lemma-above-part}
\begin{split}
& \frac{1}{(8 \pi_N(x)^2L)^r} \sum_{\underline{(p,q)}}^{(2r)} \sum_{n_1,n_2,\dots, n_j \atop {1 \leq n_i \leq \pi_N(x)}}G(n_1)G(n_2)\dots G(n_j)\\
&\sum_{l_1,l_2,\dots, l_r \atop {0 \leq l_i \leq L}}U(l_1)U(l_2)\dots U(l_r) \sum_{l_1',l_2',\dots, l_r' \atop {0 \leq l_i' \leq L}}U(l_1')U(l_2')\dots U(l_r') \\
&\frac{1}{|\mathcal F_{N,k}|} \sum_{f \in \mathcal F_{N,k}} a_f(p_1^{2m_1}) a_f(p_2^{2m_2})\dots a_f(p_j^{2m_j}) a_f(q_1^{2m'_1}) a_f(q_2^{2m'_2})\dots a_f(q_j^{2m'_j})\\
&a_f(p_{j+1}^{2l_{j+1}})a_f(p_{j+2}^{2l_{j+2}})\dots a_f(p_{r}^{2l_{r}})a_f(q_{j+1}^{2l'_{j+1}})a_f(q_{j+2}^{2l'_{j+2}})\dots a_f(q_{r}^{2l'_{r}}),\\
\end{split}
\end{equation}
where $R = \O_r(1)$. 

We count the number of $2r + j$-tuples $A_{a,b}(\underline{n},\underline{l},\underline{l'})$ 
$$\{(n_1,n_2,\dots,n_j,l_1,l_2,\dots,l_r,l_1',l_2',\dots,l_r'):\,1 \leq n_i \leq \pi_N(x),\,0 \leq l_i,l_i' \leq L\}$$
such that there are $a$ zeroes among $$\{m_1,m_2,\dots,m_j,m_1',m_2',\dots,m_j'\}$$ and $b$ zeroes among
$$\{l_{j+1},\dots l_r,l_{j+1}',\dots l_r'\}.$$
We note that $0 \leq a \leq 2j$ and $0 \leq b \leq 2r-2j$.
Let us denote
$$a_1 =\left( \#\{ 1 \leq i \leq j:\,m_i = 0,\,m_i' \neq 0\}\right) + \left(\#\{ 1 \leq i \leq j:\,m_i \neq 0,\,m_i' = 0\}\right),$$
and
$$a_2 := \#\{ 1 \leq i \leq j:\,m_i = m'_i = 0\}.$$
Since each $n_i \geq 1$, $n_i + l_i ,\,n_i + l_i' \neq 0$ for any $1 \leq i \leq j$.  Therefore, $m_i = 0$ only if $n_i = l_i \text{ or }l_i + 1$.  Thus, the component $n_i$ is determined by the corresponding choice of $l_i$. Hence, $a_1 + a_2$ counts those values of $n_i$, which are determined by the corresponding choice of $l_i$ or $l'_i$.  Further, if $m_i = m_i' = 0$, then both $n_i$ and $l_i'$ are determined by the choice of $l_i$.  Moreover, $b$ among $\{l_{j+1},\dots l_r,l_{j+1}',\dots l_r'\}$ are zero.  We deduce that
$$A_{a,b}(\underline{n},\underline{l},\underline{l'}) \ll \pi_N(x)^{j - (a_1 + a_2)}L^{2r - a_2 - b}.$$

%
%
By the above estimates for $A_{a,b}(\underline{n},\underline{l},\underline{l'})$ and Proposition \ref{sums-distinct-primes},  the sum in \eqref{Lemma-above-part} is equal to 
\begin{equation*}
\begin{split}
&\frac{1}{(8 \pi_N(x)^2L)^r}\sum_{n_1,n_2,\dots, n_j \atop {1 \leq n_i \leq \pi_N(x)}}G(n_1)G(n_2)\dots G(n_j)\\
&\sum_{l_1,l_2,\dots, l_r \atop {0 \leq l_i \leq L}}U(l_1)U(l_2)\dots U(l_r) \sum_{l_1',l_2',\dots, l_r' \atop {0 \leq l_i' \leq L}}U(l_1')U(l_2')\dots U(l_r') \\
& \sum_{\underline{(p,q)}}^{(2r)} \frac{1}{\prod_{i=1}^j p_i^{m_i}\prod_{i=1}^j q_i^{m_i'}\prod_{i=j+1}^r p_i^{l_i}\prod_{i=j+1}^r q_i^{l'_i}} \\
&+ \O\left(\frac{x^{\delta(r) \pi_N(x)}8^{\nu(N)}}{k N}\right)\\
&\ll_r  \frac{1}{( \pi_N(x)^2L)^r} \sum_{a = 0}^{2j}\sum_{b = 0}^{2r - 2j} \pi_N(x)^{a + b}(\log \log x)^{2r - (a+b)} \sum_{a_2 = 0}^{[a/2]} \pi_N(x)^{j - a_1 - a_2}L^{2r - b - a_2}\\
& + \frac{x^{\delta(r) \pi_N(x)}8^{\nu(N)}}{k  N}\\
\end{split}
\end{equation*}
where $\delta(r)$ depends only on $r$.  Since $a_1 = a - 2a_2$, we have
\begin{equation*}
\begin{split}
&\sum_{a_2 = 0}^{[a/2]} \pi_N(x)^{j - a_1 - a_2}L^{2r - b - a_2}\\
& = \sum_{a_2 = 0}^{[a/2]} \pi_N(x)^{j - a + a_2} L^{2r - b - a_2}\\
& \ll \pi_N(x)^{j - a} L^{2r-b}\left(\frac{\pi_N(x)}{L}\right)^{a/2}\\
&= \pi_N(x)^{j - a/2}L^{2r - b - a/2}.
\end{split}
\end{equation*}
Since $\frac{\pi_N(x)}{L(\log \log x)^2} > 1$, we have
\begin{equation*}
\begin{split}
&\sum_{a = 0}^{2j}\sum_{b = 0}^{2r - 2j} \pi_N(x)^{a + b}(\log \log x)^{2r - (a+b)} \sum_{a_2 = 0}^{[a/2]} \pi_N(x)^{j - (a - 2a_2) - a_2}L^{2r - b - a_2}\\
&\ll \sum_{a = 0}^{2j}\sum_{b = 0}^{2r - 2j} \pi_N(x)^{a/2 + b + j}L^{2r - b -a/2}(\log \log x)^{2r - (a+b)}\\
&=(\log \log x)^{2r}\pi_N(x)^j L^{2r} \sum_{a = 0}^{2j} \left(\sqrt{\frac{\pi_N(x)}{L(\log \log x)^2}}\right)^{a}\sum_{b = 0}^{2r - 2j} \left(\frac{\pi_N(x)}{L(\log \log x)}\right)^b\\
&=(\log \log x)^{2r}\pi_N(x)^j L^{2r}\left(\frac{\pi_N(x)}{L(\log \log x)^2}\right)^j \left(\frac{\pi_N(x)}{L(\log \log x)}\right)^{2r - 2j}\\
&=\pi_N(x)^{2r}L^{j}.
\end{split}
\end{equation*}
Thus,
\begin{equation*}
\begin{split}
&\frac{1}{( \pi_N(x)^2L)^r} \sum_{a = 0}^{2j}\sum_{b = 0}^{2r - 2j} \pi_N(x)^{a + b}(\log \log x)^{2r - (a+b)} \sum_{a_2 = 0}^{[a/2]} \pi_N(x)^{j - a_1 - a_2}L^{2r - b - a_2} +  \frac{x^{\delta(r) \pi_N(x)}8^{\nu(N)}}{k  N}\\
& \ll \frac{L^{j}}{L^{r}} +  \frac{x^{\delta(r) \pi_N(x)}8^{\nu(N)}}{k  N}
\end{split}
\end{equation*}
\end{proof}
We now evaluate the part of $T_{2,r,j}(\underline{p},\underline{q})$ in which the primes $p_1,\dots,p_r,q_1,\dots,q_r$ are not all distinct.  To do this, for each $2 \leq t \leq 2r-1$, we consider the case such that the $2r$-tuple $(p_1,\dots,p_r,q_1,\dots,q_r)$ has $t$ distinct primes.  Analogous to Notation \ref{notation-t-distinct-primes}, we use the following notation for $j >1$.
\begin{notation}\label{notation-t-distinct-primes-j}
We use the following notation to study the terms in $T_{2,r,j}(\underline{p,q})$ such that the primes in the $2r$-tuples $(p_1,p_2,\dots,p_r,q_1,q_2,\dots, q_r)$ are not all distinct. 

\begin{enumerate}
\item[{\bf (a)}] Suppose that the number of distinct primes in a tuple $(p_1,p_2,\dots,p_r,q_1,q_2,\dots, q_r)$ is $t$ for some $2 \leq t \leq 2r - 1$.  Let us denote the distinct primes in the above tuple as $s_1,s_2,\dots, s_t$.  
\item[{\bf (b)}] We write
\begin{equation}\label{t-distinct-primes}
\prod_{i=1}^j a_f(p_i^{2m_i})\prod_{i=j+1}^r a_f(p_i^{l_i})\prod_{i=1}^j a_f(q_i^{m_i'})\prod_{i=j+1}^r a_f(q_i^{l'_i}) =  \prod_{u=1}^ t \left(\prod_{i \in \mathcal I(s_u)} a_f(s_u^{2i})\right),
\end{equation}
where, for each $1 \leq u \leq t$, 
$$ \mathcal I(s_u) \subset \{m_1,\dots,m_j,l_{j+1},\dots ,l_r,m_1',\dots,m_j',l'_{j+1},\dots ,l'_r\}$$
consists of all elements which appear as exponents of $s_u$ in \eqref{t-distinct-primes}.
Note that 
\begin{itemize}
\item The sets $\mathcal I(s_u)$ are mutually disjoint.
\item We have
$$\cup_{u=1}^t \mathcal I(s_u) = \{m_1,\dots,m_j,l_{j+1},\dots ,l_r,m_1',\dots,m_j',l'_{j+1},\dots ,l'_r\},$$
and therefore, $\sum_{u=1}^t |\mathcal I(s_u)| = 2r.$
\item Denote $i_u = |\mathcal I(s_u)|$.  Let us assume that $\mathcal I(s_u)$ has $d_u$ elements among
$$\{m_1,m_2,\dots ,m_j,m'_1,m'_2,\dots,m'_j\}$$ and $i_u - d_u$ elements among 
$$\{l_{j+1},\dots,l_r,l'_{j+1},\dots,l'_r\}.$$
For each $1 \leq i \leq j$, $m_i$ and $m_i'$ cannot both belong to $\mathcal I(s_u)$.   Thus, $d_u \leq j$.  Similarly, for each $j+1 \leq i \leq r$, $l_i$ and $l_i'$ cannot both belong to $\mathcal I(s_u)$.  Thus, $i_u \leq r$ for each $1 \leq u \leq t$.
\end{itemize}

\item[{\bf (c)}] Let $\mathcal J_{\mathcal I(s_u)}$ denote the set of all integers $b_u$, such that $2b_u$ appears in the power of $s_u$ when
$$\prod_{i_u \in \mathcal I(s_u)} a_f(s_u^{2i_u})$$
is expanded into a sum.  That is,
$$\prod_{i_u \in \mathcal I(s_u)} a_f(s_u^{2i_u}) = \sum_{b_u \in \mathcal J_{\mathcal I(s_u)}}D_{\mathcal I(s_u)}(b_u)a_f(s_u^{2b_u}).$$
\item[{\bf (d)}] For each $u$, denote $v(u)$ to be the sum of all the elements of $\mathcal I(s_u)$.  Thus, $0 \leq b_u \leq v(u)$.
\end{enumerate}
\end{notation}

This leads to an evaluation of sums of the form
\begin{equation}\label{parts-of-T3}
\begin{split}
& \frac{1}{(8 \pi_N(x)^2L)^r}\sum_{n_1,n_2,\dots, n_j \atop {1 \leq n_i \leq \pi_N(x)}}G(n_1)G(n_2)\dots G(n_j)\\
&\sum_{l_1,l_2,\dots, l_r \atop {0 \leq l_i \leq L}}U(l_1)U(l_2)\dots U(l_r) \sum_{l_1',l_2',\dots, l_r' \atop {0 \leq l_i' \leq L}}U(l_1')U(l_2')\dots U(l_r') \\
&\sum_{(s_1,s_2,\dots ,s_{t})}\frac{1}{|\mathcal F_{N,k}|}\sum_{f \in \mathcal F_{N,k}}  \prod_{u = 1}^t \prod_{i \in \mathcal I(s_u)}a_f(s_u^{2i}),
\end{split}
\end{equation}
where
\begin{itemize}
 \item $\sum_{(s_1,s_2,\dots ,s_{t})}$ denotes a sum over $t$-tuples of distinct primes, such that each $s_i$ is less than or equal to $x$ and coprime to $N$.

\item The set partition $\{\mathcal I(s_u),\,1 \leq u \leq t\}$ is as defined in Notation \ref{notation-t-distinct-primes-j}.

\item For each $u$,
\begin{equation}\label{vuj}
v(u) \leq d_u(\pi_N(x) + L) + (i_u - d_u) L = d_u(\pi_N(x)) + i_u L < 2i_u \pi_N(x), \text{ since }L < \pi_N(x).
\end{equation}

\end{itemize}

We need another variant of Lemma \ref{D(m)}.
\begin{lemma}\label{D(m)-mixed}
Let $i_u = |\mathcal I(s_u)|$ and
$$\mathcal I(s_u) = \{m_1,m_2,\dots,m_{d_u},l_{d_u+1},\dots,l_{i_u}\},$$
where, for $1 \leq i \leq d_u$,
$$m_i = l_i + n_i \text{ or }l_i - n_i\, (\text{if }l_i \geq n_i) \text{ or }n_i - l_i - 1 \,(\text{if }l_i < n_i).$$
Let $1 \leq n_i \leq \pi_N(x)$ and $0 \leq l_i \leq L$.
Write
$$\prod_{i \in \mathcal I(s_u)} a_f(s_u^{2i}) = \sum_{b_u \in \mathcal J_{\mathcal I(s_u)}}D(b_u)a_f(s_u^{2b_u}).$$
\begin{enumerate}
\item For each $b_u \in \mathcal J_{\mathcal I(s_u)}$,
$$D(b_u) = \begin{cases}
$$\O(1)&\text{ if }i_u = 1,\,2\\
\O(L^{i_u-2})&\text{ if }i_u \geq 3,\,d_u = 0,1\\
\O(\pi_N(x)^{d_u -1}L^{i_u - d_u - 1})&\text { if }i_u \geq 3,\,2 \leq d_u <i_u,\\
\O(\pi_N(x)^{d_u -2})&\text { if }i_u \geq 3,\,i_u = d_u.
\end{cases}$$
\item Further,
$$D(0) = 
\begin{cases}
\O(L^{i_u - 3})&\text{ if }i_u \geq 3,\,d_u = 0,1\\
\O(\pi_N(x)^{d_u -2}L^{i_u - d_u - 1})&\text { if }i_u \geq 3,\,2 \leq d_u <i_u,\\
\O(\pi_N(x)^{d_u -3})&\text { if }i_u \geq 3,\,i_u = d_u.
\end{cases}$$
\end{enumerate}
The implied constants in all the above estimates depend only on $r$.

\end{lemma}
\begin{proof}
The estimates for the following cases immediately follow by Lemma \ref{D(m)}.
\begin{itemize} 
\item $i_u \geq 1,\,d_u = 0$, 
\item $i_u = d_u = 1$,
\item $i_u = 2,\,d_u =1$,
\item $i_u = 2,\,d_u = 2$,
\item $i_u = d_u \geq 3$ (the estimates in the above lemma for this case can be obtained if we replace $L$ by $\pi_N(x)$ in the statement and proof of Lemma \ref{D(m)}).
\end{itemize} 

Now, consider the case when $i_u \geq 3,\,i_u > d_u \geq 1$.  

If $d_u = 1$, then $\mathcal I(s_u) = \{m_{j_1}\} \cup \{l_{j_2},l_{j_3},\dots,l_{j_{i_u}}\}$.  We consider the expansion

$$a_f(s_u^{2m_{j_1}}) \left(a_f(s_u^{2l_{j_2}})a_f(s_u^{2l_{j_3}})\dots a_f(s_u^{2l_{j_{i_{u}}}})\right) = a_f(s_u^{2m_{j_1}})\sum_{t \in \mathcal J_{\{l_{j_2},l_{j_3},\dots,l_{j_{i_u}}\}}} D(t) a_f(s_u^{2t}).$$
By Lemma \ref{D(m)}, $D(t) \ll L^{i_u-3}$.  We also note that $b_u \in \mathcal J_{\mathcal I(s_u)}$ can occur at most once in each of the expansions 
$$a_f(s_u^{2m_{j_1}}) a_f(s_u^{2t}).$$
Thus, for $b_u \in \mathcal J_{\mathcal I(s_u)}$,
$$D(b_u) \ll D(t)|\mathcal J_{\{l_{j_2},l_{j_3},\dots,l_{j_{i_u}}\}}| \ll_{i_u} L^{i_u-3}L \ll_r L^{i_u-2}.$$
Thus, the inequality holds when $d_u = 1$.

Assume that the inequality holds if $d_u = b$.  
Now, suppose $$\mathcal I(s_u) = \{m_{b+1}\} \cup \{m_b,m_{b-1},\dots, m_1,l_{b+2},l_{b+3},\dots, l_{i_u}\}.$$
Denote $$\mathcal I'(s_u) := \{m_b,m_{b-1},\dots, m_1,l_{b+2},l_{b+3},\dots, l_{i_u}\}.$$
By induction hypothesis, for any $t \in \mathcal J_{\mathcal I'(s_u)}$, 
$$D(t) \ll \pi_N(x)^{b -1}L^{i_u - 1- b - 1}.$$
Thus, for $b_u \in \mathcal J_{\mathcal I(s_u)}$,
$$D(b_u)  \ll D(t)|\mathcal J_{\mathcal I'(s_u)}|\ll_r \pi_N(x)^{b -1}L^{i_u - b - 2} \left(\sum_{a \in \mathcal I'(s_u)}2a\right) \ll_r \pi_N(x)^{b}L^{i_u - b - 2}.$$
Thus, by induction, we have proved (1).

We now prove (2).  The estimate $D(0) = \O(L^{i_u - 3})$ for $i_u \geq 3,\,d_u = 0$ follows from Lemma \ref{D(m)}.  

The proof is similar to the second part of the proof of Lemma \ref{D(m)}.  
Let $d_u = 1$.  Then, $\mathcal I(s_u)$ is of the form $\{m_1\} \cup \{l_2,l_3,\dots,l_{i_u}\}$.  Denote $\mathcal I'(s_u) =  \{l_2,l_3,\dots,l_{i_u}\}$.  For a prime $s$,  
\begin{equation*}
\begin{split}
&a_f(s^{2m_1})\left(a_f(s^{2l_2})a_f(s^{2l_3})\dots a_f(s^{2l_{i_u}})\right) = a_f(s^{2m_1}) \sum_{t \in \mathcal J_{\mathcal I'(s_u)}} D(t) a_f(s^{2t})\\
&= \sum_{t \in \mathcal J_{\mathcal I'(s_u)}} \sum_{j = 0}^{\min\{2t,2m_1\}}D(t) a_f(s^{2t + 2m_1 - 2j})\\
\end{split}
\end{equation*}
Note that $2t+ 2m_1 - 2j = 0$ if and only if $2m_1 =2t =  j$.  Thus, the coefficient of $a_f(s^0)$ in the above expansion is, by Lemma \ref{D(m)},
$$\leq D_{\mathcal I'(s_u)}(m_1) \ll L^{i_u - 1 -2}.$$
That is,
$$D_{\mathcal I(s_u)}(0) \ll L^{i_u - 3}.$$
Let us now consider the case when $2 \leq d_u < i_u$.  
Let 
$$\mathcal I'(s_u) = \{m_2, m_3,\dots, m_{d_u + 1}, l_{d_u + 2},\dots,l_{i_u}\}.$$

Now, consider 
$$\mathcal I(s_u) = \{m_1\} \cup \{m_2, m_3,\dots, m_{d_u}, m_{d_u + 1},l_{i_u + 2},\dots,l_{i_u}\}.$$
We want to show that 
$$D_{\mathcal I(s_u)}(0) = \O(\pi_N(x)^{d_u + 1-2}L^{i_u - (d_u + 1) - 1}).$$

We note that  
\begin{equation*}
\begin{split}
&a_f(s_u^{2m_1})\left(\prod_{i \in \mathcal I'(s_u)}a_f(s_u^{2i})\right) = a_f(s_u^{2m_1}) \sum_{t \in \mathcal J_{\mathcal I'(s_u)}} D_{\mathcal I'(s_u)}(t) a_f(s_u^{2t})\\
&= \sum_{t \in \mathcal J_{\mathcal I'(s_u)}} \sum_{j = 0}^{\min\{2t,2m_1\}} D_{\mathcal I'(s_u)}(t) a_f(s_u^{2t + 2m_1 - 2j})\\
\end{split}
\end{equation*}
As before, $2t+ 2m_1 - 2j = 0$ if and only if $2m_1 =2t =  j$.  Thus, the coefficient of $a_f(s^0)$ in the above expansion is $\leq D_{\mathcal I'(s_u)}(m_1)$.

That is, 
$$D_{\mathcal I(s_u)}(0) \leq D_{\mathcal I'(s_u)}(m_1),$$
and by part (1) of this lemma,
$$ D_{\mathcal I'(s_u)}(m_1)=\O(\pi_N(x)^{d_u -1}L^{i_u - 1- d_u - 1}).$$
The case $d_u = i_u$ follows from Lemma \ref{D(m)}.
Thus, we have proved part (2) of the lemma.

\end{proof}

We now evaluate the expression in \eqref{parts-of-T3} for such a choice of $t$ and $\{\mathcal I(s_u)\}$.

Applying Proposition \ref{trace-estimate-later-calculation},
\begin{equation*}
\begin{split}
&\sum_{n_1,n_2,\dots, n_j \atop {1 \leq n_i \leq \pi_N(x)}}G(n_1)G(n_2)\dots G(n_j)\sum_{l_1,l_2,\dots, l_r \atop {0 \leq l_i \leq L}}U(l_1)U(l_2)\dots U(l_r)\\
& \sum_{l_1',l_2',\dots, l_r' \atop {0 \leq l_i' \leq L}}U(l_1')U(l_2')\dots U(l_r') \left \langle \prod_{u=1}^t \prod_{i \in \mathcal I(s_u)} a_f(s_u^{2i}) \right \rangle\\
&= \sum_{n_1,n_2,\dots, n_j \atop {1 \leq n_i \leq \pi_N(x)}}G(n_1)G(n_2)\dots G(n_j)\sum_{l_1,l_2,\dots, l_r \atop {0 \leq l_i \leq L}}U(l_1)U(l_2)\dots U(l_r)\\
& \sum_{l_1',l_2',\dots, l_r' \atop {0 \leq l_i' \leq L}}U(l_1')U(l_2')\dots U(l_r') \left \langle\prod_{u=1}^t \sum_{b_u \in \mathcal J_{\mathcal I(s_u)}} D(b_u) a_f(s_u^{2b_u})\right \rangle\\
&=  \sum_{n_1,n_2,\dots, n_j \atop {1 \leq n_i \leq \pi_N(x)}}G(n_1)G(n_2)\dots G(n_j)\sum_{l_1,l_2,\dots, l_r \atop {0 \leq l_i \leq L}}U(l_1)U(l_2)\dots U(l_r)\\
& \sum_{l_1',l_2',\dots, l_r' \atop {0 \leq l_i' \leq L}}U(l_1')U(l_2')\dots U(l_r')\sum_{(s_1,s_2,\dots ,s_{t})}\sum_{(b_1,b_2,\dots,b_{t}) \atop {b_u \in \mathcal J_{\mathcal I(s_u)}}}D(b_1)D(b_2)\dots D(b_{t})\\
&\times \begin{cases} 1 &\text{ if }b_1 = b_2 = \dots = b_{t} = 0\\
\frac{1}{s_1^{b_1}s_2^{b_2}\dots s_{t}^{b_{t}}}+ \O\left(\frac{\left(s_1^{2b_1}s_2^{2b_2}\dots s_{t}^{2b_{t}}\right)^{c'} 8^{\nu(N)}}{k{N}}\right)&\text{ if }b_u \geq 1 \text{ for some }u\\
\end{cases}
\end{split}
\end{equation*}
As in the remarks after equation \eqref{T1-further-t-main}, we define the following quantities. 
\begin{itemize}
\item For each $1 \leq i \leq r$, let $x_i = \#\{u:\,i_u = i\}$.
\item For each $1 \leq i \leq r$ and $0 \leq d \leq j$, let
$$x_{i,d} := \#\{u:\,i_u = i,\,d_u = d\},$$
\item For a tuple $$\underline{b} = (b_1,b_2,\dots,b_t),\,b_u \in  \mathcal J_{\mathcal I(s_u)},$$ let $a(\underline{b})$ denote the number of $u$'s such that $b_u = 0$.

\item For a tuple $\underline{b}$ as above, let 
$$B_1(\underline{b}) := \{1 \leq u \leq t:\,i_u \geq 3,\,b_u = 0,\,d_u = 0,1\},$$
$$B_2(\underline{b}) :=\{1 \leq u \leq t:\,i_u \geq 3,\,b_u = 0,\,2 \leq d_u <i_u\},$$
$$B_3(\underline{b}) :=\{1 \leq u \leq t:\,i_u \geq 3,\,b_u \neq 0,\,d_u = 0,1\},$$
$$B_4(\underline{b}) :=\{1 \leq u \leq t:\,i_u \geq 3,\,b_u \neq 0,\,2 \leq d_u <i_u\},$$
$$B_5(\underline{b}) :=\{1 \leq u \leq t:\,i_u \geq 3,\,b_u = 0,\,d_u = i_u\},$$
and
$$B_6(\underline{b}) := \{1 \leq u \leq t:\,i_u \geq 3,\,b_u \neq 0,\,d_u = i_u\}.$$
\item For each $1 \leq i \leq r$ and $0 \leq d \leq j$, let
$$a_{i,d}(\underline{b}) := \#\{u:\,i_u = i,\,d_u = d,b_u = 0\},$$
and
$$a'_{i,d}(\underline{b}):= \#\{u:\,i_u = i,\,d_u = d,b_u \neq 0\}.$$
\end{itemize}
Thus, we have
\begin{equation}\label{a-a'-x}
a_{i,d}(\underline{b}) + a'_{i,d}(\underline{b}) = x_{i,d} \text{ for each }i,d,
\end{equation}
and
\begin{equation}\label{xid}
\sum_{d = 0}^j x_{i,d} = x_i,\,\sum_{i=1}^r x_i = t.
\end{equation}
We also have
\begin{equation}\label{sums-i}
\sum_{u = 1}^t i_u = 2r,
\end{equation}
\begin{equation}\label{sums-d}
\sum_{u = 1}^t d_u = 2j,
\end{equation}
and
\begin{equation}\label{sums-aid}
\sum_{i \geq 1 \atop {d \geq 0}}a_{i,d}(\underline{b}) = a(\underline{b}).
\end{equation}

If $i_u = 1$ or 2, then $D(b_u) = \O(1)$.  Thus,
$$D(b_1)D(b_2)\dots D(b_{t}) = \prod_{k = 1}^6 A_k (\underline{b}),$$
where, for each $1 \leq k \leq 6$,
$$A_k(\underline{b}) = \prod_{u \in B_k(\underline{b})}D(b_u).$$

By Lemma \ref{D(m)-mixed},
$$A_1(\underline{b}) \ll L^{\left(\sum_{u \in B_1(\underline{b})} i_u\right) - 3|B_1(\underline{b})|},$$
$$A_2(\underline{b}) \ll \pi_N(x)^{\left(\sum_{u \in B_2(\underline{b})} d_u\right) - 2|B_2(\underline{b})|}L^{\left(\sum_{u \in B_2(\underline{b})} i_u - d_u\right) - |B_2(\underline{b})|},$$
$$A_3(\underline{b}) \ll L^{\left(\sum_{u \in B_3(\underline{b})} i_u\right) - 2|B_3(\underline{b})|},$$
$$A_4(\underline{b}) \ll \pi_N(x)^{\left(\sum_{u \in B_4(\underline{b})} d_u\right) - |B_4(\underline{b})|}L^{\left(\sum_{u \in B_4(\underline{b})} i_u - d_u\right) - |B_4(\underline{b})|},$$
$$A_5(\underline{b}) \ll \pi_N(x)^{\left(\sum_{u \in B_5(\underline{b})} d_u\right) - 3|B_5(\underline{b})|},$$
and
$$A_6(\underline{b}) \ll \pi_N(x)^{\left(\sum_{u \in B_6(\underline{b})} d_u\right) - 2|B_6(\underline{b})|}.$$
Note that in the estimates for $A_5(\underline{b})$ and $A_6(\underline{b})$ above, $d_u = i_u$.  Thus,
\begin{equation}\label{A_i-product}
\begin{split}
&D(b_1)D(b_2)\dots D(b_{t}) = A_1(\underline{b})A_2(\underline{b})A_3(\underline{b})A_4(\underline{b})A_5(\underline{b})A_6(\underline{b})\\
& \ll \pi_N(x)^{C_3(t)} \pi_N(x) ^{-2|B_2(\underline{b})| - |B_4(\underline{b})| -3|B_5(\underline{b})| -2|B_6(\underline{b})|} L^{C_1(t) - C_2(t)} L^{-3|B_1(\underline{b})| - |B_2(\underline{b})| - 2|B_3(\underline{b})| - |B_4(\underline{b})|}
\end{split}
\end{equation}
 where,
 $$C_1(t) = \sum_{u:\,i_u \geq 3 \atop {0 \leq d_u < i_u}} i_u,\,C_2(t) = \sum_{u:\,i_u \geq 3 \atop {2 \leq d_u < i_u}}d_u,$$
 and
 $$C_3(t) = \sum_{u:\,i_u \geq 3 \atop {2 \leq d_u \leq i_u}}d_u .$$
 We have
 \begin{equation}\label{C_1(t)}
C_1(t) = \sum_{u:\,i_u \geq 3 \atop {0 \leq d_u < i_u}} i_u = \sum_{i \geq 3} \sum_{d =0}^{i-1} ix_{i,d},
\end{equation}
 \begin{equation}\label{C_2(t)}
 \begin{split}
 &C_2(t) = \sum_{u:\,i_u \geq 3 \atop {2 \leq d_u < i_u}} d_u = \sum_{i \geq 3} \sum_{d = 2}^{i-1} dx_{i,d}  \\
 &= \sum_{i \geq 3}\sum_{d =2}^i dx_{i,d} - \sum_{i \geq 3} i x_{i,i},\\
 \end{split}
 \end{equation}
 \begin{equation}\label{C_3(t)}
 C_3(t) = \sum_{u:\,i_u \geq 3 \atop {2 \leq d_u \leq i_u}} d_u = \sum_{i \geq 3} \sum_{d = 2}^{i} dx_{i,d}, \\
 \end{equation}
and
 \begin{equation*}
 \begin{split}
 &2|B_2(\underline{b})| + |B_4(\underline{b})| + 3|B_5(\underline{b})| + 2|B_6(\underline{b})|\\
 &= \left(|B_2(\underline{b})| +  |B_4(\underline{b})| +  |B_5(\underline{b})| +  |B_6(\underline{b})|\right) + \left( |B_2(\underline{b})| +  |B_5(\underline{b})|\right) + \left(|B_5(\underline{b})| + |B_6(\underline{b})|\right)\\
 &= \sum_{i \geq 3}\sum_{d =2}^i x_{i,d} + \sum_{i \geq 3}\sum_{d =2}^i a_{i,d}(\underline{b}) + \sum_{i \geq 3} x_{i,i}.
 \end{split}
 \end{equation*}
 Thus,
 \begin{equation}\label{power-of-pi}
  \begin{split}
 &C_3(t) - \left(2|B_2(\underline{b})| + |B_4(\underline{b})| + 3|B_5(\underline{b})| + 2|B_6(\underline{b})|\right)\\
 &= \sum_{i \geq 3}\sum_{d =2}^i dx_{i,d} - \left(\sum_{i \geq 3}\sum_{d =2}^i x_{i,d} + \sum_{i \geq 3}\sum_{d =2}^i a_{i,d}(\underline{b}) + \sum_{i \geq 3} x_{i,i}\right)\\
 &= \sum_{i \geq 3}\sum_{d =2}^i (d-1)x_{i,d} -  \sum_{i \geq 3}  x_{i,i} - \sum_{i \geq 3}\sum_{d =2}^i a_{i,d}(\underline{b})
 \end{split}
 \end{equation}
 We also note that
 \begin{equation*}
 C_1(t) - C_2(t) = \sum_{i \geq 3} \sum_{d = 0}^{1} i x_{i,d} + \sum_{i \geq 3} \sum_{d = 2}^{i-1} (i -d)x_{i,d},
 \end{equation*}
 and
 \begin{equation*}
 \begin{split}
 &3|B_1(\underline{b})| + |B_2(\underline{b})| + 2|B_3(\underline{b})| + |B_4(\underline{b})|\\
 &= \left(|B_1(\underline{b})| + |B_2(\underline{b})| + |B_3(\underline{b})| + |B_4(\underline{b})|\right) + \left(|B_1(\underline{b})| + |B_3(\underline{b})| \right) + |B_1(\underline{b})|\\
 &= \sum_{i \geq 3}\sum_{d=0}^{i-1} x_{i,d} + \sum_{i \geq 3}\sum_{d=0}^{1} x_{i,d} + \sum_{i \geq 3}\sum_{d=0}^{1}a_{i,d}(\underline{b}).
 \end{split}
 \end{equation*}
 Thus,
 \begin{equation}\label{power-of-L}
 \begin{split}
 &C_1(t) - C_2(t) - \left(3|B_1(\underline{b})| + |B_2(\underline{b})| + 2|B_3(\underline{b})| + |B_4(\underline{b})|\right)\\
 &= \sum_{i \geq 3} \sum_{d = 0}^{1} i x_{i,d} + \sum_{i \geq 3} \sum_{d = 2}^{i-1} (i -d)x_{i,d}  - \left( \sum_{i \geq 3}\sum_{d=0}^{i-1} x_{i,d} +  \sum_{i \geq 3}\sum_{d=0}^{1} x_{i,d} + \sum_{i \geq 3}\sum_{d=0}^{1}a_{i,d}(\underline{b})\right)\\
 &= \sum_{i \geq 3} \sum_{d = 0}^{1} (i - 2) x_{i,d} + \sum_{i \geq 3} \sum_{d = 2}^{i-1} (i -d - 1)x_{i,d}  -  \sum_{i \geq 3}\sum_{d=0}^{1}a_{i,d}(\underline{b}).
 \end{split}
 \end{equation}
 By equations \eqref{A_i-product}, \eqref{power-of-pi} and \eqref{power-of-L},
 \begin{equation}\label{prod-D(b)}
\begin{split}
&D(b_1)D(b_2)\dots D(b_{t}) = A_1(\underline{b})A_2(\underline{b})A_3(\underline{b})A_4(\underline{b})A_5(\underline{b})A_6(\underline{b})\\
& \ll_r \pi_N(x)^{\sum_{i \geq 3}\sum_{d =2}^i (d-1)x_{i,d} -  \sum_{i \geq 3}  x_{i,i} - \sum_{i \geq 3}\sum_{d =2}^i a_{i,d}(\underline{b})} \\
&\times L^{\sum_{i \geq 3} \sum_{d = 0}^{1} (i - 2) x_{i,d} + \sum_{i \geq 3} \sum_{d = 2}^{i-1} (i -d - 1)x_{i,d}  -  \sum_{i \geq 3}\sum_{d=0}^{1}a_{i,d}(\underline{b})}.
\end{split}
\end{equation}
We summarize the above calculations in the following proposition.
\begin{prop}\label{Product -of-D(b_i)}
Let us choose $2 \leq t \leq 2r-1$ and $1 \leq j \leq r-1$.  Let $\mathcal P \in \mathcal P(r,t,j)$ be a partition $\{\mathcal I(s_u),\,1 \leq u \leq t\}$ of 
$$\{m_1,\dots,m_j,l_{j+1},\dots,l_r,m_1',\dots,m_j',l_{j+1}',\dots,l_r'\}$$
as defined in Notation \ref{notation-t-distinct-primes-j}.  Then, for $\underline{b} = (b_1,b_2,\dots,b_t),\,b_u \in \mathcal J_{\mathcal I(s_u)}$,
$$D(b_1)D(b_2)\dots D(b_{t}) \ll L^{P_1(\mathcal P)}\pi_N(x)^{P_2(\mathcal P)},$$
where
$$P_1(\mathcal P) = 2r - 2j - t - \sum_{i \geq 2} x_{i,0} + \sum_{i \geq 1} x_{i,i} - \sum_{i \geq 3} \sum_{d=0}^{1} a_{i,d}(\underline{b}),$$
and
$$P_2(\mathcal P) = 2j - t + \sum_{i \geq 1} x_{i,0} - \sum_{i \geq 2} x_{i,i}- \sum_{i \geq 3}\sum_{d=2}^{i} a_{i,d}(\underline{b}).$$
\end{prop}
\begin{proof}
The proposition follows from an application of the following identities to the bound in \eqref{prod-D(b)}.
$$\sum_{i \geq 1}\sum_{d = 0}^i x_{i,d} = t,$$
$$\sum_{i \geq 1}\sum_{d = 0}^i d x_{i,d} = 2j,$$
and
$$\sum_{i \geq 1}\sum_{d = 0}^i i x_{i,d} = 2r.$$
\end{proof}
Since $L < \pi_N(x)$, the above proposition gives us
$$D(b_1)D(b_2)\dots D(b_t) \ll_r \pi_N(x)^{2r}.$$
Note that the above is a crude estimate, sufficient for the following immediate purpose.  We have
\begin{equation}\label{second-term}
\begin{split}
&\frac{1}{(8\pi_N(x)^2L)^r}\sum_{n_1,n_2,\dots, n_j \atop {1 \leq n_i \leq \pi_N(x)}}G(n_1)G(n_2)\dots G(n_j)\sum_{l_1,l_2,\dots, l_r \atop {0 \leq l_i \leq L}}U(l_1)U(l_2)\dots U(l_r)\\
& \sum_{l_1',l_2',\dots, l_r' \atop {0 \leq l_i' \leq L}}U(l_1')U(l_2')\dots U(l_r')\\&\sum_{(s_1,s_2,\dots ,s_{t})}
\sum_{(b_1,b_2,\dots,b_{t}) \atop {b_u \in \mathcal J_{\mathcal I(s_u)}}}D(b_1)D(b_2)\dots D(b_{t})\frac{\left(s_1^{2b_1}s_2^{2b_2}\dots s_{t}^{2b_{t}}\right)^{c'} 8^{\nu(N)}}{k{N}}\\
&\ll \sum_{(s_1,s_2,\dots ,s_{t})}
\sum_{(b_1,b_2,\dots,b_{t}) \atop {b_u \in \mathcal J_{\mathcal I(s_u)}}}\frac{\left(s_1^{2b_1}s_2^{2b_2}\dots s_{t}^{2b_{t}}\right)^{c'} 8^{\nu(N)}}{k{N}}\\
&\ll_r   \sum_{(s_1,s_2,\dots ,s_{t})} \frac{\left(s_1^{2v(1)}s_2^{2v(2)}\dots s_{t}^{2v(t)}\right)^{c'} 8^{\nu(N)}}{k{N}}\\
&\ll_r  \frac{\pi_N(x)^t x^{\sum_{u=1}^t 2v(u)c'}8^{\nu(N)}}{k{N}}\\
& \ll_r \frac{  \pi_N(x)^t x^{4r\pi_N(x)c'}8^{\nu(N)}}{k{N}} \quad\quad\quad\quad (\text{by }\eqref{vuj})\\
& \ll_r \frac{x^{E_1(r)\pi_N(x)}8^{\nu(N)}}{k{N}},
\end{split}
\end{equation}
for a quantity $E_1(r)$ which depends only on $r$.
We now evaluate
\begin{equation*}
\begin{split}
&\frac{1}{(8\pi_N(x)^2L)^r}\sum_{n_1,n_2,\dots, n_j \atop {1 \leq n_i \leq \pi_N(x)}}G(n_1)G(n_2)\dots G(n_j)\sum_{l_1,l_2,\dots, l_r \atop {0 \leq l_i \leq L}}U(l_1)U(l_2)\dots U(l_r)\\
& \sum_{l_1',l_2',\dots, l_r' \atop {0 \leq l_i' \leq L}}U(l_1')U(l_2')\dots U(l_r')\sum_{(s_1,s_2,\dots ,s_{t})}
\sum_{(b_1,b_2,\dots,b_{t}) \atop {b_u \in \mathcal J_{\mathcal I(s_u)}}}\frac{D(b_1)D(b_2)\dots D(b_{t})}{s_1^{b_1}s_2^{b_2}\dots s_{t}^{b_{t}}}.
\end{split}
\end{equation*}
For an optimal evaluation of the above term, we will use the sharper estimate for $D(b_1)D(b_2)\dots D(b_t)$ obtained in Proposition \ref{Product -of-D(b_i)}.  We prove the following proposition.
\begin{prop}\label{T2-component}
Let us choose $2 \leq t \leq 2r-1$ and $1 \leq j \leq r-1$.  Let $\mathcal P \in \mathcal P(r,t,j)$ be a partition $\{\mathcal I(s_u),\,1 \leq u \leq t\}$ of 
$$\{m_1,\dots,m_j,l_{j+1},\dots,l_r,m_1',\dots,m_j',l_{j+1}',\dots,l_r'\}$$
as defined in Notation \ref{notation-t-distinct-primes-j}.  Let $L < \frac{\pi_{N}(x)}{\log \log x}.$  Then, \begin{equation*}
\begin{split}
&\frac{1}{(8 \pi_N(x)^2L)^r}\sum_{n_1,n_2,\dots, n_j \atop {1 \leq n_i \leq \pi_N(x)}}G(n_1)G(n_2)\dots G(n_j)\sum_{l_1,l_2,\dots, l_r \atop {0 \leq l_i \leq L}}U(l_1)U(l_2)\dots U(l_r)\\
&\sum_{l_1',l_2',\dots, l_r' \atop {0 \leq l_i' \leq L}}U(l_1')U(l_2')\dots U(l_r') \sum_{(s_1,s_2,\dots ,s_{t})}\sum_{(b_1,b_2,\dots,b_{t}) \atop {b_u \in \mathcal J_{\mathcal I(s_u)}}}\frac{ D(b_1)D(b_2)\dots D(b_{t})}{s_1^{b_1}s_2^{b_2}\dots s_{t}^{b_{t}}}\\
& \ll_{r} \frac{L^{P(\mathcal P)}}{\pi_{N}(x)^{Q(\mathcal P)}} \left(\begin{cases} (\log \log x)^{t} &\text{ if }t \leq j\\
1 &\text{ if } t > j\\
\end{cases}\right),
\end{split}
\end{equation*}
where
$$P(\mathcal P)=\begin{cases}
3r-2j-t-\sum_{i \geq 2} x_{i,0}+\sum_{i \geq 1} x_{i,i} &\text{ if }t \leq j\\
3r - j - 2t-\sum_{i \geq 2} x_{i,0}+\sum_{i \geq 1} x_{i,i} &\text{ if }t > j\\
\end{cases}$$
and
$$Q(\mathcal P)=\begin{cases}
3r - j - t-\sum_{i \geq 1} x_{i,0} +\sum_{i \geq 2} x_{i,i} &\text{ if }t \leq j\\
2r - 2j - \sum_{i \geq 1} x_{i,0} +\sum_{i \geq 2} x_{i,i} &\text{ if }t > j.\\
\end{cases}
$$
\end{prop}
\begin{proof}
For a tuple $\underline{b} = (b_1,b_2,\dots,b_t)$, if $a(\underline{b}) = a$, by Proposition \ref{Product -of-D(b_i)}, we have
\begin{equation}\label{prod-D(b)-a}
\begin{split}
&D(b_1)D(b_2)\dots D(b_{t})\\
&\ll L^{2r - 2j - t - \sum_{i \geq 2} x_{i,0} + \sum_{i \geq 1} x_{i,i} - \sum_{i \geq 3} \sum_{d=0}^{1} a_{i,d}(\underline{b})}\\
&\times \pi_N(x)^{2j - t + \sum_{i \geq 1} x_{i,0} - \sum_{i \geq 2} x_{i,i}- \sum_{i \geq 3}\sum_{d=2}^{i} a_{i,d}(\underline{b})}\\
&\ll L^{2r - 2j - t - \sum_{i \geq 2} x_{i,0} + \sum_{i \geq 1} x_{i,i}}\\
&\times \pi_N(x)^{2j - t + \sum_{i \geq 1} x_{i,0} - \sum_{i \geq 2} x_{i,i} }\\
\end{split}
\end{equation}
By equation \eqref{prod-D(b)-a}, we have
\begin{equation}\label{sum-combination}
\begin{split}
&\sum_{(n_1,n_2,\dots, n_j) \atop {1 \leq n_i \leq \pi_N(x)}}G(n_1)G(n_2)\dots G(n_j)\sum_{(l_1,l_2,\dots, l_r) \atop {0 \leq l_i \leq L}}U(l_1)U(l_2)\dots U(l_r)\\
& \sum_{(l_1',l_2',\dots, l_r') \atop {0 \leq l_i' \leq L}}U(l_1')U(l_2')\dots U(l_r')\sum_{(s_1,s_2,\dots ,s_{t})}
\sum_{(b_1,b_2,\dots,b_{t}) \atop {b_u \in \mathcal J_{\mathcal I(s_u)}}}\frac{D(b_1)D(b_2)\dots D(b_{t})}{s_1^{b_1}s_2^{b_2}\dots s_{t}^{b_{t}}}\\
&\ll \sum_{(n_1,n_2,\dots, n_j) \atop {1 \leq n_i \leq \pi_N(x)}}\sum_{(l_1,l_2,\dots, l_r) \atop {0 \leq l_i \leq L}} \sum_{(l_1',l_2',\dots, l_r') \atop {0 \leq l_i' \leq L}}\sum_{(s_1,s_2,\dots ,s_{t})}
\sum_{(b_1,b_2,\dots,b_{t}) \atop {b_u \in \mathcal J_{\mathcal I(s_u)}}}\frac{D(b_1)D(b_2)\dots D(b_{t})}{s_1^{b_1}s_2^{b_2}\dots s_{t}^{b_{t}}}\\
&\ll L^{2r - 2j - t - \sum_{i \geq 2} x_{i,0} + \sum_{i \geq 1} x_{i,i} }\pi_N(x)^{2j - t + \sum_{i \geq 1} x_{i,0} - \sum_{i \geq 2} x_{i,i}}\\
& \times \sum_{(n_1,n_2,\dots, n_j) \atop {1 \leq n_i \leq \pi_N(x)}}\sum_{(l_1,l_2,\dots, l_r) \atop {0 \leq l_i \leq L}} \sum_{(l_1',l_2',\dots, l_r') \atop {0 \leq l_i' \leq L}} \left(\sum_{(b_1,b_2,\dots,b_{t}) \atop {b_u \in \mathcal J_{\mathcal I(s_u)}}}\frac{1}{s_1^{b_1}s_2^{b_2}\dots s_{t}^{b_{t}}}\right)\\
&\ll L^{2r - 2j - t - \sum_{i \geq 2} x_{i,0} + \sum_{i \geq 1} x_{i,i} }\pi_N(x)^{2j - t + \sum_{i \geq 1} x_{i,0} - \sum_{i \geq 2} x_{i,i}}\\
&\times  \sum_{a = 0}^t \pi_N(x)^{a} (\log \log x)^{t-a} \\
&\sum_{(n_1,n_2,\dots, n_j) \atop {1 \leq n_i \leq \pi_N(x)}}\sum_{(l_1,l_2,\dots, l_r) \atop {0 \leq l_i \leq L}} \sum_{(l_1',l_2',\dots, l_r') \atop {0 \leq l_i' \leq L}} \sum_{(b_1,b_2,\dots,b_{t}) \atop {b_u \in \mathcal J_{\mathcal I(s_u)} \atop {a(\underline{b}) = a}}}1\\
&\ll  L^{2r - 2j - t - \sum_{i \geq 2} x_{i,0} + \sum_{i \geq 1} x_{i,i}  }\\
&\times \pi_N(x)^{2j - t + \sum_{i \geq 1} x_{i,0} - \sum_{i \geq 2} x_{i,i}}\\
&\sum_{a = 0}^t \pi_N(x)^a (\log \log x)^{t-a} E_a(x),\\
\end{split}
\end{equation}
where
$E_a(x)$ denotes the number of $(j+2r)$- tuples 
$$\left\{(n_1,\dots, n_j, l_1,\dots, l_r,l_1',\dots, l_r'):\,1 \leq n_i \leq \pi_N(x),\,0 \leq l_i,\,l_i' \leq L\right\}$$
such that 
$$\#\left\{1 \leq u \leq t:\,0 \in \mathcal J_{\mathcal I(s_u)}\right\} = a.$$
For $1 \leq u \leq t$ such that 
$$\mathcal I(s_u) = \{m_1,m_2,\dots,m_{d_u},l_{d_u + 1},\dots l_{i_u}\},$$
let $G(s_u)$ denote the set of corresponding $n_i$'s and $l_i$'s which  create the set $\mathcal I(s_u)$.  That is, 
$$G(s_u) = \{n_1,n_2,\dots,n_{d_u},l_1,l_2,\dots,l_{i_u}\}.$$
In case $d_u = 0$, $G(s_u) = \mathcal I(s_u) = \{l_1,l_2,\dots,l_{i_u}\}$.
For each $1 \leq u \leq t$, let $$Z_u := 
\begin{cases} 
\#\{(n_1,n_2,\dots,n_{d_u},l_1,l_2,\dots,l_{i_u}):\,0 \in \mathcal J_{\mathcal I(s_u)}\} &\text{ if }d_u \geq 1,\\
\#\{(l_1,l_2,\dots,l_{i_u}):\,0 \in \mathcal J_{\mathcal I(s_u)}\} &\text{ if }d_u = 0.
\end{cases}.
$$
We first note that, if $0 \in \mathcal J_{\mathcal I(s_u)},$ then at least one element of the set is determined by other elements of the set.  Thus, we have the inequalities,
$$Z_u \ll \pi_N(x)^{d_u - 1 }L^{i_u }.$$
as well as
$$Z_u \ll \pi_N(x)^{d_u }L^{i_u - 1}.$$
Thus, we have
\begin{equation}\label{E-A}
E_a(x) \ll \begin{cases}
\pi_N(x)^{j-a}L ^{2r} &\text{ if }a \leq j\\
L^{2r - (a-j)} &\text{ if }a > j.\\
\end{cases}
\end{equation}
By equations \eqref{sum-combination} and \eqref{E-A}, we have
\begin{equation}\label{interchange}
\begin{split}
&\frac{1}{(8 \pi_N(x)^2L)^r}\sum_{(n_1,n_2,\dots, n_j) \atop {1 \leq n_i \leq \pi_N(x)}}G(n_1)G(n_2)\dots G(n_j)\sum_{(l_1,l_2,\dots, l_r) \atop {0 \leq l_i \leq L}}U(l_1)U(l_2)\dots U(l_r)\\
& \sum_{(l_1',l_2',\dots, l_r') \atop {0 \leq l_i' \leq L}}U(l_1')U(l_2')\dots U(l_r')\sum_{(s_1,s_2,\dots ,s_{t})}
\sum_{(b_1,b_2,\dots,b_{t}) \atop {b_u \in \mathcal J_{\mathcal I(s_u)}}}\frac{D(b_1)D(b_2)\dots D(b_{t})}{s_1^{b_1}s_2^{b_2}\dots s_{t}^{b_{t}}}\\
&\ll  \frac{1}{\pi_N(x)^{2r}L^r}L^{2r - 2j - t - \sum_{i \geq 2} x_{i,0} + \sum_{i \geq 1} x_{i,i}  } \pi_N(x)^{2j - t + \sum_{i \geq 1} x_{i,0} - \sum_{i \geq 2} x_{i,i}}\\
&\times \sum_{a = 0}^t \pi_N(x)^a (\log \log x)^{t-a} \begin{cases}
\pi_N(x)^{j-a}L ^{2r} &\text{ if }a \leq j\\
L^{2r - (a-j)} &\text{ if }a > j.\\
\end{cases}\\
&\ll    \frac{1}{\pi_N(x)^{2r}L^r} L^{2r - 2j - t - \sum_{i \geq 2} x_{i,0} + \sum_{i \geq 1} x_{i,i}  } \pi_N(x)^{2j - t + \sum_{i \geq 1} x_{i,0} - \sum_{i \geq 2} x_{i,i}} \\
& \times \begin{cases} \sum_{a = 0}^t \pi_N(x)^ a (\log \log x)^{t-a} \pi_N(x)^{j-a}L ^{2r} &\text{ if }t \leq j\\
 \sum_{a = 0}^j \pi_N(x)^ a (\log \log x)^{t-a}  \pi_N(x)^{j-a}L^{2r} + \sum_{a = j+1}^ t \pi_N(x)^ a (\log \log x)^{t-a}L^{2r - (a-j)}&\text{ if }t > j\\
 \end{cases}\\
&\ll \frac{L^{r - 2j - t - \sum_{i \geq 2} x_{i,0} + \sum_{i \geq 1} x_{i,i}}}{\pi_N(x)^{2r - 2j + t - \sum_{i \geq 1} x_{i,0} + \sum_{i \geq 2} x_{i,i}}} \times \begin{cases} \pi_N(x)^j L^{2r} (\log \log x)^t &\text{ if }t \leq j\\
\pi_N(x)^j L^{2r} (\log \log x)^t + \left(\frac{\pi_N(x)}{L}\right)^t L^{2r + j}&\text{ if }t > j\\
\end{cases}\\
&= \begin{cases} 
\frac{L^{3r - 2j - t - \sum_{i \geq 2} x_{i,0} + \sum_{i \geq 1} x_{i,i}}(\log \log x)^t}{\pi_N(x)^{2r - 3j + t - \sum_{i \geq 1} x_{i,0} + \sum_{i \geq 2} x_{i,i}}} &\text{ if }t \leq j\\
\\
\frac{L^{3r - j - 2t - \sum_{i \geq 2} x_{i,0} + \sum_{i \geq 1} x_{i,i}}}{\pi_N(x)^{2r - 2j  - \sum_{i \geq 1} x_{i,0} + \sum_{i \geq 2} x_{i,i}}} &\text{ if }t > j.\\
\end{cases}
\end{split}
\end{equation}
This proves the proposition.
\end{proof}
\begin{remark}\label{j=r}
Proposition \ref{T2-component} is also applicable to the case when $j = r$.  In this case, we have to consider partitions of
$$\{m_1,m_2,\dots,m_r,m_1',m_2',\dots,m_r'\}.$$
We use Notation \ref{notation-t-distinct-primes-j} for the special case when $i_u = d_u$ for every $1 \leq u \leq t$.  This observation will be useful in the next section.
\end{remark}

\subsection{$\langle T_{3,r}(\underline{p,q}) \rangle$}\label{rM-3}
The goal of this section is to evaluate 
\begin{equation*}
\begin{split}
& \left \langle \frac{1}{(8 \pi_N(x)^2L)^r}\sum_{(p_1,q_1) \atop {p_1 \neq q_1 \leq x \atop{(p_1,N) = (q_1,N) = 1}}}\sum_{(p_2,q_2) \atop {p_2 \neq q_2 \leq x \atop{(p_2,N) = (q_2,N) = 1}}}\dots \sum_{(p_r,q_r) \atop {p_r \neq q_r \leq x \atop{(p_r,N) = (q_r,N) = 1}}}T_{3,r}(\underline{p,q})\right \rangle\\
& =  \frac{2^r}{(8 \pi_N(x)^2L)^r}\sum_{(p_1,q_1) \atop {p_1 \neq q_1 \leq x \atop{(p_1,N) = (q_1,N) = 1}}}\sum_{(p_2,q_2) \atop {p_2 \neq q_2 \leq x \atop{(p_2,N) = (q_2,N) = 1}}}\dots \sum_{(p_r,q_r) \atop {p_r \neq q_r \leq x \atop{(p_r,N) = (q_r,N) = 1}}}\\
&\sum_{n_1,n_2,\dots, n_r \atop {1 \leq n_i \leq \pi_N(x)}}G(n_1)G(n_2)\dots G(n_r) \sum_{l_1,l_2,\dots, l_r \atop {0 \leq l_i \leq L}}U(l_1)U(l_2)\dots U(l_r) \sum_{l_1',l_2',\dots, l_r' \atop {0 \leq l_i' \leq L}}U(l_1')U(l_2')\dots U(l_r')\\
&\frac{1}{\mathcal F_{N,k}}\sum_{f \in \mathcal F_{N,k}}
 \begin{cases}
\left(a_f(p_i^{2l_i+2n_i}) + a_f(p_i^{2l_i-2n_i})\right)\left(a_f(q_i^{2l'_i+2n_i}) + a_f(q_i^{2l_i'-2n_i})\right) &\text{ if }l_i,l_i' \geq n_i\\
\left(a_f(p_i^{2l_i+2n_i}) - a_f(p^{2n_i-2l_i-2})\right)\left(a_f(q_i^{2l_i'+2n_i}) + a_f(q_i^{2l_i'-2n_i})\right) &\text{ if }l_i<n_i \leq l_i'\\
\left(a_f(p_i^{2l_i+2n_i}) + a_f(p_i^{2l_i-2n_i})\right)\left(a_f(q_i^{2l_i'+2n_i}) - a_f(q_i^{2n_i-2l_i'-2})\right) &\text{ if }l_i'<n_i \leq l_i\\
\left(a_f(p_i^{2l_i+2n_i}) - a_f(p_i^{2n_i-2l_i-2})\right)\left(a_f(q_i^{2l_i'+2n_i}) - a_f(q_i^{2n_i-2l_i'-2})\right) &\text{ if }l_i,l_i' < n_i.
\end{cases}\\
 \end{split}
 \end{equation*}
An evaluation of the above term yields the ``leading"  term 
$$\left(\frac{T(g,\rho)}{4L}\right)^r$$
in Theorem \ref{higher-moments-bounds} when we consider the case when all the primes $p_1,p_2,\dots,p_r,q_1,q_2,\dots,q_r$ are distinct.  We also obtain an upper bound for the components of the above sum when the number of distinct primes is strictly less than $2r$.  This is encoded in the following proposition.
\begin{prop}\label{2r-distinct-primes-T3}
Let $\sum_{\underline{(p,q)}}^{(2r)}$ denote a sum over all $2r$-tuples of distinct primes $p_1,q_1,p_2,q_2,\dots,p_r,q_r  \leq x$ coprime to $N$.  We choose $L = L(x) \to \infty$ such that $L(x) < \frac{\pi_N(x)}{\log \log x}$.  Then,
\begin{equation*}
\begin{split}
& \frac{1}{(8 \pi_N(x)^2L)^r} \sum_{\underline{(p,q)}}^{(2r)}2^r \sum_{n_1,n_2,\dots, n_r \atop {1 \leq n_i \leq \pi_N(x)}}G(n_1)G(n_2)\dots G(n_r)\\
&\sum_{l_1,l_2,\dots, l_r \atop {0 \leq l_i \leq L}}U(l_1)U(l_2)\dots U(l_r) \sum_{l_1',l_2',\dots, l_r' \atop {0 \leq l_i' \leq L}}U(l_1')U(l_2')\dots U(l_r') \\
&\frac{1}{|\mathcal F_{N,k}|} \sum_{f \in \mathcal F_{N,k}}\prod_{i=1}^r I(p_i,q_i,n_i,l_i,l_i')\\
& =   \left(\frac{T(g,\rho)}{4L}\right)^r + \O_r\left(\frac{1}{L}\right) + \O_r\left(\frac{L^{1/2}\log \log x}{\pi_N(x)^{1/2}}\right) + \O_r\left(\frac{x^{E_3(r) \pi_N(x)}8^{\nu(N)}}{k N}\right).
\end{split}
\end{equation*}
Here, 
$$T(g,\rho) = \sum_{l_1 \geq 1}(U(l_1) - U(l_1 - 1))^2G(l_1),$$ and $E_3(r)$ is a positive real number that depends only on $r$.
\end{prop}
\begin{proof}
The innermost part of each term in the above sum is of the form
\begin{equation*}
\begin{split}
&\pm a_f(p_1^{2m_1}) a_f(p_2^{2m_2})\dots a_f(p_j^{2m_r}) a_f(q_1^{2m'_1}) a_f(q_2^{2m'_2})\dots a_f(q_j^{2m'_r})\\
\end{split}
\end{equation*}
where, for $1 \leq i \leq r$, as in equation \eqref{m_i},
\begin{equation*}
\begin{split}
&m_i = l_i + n_i \text{ or }l_i - n_i\, (\text{if }l_i \geq n_i) \text{ or }n_i - l_i - 1 \,(\text{if }l_i < n_i), \text{ and }\\
&m'_i = l'_i + n_i \text{ or }l'_i - n_i \,(\text{if }l'_i \geq n_i) \text{ or }n_i - l'_i - 1 \,(\text{if }l'_i < n_i).
\end{split}
\end{equation*}
We have to estimate sums of the form
\begin{equation}\label{Lemma-above-part-T3}
\begin{split}
& \frac{2^r}{(8 \pi_N(x)^2L)^r} \sum_{\underline{(p,q)}}^{(2r)} \sum_{n_1,n_2,\dots, n_r \atop {1 \leq n_i \leq \pi_N(x)}}G(n_1)G(n_2)\dots G(n_r)\\
&\sum_{l_1,l_2,\dots, l_r \atop {0 \leq l_i \leq L}}U(l_1)U(l_2)\dots U(l_r) \sum_{l_1',l_2',\dots, l_r' \atop {0 \leq l_i' \leq L}}U(l_1')U(l_2')\dots U(l_r') \\
&\frac{1}{|\mathcal F_{N,k}|} \sum_{f \in \mathcal F_{N,k}} a_f(p_1^{2m_1}) a_f(p_2^{2m_2})\dots a_f(p_j^{2m_j}) a_f(q_1^{2m'_1}) a_f(q_2^{2m'_2})\dots a_f(q_j^{2m'_r}).\\
\end{split}
\end{equation}
We consider the $3r$-tuples $A_{a,b}(\underline{n},\underline{l},\underline{l'})$ 
$$\{(n_1,n_2,\dots,n_r,l_1,l_2,\dots,l_r,l_1',l_2',\dots,l_r'):\,1 \leq n_i \leq \pi_N(x),\,0 \leq l_i,l_i' \leq L\}$$
such that $m_i = m'_i = 0$ for each $1 \leq i \leq r$.

In this case, 
\begin{itemize}
\item $n_i = l_i$ or $n_i = l_{i} + 1$ for each $1 \leq i \leq r$, and
\item $n_i = l'_i$ or $n_i = l'_{i} + 1$ for each $1 \leq i \leq r$.
\item Thus, $l'_i = l_i$ or $l_i+1$ or $l_i - 1$.
\end{itemize}
Thus, the contribution to \eqref{Lemma-above-part-T3} from those terms for which $m_i = m'_i = 0$ for each $1 \leq i \leq r$ is
\begin{equation}\label{all-powers-0}
\begin{split}
& \frac{2^r}{(8 \pi_N(x)^2L)^r}  \pi_N(x)(\pi_N(x) - 1)(\pi_N(x) - 2)\dots (\pi_N(x) - (2r-1))\\
&\times \left( \sum_{l_1= 0}^L U(l_1)^2 G(l_1) -2\sum_{l_1 = 0}^{L-1} U(l_1)U(l_1+1)G(l_1 + 1)  + \sum_{l_1 = 0}^L U(l_1)^2 G(l_1 + 1)\right)^r\\
&= \frac{1}{ \pi_N(x)^{2r} (4L)^r} \pi_N(x)(\pi_N(x) - 1)(\pi_N(x) - 2)\dots (\pi_N(x) - (2r-1)) \left( U(0)^2G(0) + T(g,\rho)\right)^r\\
\end{split}
\end{equation}
where
$$T(g,\rho) = \sum_{l_1 \geq 1}(U(l_1) - U(l_1 - 1))^2G(l_1).$$ 
Note that
$T(g,\rho) \ll L$.  Thus,
\begin{equation}\label{all-powers-0-more}
\begin{split}
& \frac{1}{ \pi_N(x)^{2r} (4L)^r} \pi_N(x)(\pi_N(x) - 1)(\pi_N(x) - 2)\dots (\pi_N(x) - (2r-1)) \left( U(0)^2G(0) + T(g,\rho)\right)^r\\
&= \frac{1}{ \pi_N(x)^{2r} (4L)^r} \left[\pi_N(x)^{2r} + \O_r\left(\pi_N(x)^{2r - 1}\right)\right]\left[T(g,\rho) + \O(1)\right]^r\\
& = \left[\frac{1}{(4L)^r} + \O_r\left(\frac{1}{L^r \pi_N(x)} \right)\right]\left[T(g,\rho)^r + \O_r(L^{r-1})\right]\\
&= \left(\frac{T(g,\rho)}{4L}\right)^r + \O_r\left(\frac{1}{ \pi_N(x)} \right) + \O_r\left(\frac{1}{ L} \right) + \O_r\left(\frac{1}{ L\pi_N(x)} \right)\\
&=  \left(\frac{T(g,\rho)}{4L}\right)^r + \O_r\left(\frac{1}{ L} \right).
\end{split}
\end{equation}
To obtain the contribution to \eqref{Lemma-above-part-T3} from terms in which $m_1,\dots,m_r,m_1',\dots,m_r'$ are not all zero, we apply Proposition \ref{sums-distinct-primes} (more precisely, equation \eqref{this-lemma-other}).  The analysis is very similar to the proof of Lemma \ref{2r-distinct-primes-T2}.  Let $a$ denote the number of elements in the set $\{m_1,\dots,m_r,m_1',\dots,m_r'\}$ which are equal to zero.  Here, $0 \leq a \leq 2r-1$ (since not all elements in the set are zero).
Denote $A_a(\underline{n},\underline{l},\underline{l'})$ to be the number of $3r$-tuples
$$(n_1,\dots,n_r,l_1,\dots,l_r,l_1',\dots,l_r')$$
such that the number of elements in the set $\{m_1,\dots,m_r,m_1',\dots,m_r'\}$ which are equal to zero.

With the notations $a_1$ and $a_2$ as in the proof of Lemma \ref{2r-distinct-primes-T2}, we observe that $a = a_1 + 2a_2$, and 
$$A_a(\underline{n},\underline{l},\underline{l'}) \ll \pi_N(x)^{r - (a_1 + a_2)}L^{2r - a_2} = \pi_N(x)^{r - a+a_2}L^{2r - a_2}.$$
Following an analysis similar to that in the proof of Lemma \ref{2r-distinct-primes-T2}, the contribution to 
\begin{equation*}
\begin{split}
& \frac{1}{(8 \pi_N(x)^2L)^r} \sum_{\underline{(p,q)}}^{(2r)}2^r \sum_{n_1,n_2,\dots, n_r \atop {1 \leq n_i \leq \pi_N(x)}}G(n_1)G(n_2)\dots G(n_r)\\
&\sum_{l_1,l_2,\dots, l_r \atop {0 \leq l_i \leq L}}U(l_1)U(l_2)\dots U(l_r) \sum_{l_1',l_2',\dots, l_r' \atop {0 \leq l_i' \leq L}}U(l_1')U(l_2')\dots U(l_r') \\
&\frac{1}{|\mathcal F_{N,k}|} \sum_{f \in \mathcal F_{N,k}}\prod_{i=1}^r I(p_i,q_i,n_i,l_i,l_i')\\
\end{split}
\end{equation*}
from those terms in which $m_1,\dots,m_r,m_1',\dots,m_r'$ are not all zero is
\begin{equation}\label{not-all-zero}
\begin{split}
&\ll \frac{1}{\pi_N(x)^{2r}L^r}\sum_{a = 0}^{2r - 1} \pi_N(x)^a (\log \log x)^{2r - a}\pi_N(x)^{r-a}L^{2r}\sum_{a_2 = 0}^{[a/2]}\left(\frac{\pi_N(x)}{L}\right)^{a_2}\\
&+ \frac{x^{E_3(r) \pi_N(x)}8^{\nu(N)}}{k N}\\
&\ll \frac{L^{1/2}\log \log x}{\pi_N(x)^{1/2}} + \frac{x^{E_3(r) \pi_N(x)}8^{\nu(N)}}{k N}.
\end{split}
\end{equation}
By \eqref{all-powers-0}, \eqref{all-powers-0-more} and \eqref{not-all-zero}, the proposition is proved.
\end{proof}
We now evaluate the part of
$$\left \langle \frac{1}{(8 \pi_N(x)^2L)^r}\sum_{(p_1,q_1) \atop {p_1 \neq q_1 \leq x \atop{(p_1,N) = (q_1,N) = 1}}}\sum_{(p_2,q_2) \atop {p_2 \neq q_2 \leq x \atop{(p_2,N) = (q_2,N) = 1}}}\dots \sum_{(p_r,q_r) \atop {p_r \neq q_r \leq x \atop{(p_r,N) = (q_r,N) = 1}}}T_{3,r}(\underline{p,q})\right \rangle$$
in which the primes $p_1,\dots,p_r,q_1,\dots,q_r$ are not all distinct.  To do this, for each $2 \leq t \leq 2r-1$, we consider the case such that the $2r$-tuple $(p_1,\dots,p_r,q_1,\dots,q_r)$ has $t$ distinct primes.  We use Notation \ref{notation-t-distinct-primes-j} and Remark \ref{j=r} for the case when $j = r$, that is, $d_u = i_u$ for each $1 \leq u \leq r$.  

\subsection{Proof of Theorem \ref{higher-moments-bounds}}\label{Thm2-proof}
We are now ready to put together estimates from Sections \ref{rM-1}, \ref{rM-2} and \ref{rM-3} to obtain Theorem \ref{higher-moments-bounds}.

By equations \eqref{sum-hm-0} and \eqref{sum-hm}, $R_2(g,\rho)(f)^r$ is broken down into three components as follows.
\begin{equation*}
\begin{split}
& R_2(g,\rho)(f)^r \\
&= \frac{1}{(8 \pi_N(x)^2L)^r}\sum_{(p_1,q_1) \atop {p_1 \neq q_1 \leq x \atop{(p_1,N) = (q_1,N) = 1}}}\sum_{(p_2,q_2) \atop {p_2 \neq q_2 \leq x \atop{(p_2,N) = (q_2,N) = 1}}}\dots \sum_{(p_r,q_r) \atop {p_r \neq q_r \leq x \atop{(p_r,N) = (q_r,N) = 1}}}\sum_{w=1}^3T_{w,r}(\underline{p,q}),\\
\end{split}
\end{equation*}
where $T_{w,r}$'s are defined respectively in \eqref{T1}, \eqref{T2} and \eqref{T3}.  

By Proposition \ref{T1rpq}(b), we get
\begin{equation*}
\begin{split}
&\frac{1}{(8 \pi_N(x)^2L)^r}\sum_{(p_1,q_1) \atop {p_1 \neq q_1 \leq x \atop{(p_1,N) = (q_1,N) = 1}}}\sum_{(p_2,q_2) \atop {p_2 \neq q_2 \leq x \atop{(p_2,N) = (q_2,N) = 1}}}\dots \sum_{(p_r,q_r) \atop {p_r \neq q_r \leq x \atop{(p_r,N) = (q_r,N) = 1}}}T_{1,r}(\underline{p,q})\\
&\ll_r \sum_{t=2}^{2r} \frac{L^{3r - 2t}}{\pi_N(x)^{2r - t}} +  \frac{x^{4Lrc'} 8^{\nu(N)}}{k N} L^{3r}.
\end{split}
\end{equation*}
By Lemma \ref{2r-distinct-primes-T2}, equation \eqref{second-term} and Proposition \ref{T2-component}, we have
\begin{equation*}
\begin{split}
 &\frac{1}{(8 \pi_N(x)^2L)^r} \left\langle \sum_{(p_1,q_1) \atop {p_1 \neq q_1 \leq x \atop{(p_1,N) = (q_1,N) = 1}}}\sum_{(p_2,q_2) \atop {p_2 \neq q_2 \leq x \atop{(p_2,N) = (q_2,N) = 1}}}\dots \sum_{(p_r,q_r) \atop {p_r \neq q_r \leq x \atop{(p_r,N) = (q_r,N) = 1}}}T_{2,r}(\underline{p,q})\right \rangle\\
 &\ll_r \frac{1}{L} +\sum_{1 \leq j \leq r-1} \sum_{t = 2}^{2r-1} \sum_{\mathcal P \in \mathcal P(r,t,j)} \frac{L^{P(\mathcal P)}}{\pi_N(x)^{Q(\mathcal P)}}\left(\begin{cases} (\log \log x)^{t} &\text{ if }t \leq j\\
1 &\text{ if } t > j\\
\end{cases}\right)\\
 &+ \frac{x^{\delta(r) + E_2(r)}\pi_N(x)8^{\nu(N)}}{kN},\\
 \end{split}
 \end{equation*}
 $\mathcal P(r,t,j)$ is as defined in Notation \ref{notation-r-moment},
$$P(\mathcal P)=\begin{cases}
3r-2j-t-\sum_{i \geq 2} x_{i,0}+\sum_{i \geq 1} x_{i,i} &\text{ if }t \leq j\\
3r - j - 2t-\sum_{i \geq 2} x_{i,0}+\sum_{i \geq 1} x_{i,i} &\text{ if }t > j,\\
\end{cases}$$
$$Q(\mathcal P)=\begin{cases}
3r - j - t-\sum_{i \geq 1} x_{i,0} +\sum_{i \geq 2} x_{i,i} &\text{ if }t \leq j\\
2r - 2j - \sum_{i \geq 1} x_{i,0} +\sum_{i \geq 2} x_{i,i} &\text{ if }t > j\\
\end{cases}
$$
and $E_2(r)$ is a quantity that depends only on $r$.

By Remark \ref{j=r} and Proposition \ref{2r-distinct-primes-T3},
\begin{equation*}
\begin{split}
 &\frac{1}{(8 \pi_N(x)^2L)^r} \left\langle \sum_{(p_1,q_1) \atop {p_1 \neq q_1 \leq x \atop{(p_1,N) = (q_1,N) = 1}}}\sum_{(p_2,q_2) \atop {p_2 \neq q_2 \leq x \atop{(p_2,N) = (q_2,N) = 1}}}\dots \sum_{(p_r,q_r) \atop {p_r \neq q_r \leq x \atop{(p_r,N) = (q_r,N) = 1}}}T_{3,r}(\underline{p,q})\right \rangle\\
 & =  \left(\frac{T(g,\rho)}{4L}\right)^r+\O_r\left(\frac{1}{ L} \right) + \O_r\left(\frac{L^{1/2}\log \log x}{\pi_N(x)^{1/2}}\right)\\
 &  + \O_r\left(\sum_{t = 2}^{2r-1} \sum_{\mathcal P \in \mathcal P(r,t,r)}  \frac{L^{P(\mathcal P)}}{\pi_N(x)^{Q(\mathcal P)}} \begin{cases} (\log \log x)^{t} &\text{ if }t \leq j\\
1 &\text{ if } t > j\\
\end{cases}\right)\\
 & + \O_r\left(\frac{x^{E_3(r) \pi_N(x)}8^{\nu(N)}}{k N}\right).
 \end{split}
 \end{equation*}

Combining the above estimates, we get
\begin{equation*}
\begin{split}
& \frac{1}{|\mathcal F_{N,k}|} \sum_{f \in \mathcal F_{N,k}}R_2(g,\rho)(f)^r -  \left(\frac{T(g,\rho)}{(4L)}\right)^r\\
& \ll_r \sum_{t=2}^{2r} \frac{L^{3r - 2t}}{\pi_N(x)^{2r - t}} +\frac{L^{1/2}\log \log x}{\pi_N(x)^{1/2}}\\
&+ \sum_{1 \leq j \leq r} \sum_{t = 2}^{2r-1} \sum_{\mathcal P \in \mathcal P(r,t,j)} \frac{L^{P(\mathcal P)}}{\pi_N(x)^{Q(\mathcal P)}}\left(\begin{cases} (\log \log x)^{t} &\text{ if }t \leq j\\
1 &\text{ if } t > j\\
\end{cases}\right)\\
& + \frac{x^{E(r)\pi_N(x)}8^{\nu(N)}}{kN},
\end{split}
\end{equation*}
where $E(r)$ is a quantity that depends only on $r$.  This proves Theorem \ref{higher-moments-bounds}.

%
%


\section{The second and third moments of $R_2(g,\rho)(f)$}\label{2-3M}

We conclude this article by recording the application of Theorem \ref{higher-moments-bounds} to two specific values of $r$.  The second moment is of particular interest.

\subsection{Proof of Theorem \ref{variance_main_theorem_squarefree}}\label{2M-V}
Let us consider $r=2$.  Then, $t = 2,3$.  

{\bf Case 1.} Let us first consider the case when $t = 2$.
In this case, 
\begin{equation}\label{ER1}
\sum_{t=2}^{2r} \frac{L^{3r - 2t}}{\pi_N(x)^{2r - t}} = \frac{L^2}{\pi_N(x)^2} + \frac{1}{\pi_N(x)} + \frac{1}{L^2}.
\end{equation}
We now address the sum
\begin{equation}\label{sum-partition}
\begin{split}
& \sum_{1 \leq j \leq r} \sum_{t = 2}^{2r-1} \sum_{\mathcal P \in \mathcal P(r,t,j)} \frac{L^{P(\mathcal P)}}{\pi_N(x)^{Q(\mathcal P)}}\left(\begin{cases} (\log \log x)^{t} &\text{ if }t \leq j\\
1 &\text{ if } t > j\\
\end{cases}\right)\\
& = \begin{cases} 
\frac{L^{3r - 2j - t - \sum_{i \geq 2} x_{i,0} + \sum_{i \geq 1} x_{i,i}}(\log \log x)^t}{\pi_N(x)^{2r - 3j + t - \sum_{i \geq 1} x_{i,0} + \sum_{i \geq 2} x_{i,i}}} &\text{ if }t \leq j\\
\\
\frac{L^{3r - j - 2t - \sum_{i \geq 2} x_{i,0} + \sum_{i \geq 1} x_{i,i}}}{\pi_N(x)^{2r - 2j  - \sum_{i \geq 1} x_{i,0} + \sum_{i \geq 2} x_{i,i}}} &\text{ if }t > j.\\
\end{cases}.
\end{split}
\end{equation}

{\bf Case 1(a).} Let us take $j = 1$.
For any partition $\mathcal P \in \mathcal P(2,2,1)$, $x_{2,1} = 2$ and $x_{i,d} = 0$ for all $(i,d) \neq (2,1)$.  Thus, its contribution to \eqref{sum-partition} is
$$\ll \frac{L}{\pi_N(x)^2}.$$
{\bf Case 1(b).} Let us take $j = 2$.  For any partition $\mathcal P \in \mathcal P(2,2,2)$, $x_{2,2} = 2$ and $x_{i,d} = 0$ for all $(i,d) \neq (2,2)$.  Thus, its contribution to \eqref{sum-partition} is
$$\frac{L^2 (\log \log x)^2}{\pi_N(x)^2}.$$

{\bf Case 2.} We now consider the case $t = 3$.

{\bf Case 2(a).} Let us take $j = 1$.
For any partition $\mathcal P \in \mathcal P(2,3,1)$, $x_{2,1} = x_{1,1} = x_{1,0} = 1$ and $x_{i,d} = 0$ for all other pairs $(i,d)$.  Thus, its contribution to \eqref{sum-partition} is
$$\ll \frac{1}{\pi_N(x)}.$$  
{\bf Case 2(b).} Let us take $j = 2$.
For any partition $\mathcal P \in \mathcal P(2,3,2)$, $x_{2,2} = 1$, $ x_{1,1} = 2$ and $x_{i,d} = 0$ for all other pairs $(i,d)$.  Thus, its contribution to \eqref{sum-partition} is
$$\ll \frac{L}{\pi_N(x)}.$$  
Combining the estimates in \eqref{ER1} and all the above components of \eqref{sum-partition}, by Theorem \ref{higher-moments-bounds}, we get
\begin{equation*}
\begin{split}
& \frac{1}{|\mathcal F_{N,k}|} \sum_{f \in \mathcal F_{N,k}}R_2(g,\rho)(f)^2 -  \left(\frac{T(g,\rho)}{(4L)}\right)^2\\
&\ll \frac{L^{1/2}\log \log x}{\pi_N(x)^{1/2}} + \frac{L^2}{\pi_N(x)^2} + \frac{1}{L}  + \frac{L^2 (\log \log x)^2}{\pi_N(x)^2} + \frac{L}{\pi_N(x)} +  \frac{x^{E(2)\pi_N(x)}8^{\nu(N)}}{kN}\\
\end{split}
\end{equation*}
Combining this with Theorem \ref{pc-error},
\begin{equation*}
\begin{split}
&\frac{1}{|\mathcal F_{N,k}|} \sum_{f \in \mathcal F_{N,k}} \left(R_2 (g,\rho) (f) - \frac{T(g,\rho)}{4 L}\right)^2\\
&= \frac{1}{|\mathcal F_{N,k}|} \sum_{f \in \mathcal F_{N,k}} R_2 (g,\rho) (f)^2 + \left(\frac{T(g,\rho)}{4 L}\right)^2 - \frac{2T(g,\rho)}{4 L}\left(\frac{1}{|\mathcal F_{N,k}|} \sum_{f \in \mathcal F_{N,k}} R_2 (g,\rho) (f)\right)\\
&= \left(\frac{T(g,\rho)}{4 L}\right)^2 +  \O\left(\frac{L^2}{\pi_N(x)^2}\right)  + \O\left(\frac{1}{L} \right) + \O\left(\frac{L^2 (\log \log x)^2}{\pi_N(x)^2}\right) + \O\left(\frac{L}{\pi_N(x)}\right)\\
&+ \left(\frac{T(g,\rho)}{4 L}\right)^2 - \frac{2T(g,\rho)}{4 L}\left[\frac{T(g,\rho)}{4 L} + \O\left(\frac{1}{L}\right) + \O\left(\frac{L(\log \log x)^2}{\pi_N(x)}\right)\right] + \O\left( \frac{x^{C\pi_N(x)}8^{\nu(N)}}{kN}\right)\\
&=   \O\left(\frac{1}{L} \right) + \O\left(\frac{L^2 (\log \log x)^2}{\pi_N(x)^2}\right) + \O\left(\frac{L(\log \log x)^2}{\pi_N(x)}\right) + \O\left( \frac{L^{1/2}\log \log x}{\pi_N(x)^{1/2}}\right)\\
& + \O\left(\frac{x^{C\pi_N(x)}8^{\nu(N)}}{kN}\right).
\end{split}
\end{equation*}

This proves Theorem \ref{variance_main_theorem_squarefree}(a).  

As noted in Theorem \ref{pc-error},
$$\frac{T(g,\rho)}{4 L} \sim  A^2\widehat{g}(0){\rho \ast \rho}(0) \text{ as } x\to \infty.$$
We now choose $L(x) \to \infty$ such that $$ L = \o\left(\frac{\pi_N(x)}{(\log \log x)^2}\right).$$
Then,
$$\frac{1}{L} + \frac{L^2 (\log \log x)^2}{\pi_N(x)^2} + \frac{L(\log \log x)^2}{\pi_N(x)} + \frac{L^{1/2}\log \log x}{\pi_N(x)^{1/2}} \to 0.$$
Further, as in Theorem \ref{pc-error}, if we consider families $\mathcal F_{N,k}$ with levels $N = N(x)$ and even weights $k = k(x)$ such that $$ \frac{\log \left(k{N}/8^{\nu(N)}\right)}{x} \to \infty \text{ as }x \to \infty,$$ 
then, by Lemma \ref{x-powers},
$$ \frac{x^{C\pi_N(x)}8^{\nu(N)}}{kN} \to 0.$$
Thus, we derive Theorem \ref{variance_main_theorem_squarefree}(b).

\subsection{Proof of Theorem \ref{r=3 and more}}\label{3M}

We now apply Theorem \ref{higher-moments-bounds} to $r = 3$; for this value, the range of $L$ that leads to the convergence of $\langle R_2(g,\rho)^r \rangle$ is not as flexible as in Theorem  \ref{variance_main_theorem_squarefree}. The limitation, as we will see below, comes from calculating the contribution to \eqref{sum-partition} of those partitions  for which $t \leq j$.

If $r = 3$, then $t =2, 3, 4$ or 5 and $j = 1,2$ or 3.  We have to consider partitions in $\mathcal P(3,t,j)$ for each value of $t$ and $j$.  Applying Theorem \ref{higher-moments-bounds}, 
\begin{enumerate}
\item the contribution of $\mathcal P(3,2,1)$ to \eqref{sum-partition} is
$$\ll \frac{L^4}{\pi_N(x)^4},$$
\item the contribution of $\mathcal P(3,3,1)$ is
$$\ll \frac{L^2}{\pi_N(x)^3},$$
\item the contribution of $\mathcal P(3,4,1)$ is
$$\ll \frac{L^4}{\pi_N(x)^6} + \frac{L}{\pi_N(x)},$$
\item the contribution of $\mathcal P(3,5,1)$ is
$$\frac{L}{\pi_N(x)},$$
\item the contribution of $\mathcal P(3,2,2)$ is
$$\frac{L^3(\log \log x)^2}{\pi_N(x)^2},$$
\item the contribution of $\mathcal P(3,3,2)$ is
$$\frac{L^2}{\pi_N(x)^2},$$
\item the contribution of $\mathcal P(3,4,2)$ is
$$\frac{L}{\pi_N(x)},$$
\item the contribution of $\mathcal P(3,5,2)$ is
$$\frac{1}{\pi_N(x)},$$
\item the contribution of $\mathcal P(3,2,3)$ is
$$\frac{L^3(\log \log x)^2}{\pi_N(x)},$$
\item the contribution of $\mathcal P(3,3,3)$ is
$$\frac{L^3(\log \log x)^3}{\pi_N(x)^2},$$
\item the contribution of $\mathcal P(3,4,3)$ is
$$\frac{L^2}{\pi_N(x)},$$
\item and the contribution of $\mathcal P(3,5,3)$ is
$$\frac{L}{\pi_N(x)}.$$
\end{enumerate}
The dominant contributions above come from those partitions for which $t \leq j$, namely $\mathcal P(3,2,2)$, $\mathcal P(3,2,3)$ and $\mathcal P(3,3,3)$. Putting all the estimates together and applying Theorem \ref{higher-moments-bounds}, we get
\begin{equation*}
\begin{split}
& \frac{1}{|\mathcal F_{N,k}|} \sum_{f \in \mathcal F_{N,k}}R_2(g,\rho)(f)^3 -  \left(\frac{T(g,\rho)}{(4L)}\right)^3\\
&\ll \sum_{t=2}^{6} \frac{L^{9 - 2t}}{\pi_N(x)^{6- t}} + \frac{L^3(\log \log x)^2}{\pi_N(x)} + \frac{L^3(\log \log x)^3}{\pi_N(x)^2} \\
&+  \frac{L^2}{\pi_N(x)^3} +   \frac{L^4}{\pi_N(x)^4} +  \frac{L^4}{\pi_N(x)^6} + \frac{L^{1/2}\log \log x}{\pi_N(x)^{1/2}}\\
& + \frac{x^{E(3)\pi_N(x)}8^{\nu(N)}}{kN}.\\
\end{split}
\end{equation*}
If we choose $L(x) \to \infty$ such that 
$$L(x) = \o\left(\frac{\pi_N(x)}{(\log \log x)^2}\right)^{1/3},$$
then
$$\sum_{t=2}^{6} \frac{L^{9 - 2t}}{\pi_N(x)^{6- t}} \to 0$$
and
$$ \frac{L^3(\log \log x)^2}{\pi_N(x)} + \frac{L^3(\log \log x)^3}{\pi_N(x)^2} +  \frac{L^2}{\pi_N(x)^3} +   \frac{L^4}{\pi_N(x)^4} +  \frac{L^4}{\pi_N(x)^6} + \frac{L^{1/2}\log \log x}{\pi_N(x)^{1/2}}\to 0.$$
If we consider families $\mathcal F_{N,k}$ with levels $N = N(x)$ and even weights $k = k(x)$ such that $$ \frac{\log \left(k{N}/8^{\nu(N)}\right)}{x} \to \infty \text{ as }x \to \infty,$$ 
then, by Lemma \ref{x-powers},
$$ \frac{x^{E(3)\pi_N(x)}8^{\nu(N)}}{kN} \to 0.$$
This proves Theorem  \ref{r=3 and more}.

\subsection{Conclusion}\label{conclusion}

As we saw above, Theorem \ref{higher-moments-bounds} gives us an ``all-purpose" bound for the higher power moments of $R_2(g,\rho)(f)$.  The application of this theorem, even for small values of $r$ can become tedious.  For example, if $r=4$, we have $1 \leq j \leq 4$ and $2 \leq t \leq 7$.  For each choice of $(t,j)$, we have to list out the partitions of $\mathcal P(4,t,j)$ and evaluate the contribution to \eqref{sum-partition} for each of them.  So, we have to consider 24 cases of $\mathcal P(4,t,j)$, and each case has numerous partitions. The parameters $x_{i,d}$ can take different values for different partitions in $\mathcal P(r,t,j)$.  As $r$ increases, the number of cases, as well as the number of partitions in each case grow rapidly.  Nonetheless, the explicit terms for the bound in Theorem \ref{higher-moments-bounds} give us a template to compute the $r$-th power moment for a fixed $r$.  Thus, for each $r$, we may try to establish a range of $L$ for which this moment would converge to $(T(g,\rho)/4L)^r$.  However, due to the unwieldy behaviour of $\mathcal P(r,t,j)$, we are currently unable to specify either a ``uniform" range of $L$, or a range of $L$ that depends on $r$ such that $r$-th power moments of $R_2(g,\rho)(f)$ would converge to $(T(g,\rho)/4L)^r$ for all $r$.  More precisely, we conclude with the following open questions:

\begin{enumerate}
\item Can we determine a function $L = L(x) \to \infty$ such that
$$\left \langle R_2(g,\rho)(f)^r \right \rangle\sim \left(\frac{T(g,\rho)}{4L}\right)^r\text{ as }x \to \infty?$$
for {\bf all} integers $r \geq 2$?
\item Let $r \geq 2$ be a fixed integer.  Can we determine a function $L = L_r = L_r(x) \to \infty$ (which depends on $r$) such that
$$\left \langle R_2(g,\rho)(f)^r  \right \rangle \sim \left(\frac{T(g,\rho)}{4L_r}\right)^r\text{ as }x \to \infty?$$
\end{enumerate}


\bibliographystyle{amsalpha}
\bibliography{Mahajan-Sinha-02-June}

\end{document}